\documentclass{article}
\usepackage{amsmath}
\usepackage{amssymb}
\usepackage{amsfonts}
\usepackage{latexsym}
\usepackage{mathrsfs}
\usepackage{amsthm}
\usepackage{amscd}
\usepackage{bussproofs}
\usepackage{txfonts}
\usepackage{rotating}
\usepackage{stmaryrd}
\usepackage{tikz}
\usetikzlibrary{matrix,arrows,decorations}
\usepackage{calc,url}
\usepackage[geometry]{ifsym}
\usepackage{pgfplots}

\newtheorem{theorem}{Theorem}[section]
 \newtheorem{lemma}[theorem]{Lemma}
 \newtheorem{prop}[theorem]{Proposition}
 
 \newtheorem{cor}[theorem]{Corollary}

 \theoremstyle{definition}
 \newtheorem{definition}[theorem]{Definition}
 \newtheorem{example}[theorem]{Example}

 \newtheorem{remark}[theorem]{Remark}

 \theoremstyle{remark}

\renewcommand{\phi}{\varphi}

\font\bmi=cmmi8 scaled 1440
\newcommand{\powerset}{\raise.6ex\hbox{\bmi\char'175 }}
\font\intix=cmex9

\newcommand{\TM}{{{\hbox{$\textfont3=\intix{\bigwedge}$}}}}
\newcommand{\TJ}{{{{\hbox{$\textfont3=\intix{\bigvee}$}}}}}
\newcommand{\nomi}{\mathbf{i}}
\newcommand{\nomj}{\mathbf{j}}
\newcommand{\nomk}{\mathbf{k}}
\newcommand{\noml}{\mathbf{l}}

\newcommand{\NOM}{\mathsf{NOM}}
\newcommand{\CNOM}{\mathsf{CNOM}}

\newcommand{\cnomm}{\mathbf{m}}
\newcommand{\cnomn}{\mathbf{n}}

\newcommand{\jty}{J^{\infty}}
\newcommand{\mty}{M^{\infty}}

\newcommand{\C}{\mathbf{C}}
\newcommand{\A}{\mathbf{A}}
\newcommand{\B}{\mathbf{B}}
\newcommand{\Lat}{\mathbf{L}}
\newcommand{\M}{\mathbf{M}}

\newcommand{\Db}{\Diamondblack} %diamond black
\newcommand{\Bb}{\blacksquare} %box black

\newcommand{\Clos}{\mathbb{K}}
\newcommand{\Open}{\mathbb{O}}

\newcommand{\ox}{\overline{x}}
\newcommand{\oy}{\overline{y}}
\newcommand{\oz}{\overline{z}}
\newcommand{\oX}{\overline{X}}

\newcommand{\op}{\overline{p}}

\newcommand{\ophi}{\overline{\phi}}
\newcommand{\opsi}{\overline{\psi}}
\newcommand{\oga}{\overline{\gamma}}

\newcommand{\bigamp}{\mathop{\mbox{\Large \&}}}

\newcommand{\amp}{\mathop{\&}}

\newcommand{\iamp}{\bindnasrepma}
\newcommand{\bigiamp}{\mbox{\Large $\iamp$}}

\newcommand{\phipqim}{\varphi(p, \overline{q}, \overline{\nomi}, \overline{\cnomm})}
\newcommand{\psipqim}{\psi(p, \overline{q}, \overline{\nomi}, \overline{\cnomm})}

\newcommand{\Cc}{\mathbf{C}}

\newcommand{\ifBox}[1]{IF$^{\Box}_{#1}$}
\newcommand{\ifDia}[1]{IF$^{\Diamond}_{#1}$}

\newcommand{\LFPone}{\mu}
\newcommand{\LFPtwo}{\mu_{2}}
\newcommand{\LFPstar}{\mu^{*}}

\newcommand{\GFPtwo}{\nu_{2}}

\newcommand{\Tm}{\mathcal{L}}
\newcommand{\Tmone}{\mathcal{L}_1}
\newcommand{\Tmtwo}{\mathcal{L}_2}
\newif\ifmargincoms

\margincomstrue

\begin{document}

%\title{Algorithmic-algebraic canonicity for mu-calculi}
\title{Canonicity results for mu-calculi: an algorithmic approach}
\author{Willem Conradie 	\\ \\  \texttt{wconradie@uj.ac.za} \\University of Johannesburg\\South Africa  \and Andrew Craig \\ \\ \texttt{acraig@uj.ac.za}
\\University of Johannesburg\\South Africa }
\maketitle

\begin{abstract}
We investigate the canonicity of inequalities of the intuitionistic mu-calculus. The notion of canonicity in the presence of fixed point operators is not entirely straightforward. In the algebraic setting of canonical extensions we examine both the usual notion of canonicity and what we will call tame canonicity. This latter concept has previously been investigated for the classical mu-calculus by Bezhanishvili and Hodkinson. Our approach is in the spirit of Sahlqvist theory. That is, we identify syntactically-defined classes of inequalities, namely the restricted inductive and tame inductive inequalities, which are, respectively, canonical or tame canonical. Our approach is to use an algorithm which processes inequalities with the aim of eliminating propositional variables. The algorithm we introduce is closely related to the algorithms ALBA and mu-ALBA studied by Conradie, Palmigiano, et al. It is based on a calculus of rewrite rules, the soundness of which rests upon the way in which algebras embed into their canonical extensions and the order-theoretic properties of the latter. We show that the algorithm succeeds on every restricted inductive inequality by means of a so-called proper run, and that this is sufficient to guarantee their canonicity. Likewise, we are able to show that the algorithm succeeds on every tame inductive inequality by means of a so-called tame run. In turn, this guarantees their tame canonicity.
\end{abstract}
%Our main results show that every member of each class of inequalities can (a) be successfully processed by the algorithm and hence (b) is, respectively, canonical or tame canonical.
%The join-density of the completely join-irreducibles and meet-density of the completely meet-irreducibles in the perfect lattice which is the canonical extension allow us to eliminate propositional variables in inequalities and systematically replace them with special variables called nominals and co-nominals.


\section{Introduction}\label{sec:Intro}

The modal mu-calculus was defined in 1983 by Kozen~\cite{Koz83} and is obtained by adding the least and greatest fixed point operators to the basic modal logic. An overview of the modal mu-calculus can be found in the chapter by Bradfield and Stirling~\cite{BradStir06}.
Canonical models and completeness results for the finitary and infinitary logics defined by Kozen were obtained by Ambler et al.~\cite{Ambler:TCS}.
In our algebraic canonicity proofs we will build on the definitions of modal mu-algebras from~\cite{Ambler:TCS}.
The correspondence and completeness of logics with fixed point operators has been the subject of recent studies by
Bezhanishvili and Hodkinson~\cite{Bezhanishvili:Hodkinson:TCS} and Conradie et al.~\cite{muALBA}. Both of these works aim to
develop a Sahlqvist-like theory for their respective fixed point settings.

Sahlqvist theory, first developed in 1975~\cite{Sahl75}, is one of the most important and powerful ideas in the study of modal and related logics.
The theory consists of two parts: \emph{canonicity} and \emph{correspondence}. The Sahlqvist formulas are a recursively defined class of modal formulas
with a particular syntactic shape. Any modal logic axiomatized by Sahlqvist formulas is strongly complete (via canonicity) with respect to its class of Kripke frames, and the latter is moreover guaranteed to be an elementary class. This last fact, that the class of frames can
be characterized by first-order conditions, is the correspondence aspect of Sahlqvist theory.
This Sahlqvist-style approach of describing a class of formulas of a certain syntactic shape for which correspondence and completeness results can be proved, has been extensively developed by van Benthem~\cite{VB83,VB2001} and others. The so-called Sahlqvist--van Benthem algorithm is used
to find the first-order condition that corresponds to a given Sahlqvist formula. Various generalizations of the Sahlqvist class exist, including the inductive formulas, introduced by Goranko and Vakarelov~\cite{GV-ind}.

The work in~\cite{Bezhanishvili:Hodkinson:TCS} looks at both correspondence and a certain type of canonicity for the classical mu-calculus. They define
\emph{Sahlqvist fixed point formulas}: a syntactic class which allows for limited use of fixed point operators.
A modified version of canonicity is proved for these Sahlqvist fixed point formulas. (We also note the related algebraic
work on preservation of Sahlqvist fixed point equations under MacNeille completions by Bezhanishvili and Hodkinson~\cite{BH-AU-relativized}.)

In contrast to the work on Sahlqvist fixed point formulas in \cite{Bezhanishvili:Hodkinson:TCS}, \cite{muALBA} examines correspondence only,
and, using an algorithmic approach, obtains results for a broader class of formulas in the setting of bi-intuitionistic mu-calculus. The correspondence
results are achieved by examining validity on the complex algebras dual to Kripke frames. The algorithmic approach of~\cite{muALBA} builds on
work by Conradie and Palmigiano~\cite{DistMLALBA} on canonicity and correspondence for distributive modal logic.

In this paper we will make use of purely algebraic and order-theoretic techniques. This approach has proved fruitful in obtaining completeness results for many non-classical logics. 

We prove two different canonicity results for two classes of intuitionistic mu-formulas.
We show that the members of a certain class of intuitionistic mu-formulas are
canonical, in the sense of~\cite{Bezhanishvili:Hodkinson:TCS};
that is, they are preserved under certain modified canonical extensions. We refer to this
modified form of canonicity (described in detail in Section~\ref{sec:canon-summary}) as \emph{tame canonicity}.
We also define a second class of formulas for which the usual notion of canonicity holds.
However, in this second case we do not get completeness of the logics defined by the
canonical mu-formulas. This lack of completeness is explained in more detail in Section~\ref{sec:canon-summary}.

Our methods use a variation of the algorithm ALBA (Ackermann Lemma Based Algorithm) developed in~\cite{DistMLALBA}.
The key step in our algorithm $\mu^*$-ALBA is the elimination of propositional variables via an Ackermann-style rule.
We define \emph{tame} and \emph{proper} runs of our algorithm and show that all mu-inequalities that can be successfully
processed by these runs are, respectively, tame canonical or canonical.

In Section~\ref{sec:language} we define the languages in which we operate and establish the algebraic setting for the
interpretation of these languages.
Section~\ref{sec:canon-summary} summarizes our canonicity results and outlines the method for achieving these results.
The syntactic classes of inequalities for which we obtain
our canonicity results are defined in Section~\ref{sec:synclass}; the examples in this section
will assist the reader in getting to grips with the rather technical syntactic definitions.
The algorithm $\mu^*$-ALBA is presented in Section~\ref{sec:algomustar}. The soundness of most of the rules of $\mu^*$-ALBA
follows easily from properties of the interpreting algebras.
However, the proofs of the fixed point approximation rules and the Ackermann rules require more intricate algebraic manipulations.
These proofs are given in Section~\ref{sec:sound-fp-approx} and Section~\ref{sec:soundAckermann} respectively.
Some technical lemmas required for the proofs in Section~\ref{sec:soundAckermann} are given in Appendix~\ref{app:alg-lemmas}.
Section~\ref{sec:innerformulas} describes a syntactically defined class of formulas, the term functions of which satisfy the order-theoretic conditions required by the fixed point approximation rules.
%\marginpar{\raggedright \tiny AC: Not sure how to fit these in to the story.}

Section~\ref{sec:canonicity} demonstrates the tame canonicity of mu-inequalities on which a tame run of our algorithm succeeds and
also shows the canonicity of mu-inequalities on which a proper run of $\mu^*$-ALBA succeeds.
In Section~\ref{sec:canon-of-synclasses} we prove that the members of the two different syntactic classes defined in Section~\ref{sec:synclass} are, respectively, tame canonical
and canonical.
To end the paper, in Section~\ref{sec:examples} we present two examples of the algorithm at work.

\section{Language and interpretation}\label{sec:language}

In this section we collect the essential details of the syntax and semantics we will be using. We have opted to work in an intuitionistic rather than classical setting for two reasons: firstly, it allows us to carefully disentangle the order theoretic properties of connectives which make our approach tick, in a way that would seem unnecessary and pedantic if, e.g., classical negation was available; secondly, this added generality comes at very little extra cost. Since the focus of this paper is canonicity, we will work almost exclusively with algebraic semantics, which in this case takes the form of bi-Heyting algebras with additional modal operators.  The relational semantics can be given, as usual, by intuitionistic Kripke frames with additional relations for interpreting the modalities, see e.g., \cite{WolterZakhIntML} and \cite{FischerServi}.

\paragraph{Modal bi-Heyting algebras.} A \emph{bi-Heyting algebra} is an algebra $\A=(A, \wedge, \vee, \rightarrow, -, \top, \bot)$ such that both
the reducts
$(A, \wedge, \vee, \rightarrow, \top, \bot)$ and $(A, \wedge, \vee, -, \top, \bot)^\partial$ are Heyting algebras. A \emph{modal bi-Heyting algebra} is an algebra $$\A=(A, \wedge, \vee, \rightarrow, -, \top, \bot, \Box, \Diamond)$$
such that $(A, \wedge, \vee, \rightarrow, -, \top, \bot)$ is a bi-Heyting algebra and $\Box$ and $\Diamond$ preserve finite meets and joins, respectively.
We observe that $\rightarrow$ and $-$ satisfy the inequalities
$$u \wedge v \le w \quad \text{iff} \quad u \le v \rightarrow w \qquad \text{and} \qquad u - v \le w \quad \text{iff} \quad u \le v \vee w.$$

The completely join-irreducible elements and completely meet-irreducible elements of a complete lattice will play a
very important role in our algorithmic approach.
\begin{definition}\label{def:CJI} Let $\mathbf{C}$ be a complete lattice. Then
\begin{enumerate}
\item[(i)] $j \in C$ is \emph{completely join-irreducible} if for any $X \subseteq C$, if $j=\bigvee X$ then $j=x$ for some $x \in X$;
\item[(ii)] $m \in C$ is \emph{completely meet-irreducible} if for any $Y \subseteq C$, if $m = \bigwedge Y$ then $m=y$ for some $y \in Y$;
\item[(iii)] $j \in C$ is \emph{completely join-prime} if for any $X \subseteq C$, if $j \le \bigvee X$ then $j \le x$ for some $x \in X$;
\item[(iv)] $m \in C$ is \emph{completely meet-prime} if for any $Y \subseteq C$, if $m \geq \bigwedge Y$, then $m \geq y$ for some
$y \in Y$.
\end{enumerate}
\end{definition}

A \emph{perfect lattice} is a complete lattice in which the completely join-irreducible elements are join-dense (i.e., every element is a join of join-irreducibles), and the completely meet-irreducible elements are meet-dense (i.e., every element is a meet of meet-irreducibles). A \emph{perfect distributive lattice} is a perfect lattice that is also completely distributive, i.e., arbitrary meets distribute over arbitrary joins and vice versa. In this case the completely join-irreducible (completely meet-irreducible) elements
coincide with the completely join-prime (completely meet-prime) elements. (In general, a completely join-prime (completely meet-prime) element of a complete lattice is completely join-irreducible (completely meet-irreducible) but not vice versa.)

A bi-Heyting algebra is \emph{perfect} if its lattice reduct is a perfect distributive lattice.  If follows that in a perfect bi-Heyting algebra $\A$, for any $S \subseteq A$, we have that $\bigvee S \rightarrow a = \bigwedge_{s \in S} (s \rightarrow a)$, $a \rightarrow \bigwedge S = \bigwedge_{s \in S} (a \rightarrow s)$, $\bigvee S - a = \bigvee_{s \in S}(s - a)$ and $a - \bigwedge S = \bigvee_{s \in S}(a - s)$. A \emph{perfect modal bi-Heyting algebra} is a modal bi-Heyting algebra the bi-Heyting reduct of which is a perfect bi-Heyting algebra, and moreover such that $\Box$ and $\Diamond$ preserve arbitrary meets and joins, respectively. The latter property allows us to add to any perfect modal bi-Heyting algebra the adjoint operations $\Db$ and $\Bb$ uniquely defined by the inequalities:
\[
\Diamond a \le b  \quad\Longleftrightarrow\quad a \le \Bb b \quad\qquad \text{and} \qquad\quad  a \le \Box b \quad\Longleftrightarrow\quad \Db a \le b.
\]

As usual we say that maps $f : \A \rightarrow \B$ and $g : \B \rightarrow \A$ form an \emph{adjoint pair} if, for all $a \in A$ and $b \in B$, it holds that $f(a) \leq b$ iff $a \leq g(b)$. Here $f$ is the \emph{left adjoint} and $g$ the \emph{right adjoint}. It is well known that a map between complete lattices is a left (right) adjoint iff it is completely join-preserving (completely meet-preserving).

A map $f: \A^n \rightarrow \A$ is the  \emph{left residual in the $i$-th coordinate} of a map $g_i: \A^n \rightarrow \A$ if, for all $a_1, \ldots, a_n, b \in A$, it holds that $f(a_1, \ldots, a_n) \leq b$ iff $g_i(a_1, \ldots, a_{i-1}, b, a_{i+1}, \ldots a_n) \leq a_i$. Here $g_i$ is the \emph{right residual of $f$ in the $i$-th coordinate}. It is easy to check that, if $\A$ is a complete lattice, then  $f$ has a right-residual (left-residual) in the $i$-th coordinate iff it is completely join-preserving (completely meet-preserving) in that coordinate. Recall that if $f: \A^n \rightarrow \A$ is completely join-preserving (completely meet-preserving) then it preserves all non-empty joins (meets) in each coordinate, but need not preserve empty joins (meets) in each coordinate. Thus $f: \A^n \rightarrow \A$ being an adjoint does not guarantee that it is a residual coordinatewise, nor vice versa.

\paragraph{The canonical extension.} Let $\A$ be a bounded distributive lattice with additional operations, in particular, $\A$ could be a modal bi-Heyting algebra.  The canonical extension of $\A$, defined by Gehrke and Harding~\cite{GH2001} and denoted $\A^\delta$, is a perfect bounded distributive lattice which, up to an isomorphism fixing $\A$,  is the unique extension in which $\A$ is dense and compact:

\begin{description}
\item[density]  every element of $\A^\delta$ can be
%ac1408
written
%expressed
both as a join of meets and as a meet of joins of elements from $\A$.
\item[compactness] for all $S, T \subseteq A$, if $\bigvee S \leq \bigwedge T$ then $\bigvee S'\leq \bigwedge T'$ for some finite $S'\subseteq S$ and $T'\subseteq T$.
\end{description}

Additional operations on $\A$ can be extended to $\A^\delta$ in a standard way. For more details the reader is referred to the appendix.

\paragraph{Languages and their interpretations.} Let $\mathsf{PROP}$, $\mathsf{FVAR}$, $\mathsf{PHVAR}$, $\mathsf{NOM}$ and $\mathsf{CNOM}$ be disjoint sets of propositional variables, fixed point variables, placeholder variables, nominals and co-nominals, respectively. Formulas in the basic language $\mathcal{L}$ of modal bi-Heyting algebras are defined recursively by
\begin{center}
$\phi ::= \bot \mid \top \mid p \mid X  \mid \phi \wedge \psi \mid \phi \vee \psi \mid \phi \rightarrow \psi \mid \phi - \psi
\mid \Diamond \phi \mid \Box \phi$
\end{center}
where
$p \in \mathsf{PROP}$ and $X \in \mathsf{FVAR}$. We identify the language with its set of formulas/terms.
Formulas in the extended language $\mathcal{L}^+$ are defined by
\begin{center}
$\phi ::= \bot \mid \top \mid p \mid X \mid \mathbf{j}  \mid \mathbf{m} \mid \phi \wedge \psi \mid \phi \vee \psi \mid \phi \rightarrow \psi \mid \phi - \psi
\mid \Diamond \phi \mid \Box \phi \mid \Bb \phi  \mid \Db \phi$
\end{center}
where
$p \in \mathsf{PROP}$, $X \in \mathsf{FVAR}$, $\nomj \in \mathsf{NOM}$ and $\cnomm \in \CNOM$. Placeholder variables from $\mathsf{PHVAR}$, denoted $x, y, z$, will be used as generic variables which can take on the roles of propositional and fixed point variables. They will also be used to enhance the clarity of the exposition when dealing with substitution instances of formulas.

On perfect modal bi-Heyting algebras $\Bb$ and $\Db$ are interpreted as the right and left adjoints of $\Diamond$ and $\Box$, respectively. Elements of $\mathsf{NOM}$ ($\CNOM$) are interpreted
as elements of $\jty(\A^\delta)$ ($\mty(\A^\delta)$).
%ac1403
%We will use the terms `formula' and `term' interchangeably, and also denote the set of all terms in $\mathcal{L}$ by $\Tm$ and the terms of $\mathcal{L}^+$ by $\Tm^+$.
We now describe two extensions of $\Tm$ obtained by adding fixed point operators.
The distinction between the two extensions will become clear when we define their interpretations
on distributive lattices with operators.

We define $\Tmone$ to be the set of terms which extends $\Tm$ by allowing terms
$\mu x. t(x)$ and $\nu x.t(x)$ where $t \in \Tmone$, $x \in \mathsf{FVAR}$ and $t(x)$ is positive in $x$.
The second extension is denoted $\Tmtwo$
%ac2310 again should we write $(\Tm^+)_2$?
and extends $\Tm$ by allowing construction of the terms $\mu_2 x. t(x)$ and $\nu_2 x.t(x)$
where $t \in \Tmtwo$, $x \in \mathsf{FVAR}$  and $t(x)$ is positive in $x$.

The terms of $\Tm$ are interpreted as usual on modal bi-Heyting algebras. The additional terms of
$\Tmone$ and $\Tmtwo$ are interpreted as follows: Suppose $t(x_1,x_2, \ldots,x_n) \in \Tmone$ and $a_1,\ldots,a_{n-1} \in A$. Then
$$\mu x. t(x,a_1,\ldots,a_{n-1}) := \bigwedge \{\,a \in A \mid t(a,a_1,\ldots,a_{n-1}) \le a \,\}$$
if this meet exists, otherwise $\mu x. t(x,a_1,\ldots,a_{n-1})$ is undefined. Similarly,
$$\nu x.t(x,a_1,\ldots,a_{n-1}) := \bigvee \{\, a \in A \mid a \le t(a,a_1,\ldots,a_{n-1}) \,\}$$
if this join exists, otherwise $\nu x.t(x,a_1,\ldots,a_{n-1})$ is undefined. For each ordinal $\alpha$ we define $t^\alpha(\bot,a_2,\ldots,a_n)$ as follows:
\begin{align*}
t^0(\bot,a_1,\ldots,a_{n-1}) &= \bot, \qquad t^{\alpha+1}(\bot,a_1,\ldots,a_{n-1}) = t\big( t^{\alpha}(\bot,a_1,\ldots,a_{n-1}),a_1,\ldots,a_{n-1}\big),\\
t^{\lambda}(\bot,a_1,\ldots,a_{n-1}) &= \bigvee_{\alpha < \lambda} t^{\alpha}(\bot, a_1, \ldots,a_{n-1}) \quad \text{for limit ordinals } \lambda;\\
t_0(\top,a_1,\ldots,a_{n-1}) &= \top, \qquad t_{\alpha+1}(\top,a_1,\ldots,a_{n-1}) = t\big( t_{\alpha}(\top,a_1,\ldots,a_{n-1}),a_1,\ldots,a_{n-1}\big),\\
t_{\lambda}(\top,a_1,\ldots,a_{n-1})&= \bigwedge_{\alpha < \lambda} t_{\alpha}(\top, a_1, \ldots,a_{n-1}) \quad \text{for limit ordinals } \lambda.
\end{align*}

For $t(x_1,\ldots,x_n) \in \Tmtwo$ we then define
$$\LFPtwo x.t(x,a_1,\ldots,a_{n-1}) := \bigvee_{\alpha \geq 0} t^\alpha(\bot,a_1,\ldots,a_{n-1})
\quad \text{and}\quad
\GFPtwo x.t(x,a_1,\ldots,a_{n-1}):= \bigwedge_{\alpha \geq 0} t_\alpha(\top,a_1,\ldots,a_{n-1})$$
if this join and this meet exist, and they are undefined otherwise.

A modal bi-Heyting algebra $\A$ is said to be \emph{of the first kind} (\emph{of the second kind}) if
$t^\A(a_1,\ldots,a_n)$ is defined for all $a_1,\ldots,a_n \in \A$ and all $t \in \Tmone$ ($t \in \Tmtwo$).
Henceforth we will refer to these algebras as \emph{mu-algebras of the first kind} (\emph{of the second kind}).
When restricted to the Boolean case, our mu-algebras of the first kind are essentially the modal mu-algebras
defined in~\cite[Definition 2.2]{Bezhanishvili:Hodkinson:TCS} and~\cite[Definition 5.1]{Ambler:TCS}.

\begin{lemma}\label{lem:type2=>type1} {\upshape \cite[Proposition 2.4]{Ambler:TCS}} If\, $\A$ is a mu-algebra of the second kind, then $\A$ is a mu-algebra of the first kind.
\end{lemma}
\begin{proof}
Suppose $t(x, x_1, \ldots, x_{n-1}) \in \Tmtwo$ is positive in $x$ and let $a_1, \ldots, a_{n-1} \in A$.  Further, let $\gamma$ be
the first ordinal such that $t^\beta(\bot, a_1, \ldots, a_{n-1}) = t^\gamma(\bot, a_1, \ldots, a_{n-1})$ for all $\beta > \gamma$ --- such a $\gamma$ exists since $t$ is monotone in $x$. We will show that the meet
$\LFPone x.t(x, a_1, \ldots, a_{n-1})=\bigwedge \{\, a \in A \mid t(a,a_1,\ldots,a_{n-1}) \le a \,\}$ exists by showing that
$\LFPtwo x.t(x, a_1, \ldots, a_{n-1})=\LFPone x.t(x, a_1, \ldots, a_{n-1})$.

Suppose that $a \in A$ is a pre-fixed point of $t$, that is, $t(a,a_1,\ldots,a_{n-1}) \le a$. We will prove by transfinite induction that $t^\alpha(\bot, a_1, \ldots, a_{n-1}) \le a$ for all ordinals $\alpha$.

\begin{itemize}
\item Base case: clearly $t^0(\bot, a_1, \ldots, a_{n-1})=\bot \le a$.
\item Suppose $t^\alpha(\bot, a_1, \ldots, a_{n-1}) \le a$. Then\\
$t^{\alpha+1}(\bot, a_1, \ldots, a_{n-1}) = t (t^\alpha(\bot, a_1, \ldots, a_{n-1}), a_1, \ldots, a_{n-1}) \le t(a,a_1,\ldots,a_{n-1}) \le a$ since $t$ is positive in $x$.
\item Let $\lambda$ be a limit ordinal with $t^\alpha(\bot,a_1, \ldots, a_{n-1}) \le a$ for all $\alpha < \lambda$. Then\\
$t^\lambda(\bot, a_1, \ldots, a_{n-1}) = \bigvee \{\, t^\alpha(\bot, a_1, \ldots, a_{n-1}) \mid \alpha < \lambda \,\} \le a$ since $a$ is an upper bound for the set\\
$\{\,t^\alpha(\bot, a_1, \ldots, a_{n-1}) \mid \alpha < \lambda\,\}$.
\end{itemize}
Thus we have that $t^\gamma(\bot, a_1, \ldots, a_{n-1}) \le \bigwedge \{\,a \in A \mid t(a,a_1,\ldots,a_{n-1}) \le a \,\}$. Note also that $t^\gamma(\bot, a_1, \ldots, a_{n-1})$ is a pre-fixed point as
$t(t^\gamma(\bot, a_1, \ldots, a_{n-1}), a_1, \ldots, a_{n-1}) = t^{\gamma+1}(\bot, a_1, \ldots, a_{n-1})=t^\gamma(\bot, a_1, \ldots, a_{n-1})$. Thus $t^\gamma(\bot, a_1, \ldots, a_{n-1}) \in \{\,a \in A \mid t(a,a_1,\ldots,a_{n-1}) \le a \,\}$ and so we have the desired equality.
\end{proof}

The importance of Lemma~\ref{lem:type2=>type1} is that if we are interpreting formulas/terms on a mu-algebra of the second kind the interpretations of terms with the two different fixed point binders will agree. That is, $\mu X.\phi(X) =\mu_2 X.\phi(X)$ and $\nu X.\psi(X) = \nu_2 X.\psi(X)$.

The final sets of terms,  $\Tm_*$ (respectively,  $\Tm^+_*$), are obtained as an extension of $\mathcal{L}$ (respectively, $\mathcal{L}^+$) by allowing $\mu^* x. t(x)$ and $\nu^* x.t(x)$ whenever $t \in \Tm_*$ (respectively, $t \in \Tm^+_*$) and positive in $x$. Terms in $\Tm_*$ and $\Tm^+_*$ are only interpreted in the canonical extensions $\A^\delta$ of modal bi-Heyting algebras $\A$. If $t(x_1,x_2, \ldots,x_n) \in \Tm_* \cup \Tm^+_*$ and $a_1,\ldots,a_{n-1} \in A^\delta$, then $\mu^* x_1. t(x_1,a_1,\ldots,a_{n-1}) := \bigwedge \{\,a \in A \mid t(a,a_1,\ldots,a_{n-1}) \le a \,\}$ and $\nu^* x_1.t(x_1,a_2,\ldots,a_n) := \bigvee \{\, a \in A \mid a \le t(a,a_1,\ldots,a_{n-1}) \,\}$.
As the canonical extension $\A^\delta$ is a complete lattice, the interpretation of $\mu^* x. t(x)$ or $\nu^* x.t(x)$ is always defined.
Given a term $\phi \in \Tm^+_1$ we write $\phi^*$ for the $\Tm^+_*$ term obtained from $\phi$ by replacing all occurrences of $\mu$ and $\nu$ with $\mu^*$ and $\nu^*$, respectively.

An $\mathcal{L}^+$ formula is \emph{pure} if it contains no ordinary (propositional) variables but only, possibly, nominals and co-nominals. A formula of $\mathcal{L}_1$
(respectively $\Tm_1^+, \Tm_2,\Tm_2^+,\Tm_*,\Tm_*^+$) is an \emph{$\mathcal{L}_1$-sentence} if it contains no free fixed point variables (and similarly for
$\Tm_1^+, \Tm_2,\Tm_2^+,\Tm_*,\Tm_*^+$).

The reason for using $\mu^*$ and $\nu^*$ in formulas/terms is so that the interpretation does not change when moving
between $\A$ and $\A^\delta$.

\paragraph{Quasi-inequalities, assignments, validity.} An assignment on $\A$ sends propositional variables to elements of $A$  and is extended to formulas of $\mathcal{L}_1$, $\mathcal{L}_2$ and $\mathcal{L}_*$ in the usual way, where these are defined. An assignment on $\A^\delta$ sends propositional variables to elements of $\A^\delta$, nominals into $\jty(\A^\delta)$ and co-nominals into $\mty(\A^\delta)$ and extends to all formulas of $\mathcal{L}^+_*$, $\mathcal{L}^+_1$ and $\mathcal{L}^+_2$. An \emph{admissible assignment} on $\A^\delta$ is an assignment which takes all propositional variables to elements of $\A$. An $\mathcal{L}^+$-inequality $\alpha \leq \beta$ is \emph{admissibly valid} on $\A^\delta$, denoted $\A^\delta \models_\A \alpha \leq \beta$, if it holds under all admissible assignments.

A \emph{quasi-inequality} of $\mathcal{L}_1$ (resp., $\mathcal{L}^+_1$, $\mathcal{L}_2$, $\mathcal{L}^+_2$, $\mathcal{L}_*$, $\mathcal{L}^+_*$) is an expression of the form $\phi_1 \leq \psi_1 \amp \cdots \amp \phi_n \leq \psi_n \Rightarrow \phi \leq \psi$ where the $\phi_i$, $\psi_i$, $\phi$ and $\psi$ are formulas of  $\mathcal{L}_1$ (resp., $\mathcal{L}^+_1$, $\mathcal{L}_2$, $\mathcal{L}^+_2$, $\mathcal{L}_*$, $\mathcal{L}^+_*$). A quasi-inequality $\phi_1 \leq \psi_1 \amp \cdots \amp \phi_1 \leq \psi_1 \Rightarrow \phi \leq \psi$ is satisfied under an assignment $V$ in an algebra $\A$ of the appropriate sort, written $\A, V \models \phi_1 \leq \psi_1 \amp \cdots \amp \phi_n \leq \psi_n \Rightarrow \phi \leq \psi$ if $\A, V \not \models \phi_i \leq \psi_i$ for some $1 \leq i \leq n$ or $\A, v \models \phi \leq \psi$. A quasi-inequality is (admissibly) valid in an algebra if it is satisfied by every (admissible) assignment.

\paragraph{Signed generation trees}
%\texttt{Include all the languages.}
To any formula/term in $\mathcal{L}_1^+$ and $\mathcal{L}_*^+$ we assign two \emph{signed generation trees}. That is,
for $\varphi \in \mathcal{L}$ we consider two trees $+\phi$ and $-\phi$. The generation tree is
constructed as usual, beginning at the root with the main connective and then
branching out into $n$-nodes at each $n$-ary connective. Each leaf is either a propositional
variable, a fixed point variable, or a constant. Each node is  signed as follows:
\begin{itemize}
\item the root node of $+\phi$ is signed $+$ and the root node of $-\phi$ is signed $-$;
\item if a node is $\vee, \wedge, \Diamond, \Box$, $\Diamondblack$, or $\blacksquare$ assign the same sign to its successor nodes;
\item if a node is $\rightarrow$, assign the opposite sign to its left successor, and the same sign to its right successor;
\item if a node is $-$, assign the same sign to its left successor, and the opposite sign to its right successor;
\item if a node is $\mu x.\varphi(x)$, $\mu^* x.\varphi(x)$, $\nu x.\varphi(x)$ or $\nu^* x.\varphi(x)$ (with every free occurrence of $x$ in the positive generation tree of $\varphi$ labelled positively)
then assign the same sign to the successor node.
\end{itemize}
A node in a signed generation tree is said to be
\emph{positive} if it is signed ``$+$'' and \emph{negative} if it is signed ``$-$''. Examples of signed generation trees can be
found in Figure~\ref{Fig:GoodBranches} and Figure~\ref{Fig:GoodBranches:Inductive}.

\paragraph{Order types}
We will often be using formulas in $n$ variables and hence use $\overline{x}$ to denote $n$-tuple of variables.
An \emph{order-type} over $n \in \mathbb{N}$ is an $n$-tuple $\epsilon \in \{1,\partial\}^n$.	Given an order-type $\epsilon$,
its opposite order type, denoted $\epsilon^\partial$, is given by $\epsilon^\partial_i=1 \Leftrightarrow \epsilon_i =\partial$ for
$1 \le i \le n$. We will also use the symbol $\tau$ to denote an order type over $n$.
When we define the Approximation Rules in Section~\ref{sec:algomustar} it will be useful
to write $\top^1$ and $\top^\partial$ for $\top$ and $\bot$ respectively. Similarly we will write $\bot^1$ and
$\bot^\partial$ for $\bot$ and $\top$ respectively.
%Given an order type $\epsilon$, the notation $\bot^{\epsilon_i}$ will denote $\bot$ if $\epsilon_i=1$ and $\top$ if $\epsilon_i=\partial$.

For both order types and tuples of variables, we will use the symbol $\oplus$ to denote concatenation.

\paragraph{Join- and meet-irreducible elements in products} Let $\A$ be a perfect lattice and $\tau$ and order type over $n$. It is not difficult to
prove that in the product $\A^n$ every join-irreducible element $\overline{j}$  is an $n$-tuple such that
$(\overline{j})_k \in \jty(\A)$ for some $1 \le k \le n$ and $(\overline{j})_i=\bot$ for all $i \neq k$. Dually,
every  completely meet-irreducible element in the product $\A^n$ is an $n$-tuple
$\overline{m}$ such that $(\overline{m})_k \in \mty(\A)$ for some $1 \le k\le n$ and
$(\overline{m})_i=\top$ for all $i \neq k$.

This characterization of $\jty(\A^n)$ and $\mty(\A^n)$ can be generalized to the case
of $\jty(\A^\tau)$ and $\mty(\A^\tau)$: every join-irreducible element $\overline{j} \in \jty(\A^\tau)$  is an $n$-tuple such that, for some $1 \le k \le n$,
$(\overline{j})_k \in \jty(\A)$ if $\tau_k = 1$ and $(\overline{j})_k \in \mty(\A)$ if $\tau_k = \partial$ while, for $i \neq k$, $(\overline{j})_i=\bot$ if $\tau_i = 1$ and $(\overline{j})_i=\top$ if $\tau_i = \partial$. Order-dually, every meet-irreducible element $\overline{m} \in \mty(\A^\tau)$  is an $n$-tuple such that, for some $1 \le k \le n$,
$(\overline{m})_k \in \mty(\A)$ if $\tau_k = 1$ and $(\overline{j})_k \in \jty(\A)$ if $\tau_k = \partial$ while, for $i \neq k$, $(\overline{m})_i=\top$ if $\tau_i = 1$ and $(\overline{m})_i=\bot$ if $\tau_i = \partial$.

Let $\A$ be a perfect lattice and let $\tau$ be an order type over $n$. The notation $\nomj^{\tau_i}$ will denote
a nominal if $\tau_i=1$ and a co-nominal if $\tau_i=\partial$. Dually, $\cnomn^{\tau_i}$ denotes a co-nominal
if $\tau_i=1$ and a nominal if $\tau_i=\partial$.
%The notation $\top^1$ and $\bot^{\partial}$ to denote $\top$, and the notation $\bot^1$ and $\top^{\partial}$ to denote $\bot$.

Accordingly, we will use the notation $\overline{\nomj}^{\tau}_i$ to denote an $n$-tuple in which the $i$-th component is $\nomj^{\tau_i}$ and, for all
$k \neq i$, the $k$-th component is $\bot^{\tau_k}$. Accordingly, such tuples range over the subset of $\jty(\A^{\tau})$ in which the $i$-th component comes from $\jty(\A^{\tau_i})$ and all other components are $\bot^{\tau_k}$, $k \neq i$. Dually, $\overline{\cnomn}^{\tau}_i$ denotes an $n$-tuple in which the $i$-th component is $\cnomn^{\tau_i}$ and, for all $k \neq i$, the $k$-th component is $\top^{\tau_k}$.

%\texttt{ELIMINATE $\backslash$ MATHBB NOTATION FOR ALGEBRAS, MAKE IT  $\backslash$ MATHBF}

\section{Canonicity results} \label{sec:canon-summary}

In this section we highlight and explain the main results of this paper and describe our methodology in broad strokes. Our approach goes via a ``U-shaped argument''
(see Figure~\ref{fig:generalU}), a generic version of which we now outline.

\begin{figure}
\begin{center}
\begin{tikzpicture}
\path (0,0) node(a) {$\A \models \alpha \le \beta$}
			(0,-0.5) node(b) [rotate=90] {$\Leftrightarrow$}
			(0,-1) node(c) {$\A^{\delta} \models_\A \alpha \le \beta$}
			(0,-1.5) node(i) [rotate=90] {$\Leftrightarrow$}
			(0,-2) node(d) {$\A^{\delta} \models_\A \texttt{pure}(\alpha \le \beta)$}
			(5,-2) node(f) {$\A^\delta \models \texttt{pure}(\alpha \le \beta)$}
			(5,-1) node(g) [rotate=90] {$\Longleftrightarrow$}
			(2.5,-2) node(j) {$\Longleftrightarrow$}
			(5,0) node(k) {$\A^\delta \models \alpha \le \beta$};
\end{tikzpicture}
\caption{The U-shaped argument for canonicity of inequalities interpreted on a lattice-based algebra $\A$.}
\label{fig:generalU}
\end{center}
\end{figure}

Going down the left-hand arm of the diagram, the first bi-implication is given by the fact that validity in $\A$ coincides with admissible validity in $\A^\delta$, modulo certain provisions regarding the fixed point binders. In the richer setting of $\A^\delta$ we can now interpret the extended language $\mathcal{L}^+_*$ where equivalences involving the adjoints $\Diamondblack$ and $\blacksquare$ as well as nominals and co-nominals are available. The aim is now to transform the inequality into a set of pure (quasi-)inequalities, denoted $\texttt{pure}(\alpha \le \beta)$ in Figure~\ref{fig:generalU}. This is done by means of an algorithm, $\mu^*$-ALBA, based on a calculus of rewrite rules which are presented in
%later sections, beginning with
Section~\ref{sec:algomustar}. The fact that admissible and ordinary validity coincide for pure inequalities is the linchpin for the transition from validity in $\A$ (simulated as admissible validity in $\A^{\delta}$) to validity in $\A^{\delta}$. This justifies the bi-implication forming the base of the ``U''.

We progress up the right-hand arm of the `U' by reversing the rewrite rules applied when coming down the left-hand side. The equivalences are justified by the fact that these rules preserve validity on perfect algebras. (The last statement needs to be qualified somewhat to accommodate fixed point binders, as will be specified in Section \ref{sec:canonicity}.) We observe that the equivalences on the right-hand arm of Figure~\ref{fig:generalU} justifies the first-order frame definability of $\alpha \leq \beta$, which is the topic of \cite{muALBA}. Indeed, the fact that $\alpha \leq \beta$ is equivalent on perfect algebras (or, dually, relational structures) to a set of pure quasi-inequalities guarantees first-order definability since the absence of propositional variables in  $\texttt{pure}(\alpha \le \beta)$ means that all quantification ranges over first-order definable subsets of the dual relational structure of $\A^{\delta}$.

As we now explain, the notion of canonicity of a formula in the presence of fixed points admits some variation.
%is not entirely straightforward.
If $\phi$ is a formula without fixed point binders, then the term function $\phi^{\A^{\delta}}$ \emph{extends} the term function $\phi^{\A}$, i.e., they agree on arguments from $A$. This is something which is usually of crucial importance in proving that an equation is canonical. As soon as we add fixed point binders this extension property fails. Indeed, $(\phi(X))^{\A^{\delta}}$ can have more pre-fixed points in $\A^{\delta}$ than $(\phi(X))^{\A}$ has in $\A$, and so $(\mu X. \phi(X))^{\A^{\delta}}$ would generally be smaller than $(\mu X. \phi(X))^{\A}$. This phenomenon creates significant obstacles for standard canonicity arguments.

One possible remedy to regain the extension property, is to insist that only pre-fixed points from the smaller algebra $\A$ be considered in calculating $(\mu X. \phi(X))^{\A^{\delta}}$, i.e., rather calculating $(\mu^* X. \phi(X))^{\A^{\delta}}$. This leads to the first notion of canonicity that we will examine. This is the notion described
by Bezhanishvili and Hodkinson~\cite{Bezhanishvili:Hodkinson:TCS}, and used by them to obtain completeness results for some axiomatic extensions of the basic (classical)
mu-calculus with respect to certain types of general frames.
We will refer to this as \emph{tame canonicity}.
In tame canonicity, all fixed point binders $\mu$ and $\nu$ in inequalities are replaced by the
binders $\mu^*$ and $\nu^*$ and interpreted as such in the canonical extension. Proving
that an $\Tm_1$-inequality is \emph{tame canonical}
is proving that $\A \models \phi \le \psi$ if and only if $\A^\delta \models \varphi^* \le \psi^*$ where
$\A$ is mu-algebra of the first kind. (Further explanation regarding the use of $\mu^*$ and $\nu^*$ is given at the end of
Section~\ref{sec:sound-fp-approx}.)

Using our algorithmic approach we are able to prove tame canonicity for what we
call \emph{tame inductive} mu-inequalities. This is a smaller class than that for
which correspondence results were shown in~\cite{muALBA} but extends the Sahlqvist mu-inequalities of~\cite{Bezhanishvili:Hodkinson:TCS} when projected to the classical case.

The second type of canonicity which we investigate is essentially the usual notion of canonicity  and thus we simply refer to this as
\emph{canonicity}. That is, for $\A$ a mu-algebra of the second kind, if $\A \models \phi \le \psi$, then $\A^\delta \models \phi \le \psi$.
(Note the additional assumption that $\A$ is of the second kind.) Our method of proof is as follows: we first show that $\A \models \phi \le \psi$ if and only if
$\A^\delta \models_\A \phi^* \le \psi^*$ and then that this implies
$\A^\delta \models \varphi \le \psi$. However, we are not able to show the converse, i.e., that $\A^\delta \models \phi \le \psi$ implies $\A \models \phi \le \psi$.

We will call the syntactically specified class of formulas for which we can prove canonicity the
\emph{restricted inductive mu-inequalities}.
%This class restricts the occurrence of
%fixed point binders so that all propositional variables under the scope of a fixed point binder must
%be eliminated.
This class is again a restriction of the class for which correspondence results were shown in~\cite{muALBA}, but it is a generalization of the inductive
inequalities from~\cite{DistMLALBA}.
%The generalization from~\cite{DistMLALBA} is a result of the fact that certain occurrences of $\mu$ and $\nu$ are permitted.

From a logical perspective, one of the main motivations for proving canonicity results is to obtain relational completeness results for axiomatic extension of logics --- if we have a general algebraic completeness result any non-theorem is refuted on an algebra of the logic and, if the axioms are preserved under canonical extensions, we can transfer this refutation to a relational structure obtained as the dual of the perfect algebra which is the canonical extension.  Unfortunately this does not work in the setting of the mu-calculus, since generally it will not be true that  $\A^\delta \models \phi \le \psi$ implies that  $\A \models \phi \le \psi$ or, contrapositively, that $\A \not \models \phi \le \psi$ implies $\A^\delta \not \models \phi \le \psi$. Hence refutations
will not necessarily be preserved when taking canonical extensions.

The proof of the tame canonicity of the tame inductive mu-inequalities and the  proof of the
canonicity of the restricted inductive mu-inequalities is concluded in Section~\ref{sec:canon-of-synclasses}.
The proof is a two-step process, beginning in Section~\ref{sec:canonicity}. There it is shown that
whenever a \emph{tame} run of $\mu^*$-ALBA succeeds on a mu-inequality $\phi \le \psi$, we have that $\phi \le \psi$ is \emph{tame canonical}.
It is also shown that whenever a \emph{proper} run succeeds on a mu-inequality $\alpha \le \beta$, then $\alpha \le \beta$ will be \emph{canonical}.
Section~\ref{sec:canon-of-synclasses} then completes the overall result by showing that for every tame inductive mu-inequality (respectively, a restricted inductive mu-inequality), there exists a tame (respectively, proper) run of $\mu^*$-ALBA which succeeds on that inequality.

%TO BE MOVED TO NEXT SECTION
%The reason we concern ourselves with the canonicity of \emph{inequalities} is because certain fragments
%(for example, distributive modal logic~\cite{GNV05}) do not have an implication in the language.
%In the absence of an implication, modal sequents of the form $\alpha \Rightarrow \beta$ are used to capture the logic.
%When using the algebraic semantics, the interpretation of inequality $\alpha \le \beta$ is used
%as the interpretation of the sequent $\alpha \Rightarrow \beta$. In addition, when working algebraically
%it is often easier to prove the order-theoretic relationship between the interpretation of two formulas
%than to show that a certain formula is always evaluated to $1$ in the algebra.

\section{Syntactic classes}
\label{sec:synclass}

In this section we introduce two new syntactically defined classes of mu-inequalities. Both will be subclasses of the recursive mu-inequalities introduced in \cite{muALBA}, and therefore the members of both classes will all have first-order frame correspondents.

The reason we concern ourselves with \emph{inequalities} is because certain fragments
(for example, distributive modal logic~\cite{GNV05}) do not have an implication in the language.
In the absence of an implication, modal sequents of the form $\phi \Rightarrow \psi$ are used to capture the logic.
When using the algebraic semantics, the interpretation of inequality $\phi \le \psi$ is used
as the interpretation of the sequent $\phi \Rightarrow \psi$.

For any $\Tmone$-sentence $\varphi(p_1,\ldots p_n)$, any order-type $\epsilon$ over $n$, and any $1\leq i\leq n$, an  \emph{$\epsilon$-critical node} in a signed generation tree of $\varphi$ is a (leaf) node $+p_i$ with $\epsilon_i = 1$, or $-p_i$ with $\epsilon_i = \partial$. An $\epsilon$-{\em critical branch} in the tree is a branch terminating in an $\epsilon$-critical node. The intuition, which will be built upon later, is that variable occurrences corresponding to  $\epsilon$-critical nodes are \emph{to be solved for, according to $\epsilon$}.

In the signed generation tree of a $\Tmone$-sentence $\varphi(p_1,\ldots p_n)$ a \emph{live branch} is a branch ending in a (signed) propositional variable. In particular, all critical branches are live. It follows that a branch is not live iff it ends in a propositional constant ($\top$ or $\bot$) or in a fixed point variable.

For every $\Tmone$-sentence $\varphi(p_1,\ldots p_n)$, and every order-type $\epsilon$, we say that $+\varphi$ (resp. $-\varphi$) {\em agrees with} $\epsilon$, and write $\epsilon(+\varphi)$ (resp.\ $\epsilon(-\varphi)$), if every leaf node in the signed generation tree of $+\varphi$ (resp.\ $-\varphi$) which is labelled with a propositional variable  is $\epsilon$-critical.
In other words, $\epsilon(+\varphi)$ (resp.\ $\epsilon(-\varphi)$) means that all propositional variable occurrences corresponding to leaves of $+\varphi$ (resp.\ $-\varphi$) are to be solved for according to $\epsilon$. We will also make use of the {\em sub-tree relation} $\gamma\prec \varphi$, which extends to signed generation trees, and we will write $\epsilon(\gamma)\prec \ast \varphi$ to indicate that $\gamma$, regarded as a sub- (signed generation) tree of  $\ast \varphi$, agrees with $\epsilon$.

\begin{table}
\begin{center}
\begin{tabular}{| c | c | c | c |}
\hline
Outer Skeleton ($P_3$) &Inner Skeleton ($P_2$) &PIA ($P_1$)\\
\hline
$\Delta$-adjoints &Binders &Binders\\
\begin{tabular}{ c c c c c c c  }
$+$ &$\vee$ &$\wedge$ &${}$ &${}$  &${}$ &${}$\\
$-$ &$\wedge$ &$\vee$\\
\hline
\end{tabular}
&\begin{tabular}{c c c c c c c c }
& & &$+$ &$\mu$ & & &\\
& & &$-$ &$\nu$ & & &\\
\hline
\end{tabular}
&\begin{tabular}{c c c c c c c c }
& & &$+$ &$\nu$ & & &\\
& & &$-$ &$\mu$ & & &\\
\hline
\end{tabular}
\\
SLR &SLA &SRA\\
\begin{tabular}{c c c c c}
$+$ &$\Diamond$ &$\lhd$ &$\circ$ &$-$\\
$-$ &$\Box$ &$\rhd$ &$\star$ &$\rightarrow$\\
\end{tabular}
&\begin{tabular}{c c c c c c c c}
&$+$ &$\Diamond$ &$\lhd$ &$\vee$ &\\
&$-$ &$\Box$ &$\rhd$ &$\wedge$ &\\
\hline
\end{tabular}
&\begin{tabular}{c c c c c c c c}
&$+$ &$\Box$ &$\rhd$ &$\wedge$ &\\
&$-$ &$\Diamond$ &$\lhd$ &$\vee$ &\\
\hline
\end{tabular}
\\
&SLR &SRR\\
&\begin{tabular}{c c c c c c}
$+$ &$\wedge$ &$\circ$ &$-$\\
$-$ &$\vee$ &$\star$ &$\rightarrow$\\
\end{tabular}
&\begin{tabular}{c c c c c c}
$+$ &$\vee$ &$\star$ &$\rightarrow$\\
$-$ &$\wedge$ &$\circ$ &$-$\\
\end{tabular}
\\
\hline
\end{tabular}
\end{center}
\caption{Skeleton and PIA nodes.}\label{Skeleton:PIA:Node:Table}
\end{table}

While reading the following definition, the reader might find it useful to refer to Example \ref{Examp:Good:Branches} for an illustration of the concepts being introduced.

\begin{definition}
Nodes in signed generation trees will be called \emph{skeleton nodes} and \emph{PIA nodes} and further classified as $\Delta$-adjoint, SLR, Binders, SLA, SRA or SRR,  according to the specification given in table \ref{Skeleton:PIA:Node:Table}.\footnote{The interpretations of the binary connectives {\em fusion} $\circ$ and {\em fission} $\star$ included in table \ref{Skeleton:PIA:Node:Table} preserve, respectively, joins and meets coordinatewise. The unary connectives $\lhd$ and $\rhd$ are, respectively meet- and join-reversing and were studied in the context of Distributive Modal Logic by Gehrke, Nagahashi and Venema~\cite{GNV05}, see also \cite{DistMLALBA}. Although $\circ$, $\star$, $\lhd$ and $\rhd$ are not part of our language, their inclusion here illustrates the fact that the classification and subsequent definitions are based solely on the order-theoretic behaviours of the interpretations of connectives and can hence be easily ported to other languages.} The acronym \emph{PIA} stands for ``positive implies atomic''.
\end{definition}

\begin{definition}
\label{def:good branch}

Let $\phi(p_1,\ldots,p_n)$ be a formula in the propositional variables $p_1, \ldots, p_n$,
and let
$\epsilon$ be an order type on $\{1, \ldots, n\}$.
% and $<_{\Omega}$ a strict partial order on the variables $p_1,\ldots p_n$.

A branch in a signed generation tree $\ast \varphi$, for $\ast \in \{+, - \}$, ending in a propositional variable is an $\epsilon$-\emph{good branch} if, apart from the leaf, it is the concatenation of three paths $P_1$, $P_2$, and $P_3$, each of which may possibly be of length $0$, such that $P_1$ is a path from the leaf consisting only of PIA-nodes, $P_2$ consists only of inner skeleton-nodes, and $P_3$ consists only of outer skeleton-nodes and, moreover, it satisfies conditions (GB1), (GB2) and (GB3), below.

\begin{description}

\item[{\bf (GB1)}] The formula corresponding to the uppermost node on $P_1$ is a mu-sentence.

\item[{\bf (GB2)}] For every SRR-node in $P_1$ of the form $\gamma \odot \beta$ or $\beta \odot \gamma$, where $\beta$ is the side where the branch lies, $\gamma$ is a mu-sentence and $\epsilon^\partial (\gamma)\prec \ast\varphi$ (i.e., $\gamma$ contains no variable occurrences to be solved for --- see above).\\
Unravelling the condition $\epsilon^\partial (\gamma)\prec \ast\varphi$ specifically to the $\mathcal{L}_1$-signature (expanded with $\circ$ and $\star$), we obtain:\\
\ a) if $\gamma\odot \beta$ is $+(\gamma \star \beta)$, $+(\gamma \vee \beta)$, $+(\beta \rightarrow \gamma)$, or $-(\beta - \gamma)$, then $\epsilon^\partial(+\gamma)$;\\
\ b) if $\gamma\odot \beta$ is $+(\gamma \to \beta)$, $-(\gamma \wedge \beta)$, $-(\gamma \circ \beta)$, or $-(\gamma - \beta)$, then $\epsilon^{\partial}(-\gamma)$ (equivalently, $\epsilon(+\gamma)$).

\item[{\bf (GB3)}] For every SLR-node in $P_2$ of the form $\gamma \odot \beta$ or $\beta \odot \gamma$, where $\beta$ is the side where the branch lies, $\gamma$ is a mu-sentence and $\epsilon^\partial (\gamma)\prec \ast\varphi$ (see above for this notation).\\
Unravelling the condition $\epsilon^\partial (\gamma)\prec \ast\varphi$ specifically to the $\mathcal{L}_1$-signature (expanded with $\circ$ and $\star$), we obtain:\\
\ a) if $\gamma \odot \beta$ is $-(\gamma \star \beta)$,$-(\gamma \vee \beta)$, $-(\beta \rightarrow \gamma)$, or $+(\beta - \gamma)$, then $\epsilon^{\partial}(-\gamma)$ (equivalently, $\epsilon(+\gamma)$);\\
\ b) if $\gamma \odot \beta$ is $-(\gamma \to \beta)$, $+(\gamma \wedge \beta)$, $+(\gamma \circ \beta)$, or $+(\gamma - \beta)$, then $\epsilon^{\partial}(+\gamma)$.
\end{description}
\end{definition}

As promised, the example below illustrates the concept of an $\epsilon$-good branch.

\begin{example}\label{Examp:Good:Branches}
The generation trees of $+ \Diamond \mu X. (\Diamond X \vee \Box (\Box \Diamond q \vee p ))$ and  $- \nu Y. ([ \Box((q \rightarrow \bot) \wedge (p \rightarrow \bot)) \rightarrow \bot] \wedge \Box Y)$ are given in Figure \ref{Fig:GoodBranches}. Taking $\epsilon$ to be the order type with $\epsilon_p = 1$ and $\epsilon_q = \partial$ there are two $\epsilon$-critical branches, namely the one ending in $+p$ and the one ending in $-q$. These are both $\epsilon$-good. Indeed, they can correctly be split into $P_1$, $P_2$ and $P_3$ paths as indicated in the figure.  Let us verify that the branch ending in $+p$ moreover satisfies (GB1), (GB2) and (GB3): the formula $\Box (\Box \Diamond q \vee p )$ corresponding to the uppermost $P_1$ node is a mu-sentence as it contains no fixed point variables, so (GB1) holds. The only SRR-node on $P_1$ is $+ \vee$ and we must check that it satisfies (GB2): here the role of $\gamma$ is played by $\Box \Diamond q$ which, firstly, is a mu-sentence; secondly, the only occurring propositional variable is $q$ which occurs positively while $\epsilon_q = \partial$, hence $\epsilon^{\partial}(\Box \Diamond q) \prec + \Diamond \mu X. (\Diamond X \vee \Box (\Box \Diamond q \vee p ))$. There are no SLR nodes in $P_2$, so (GB3) holds vacuously.

Turning our attention to the branch ending in $-q$: the formula $\Box((q \rightarrow \bot) \wedge (p \rightarrow \bot))$ corresponding to the uppermost $P_1$ node is a mu-sentence, and hence (GB1) is satisfied. There is one SLR-node, namely $- \rightarrow$, which verifies (GB3) as the formula playing the role of $\gamma$ is the constant sentence $\bot$. Similarly, there is one SRR-node, $+ \rightarrow$, which satisfies (GB2) as, in this instance, formula playing the role of $\gamma$ is again the constant sentence $\bot$.
\end{example}

%We will further be interested in
\noindent Our main interest is in
$\epsilon$-good branches satisfying some of the
%following additional properties:
additional properties in the following definition.

\begin{definition}\label{def:nbpia-omconf}
Let $\phi(p_1,\ldots,p_n)$ be a formula in the propositional variables $p_1, \ldots, p_n$, let $\epsilon$ be an order type on $\{1, \ldots, n\}$ and $<_{\Omega}$ a strict partial order on the variables $p_1,\ldots p_n$. An $\epsilon$-good branch may satisfy one or more of the following properties:

\begin{description}
\item[{\bf (NB-PIA)}] $P_1$ contains no fixed point binders.

%\item[{\bf (NB-SKEL)}] $P_2$ contains no binders.

\item[{\bf (NL)}] For every SLR-node in $P_2$ of the form $\gamma \odot \beta$ or $\beta \odot \gamma$, where $\beta$ is the side where the branch lies, the signed generation tree of $\gamma$ contains no live branches.
Unravelling this condition specifically to the $\mathcal{L}$-signature, we obtain:\\
\ a) if $\gamma \odot \beta$ is $-(\gamma \star \beta)$,$-(\gamma \vee \beta)$, $-(\beta \rightarrow \gamma)$, or $+(\beta - \gamma)$, then $-\gamma$ contains no live branches;\\
\ b) if $\gamma \odot \beta$ is $-(\gamma \to \beta)$, $+(\gamma \wedge \beta)$, $+(\gamma \circ \beta)$, or $+(\gamma - \beta)$, then $+\gamma$ contains no live branches.

\item[{\bf ($\Omega$-CONF)}] For every SRR-node in $P_1$ of the form $\gamma \odot \beta$ or $\beta \odot \gamma$, where $\beta$ is the side where the branch lies: $p_j <_{\Omega} p_i$ for every $p_j$ occurring in $\gamma$, where $p_i$ is the propositional variable labelling the leaf of the branch.
\end{description}

\end{definition}

\begin{figure}
\begin{center}
\begin{tikzpicture}[level distance=6mm]
%\tikzset{style={align=center,anchor=north}}
\tikzstyle{level 6}=[sibling distance=30mm]
\node(A){$+ \Diamond$}
      child{node(B) {$+ \mu$}
            child{node(C) {$+\vee$}
                child{node {$+ \Diamond$}
                    child{node {$+ X$}}
                    }
                child{node(D) {$+\Box$}
                    child{node(E) {$+\vee$}
                        child{node {$+\Box$}
                            child{node {$+ \Diamond$}
                                child{node {$+q$}}
                            }
                        }
                        child{node {$+p$}}
                    }
                }
            }
      };
%ADDING THE BRACES
\draw [decorate,decoration={brace,amplitude=2pt},xshift=-2pt,yshift=0pt]
(0.4, 0.2) -- (0.4,-0.3) node [black,midway,xshift=0.3cm, yshift=0.0cm]
{{\footnotesize{$P_3$}}};

\draw [decorate,decoration={brace,amplitude=2pt},xshift=-2pt,yshift=0pt]
(0.4,-0.4) -- (0.4,-1.3) node [black,midway,xshift=0.3cm, yshift=0.0cm]
{{\footnotesize{$P_2$}}};

\draw [decorate,decoration={brace,amplitude=2pt},xshift=-2pt,yshift=0pt]
(1.2,-1.6) -- (1.2,-2.5) node [black,midway,xshift=0.3cm, yshift=0.0cm]
{{\footnotesize{$P_1$}}};

\end{tikzpicture}
\hspace{2cm}
\begin{tikzpicture}[level distance=6mm]
%\tikzset{style={align=center,anchor=north}}
\tikzstyle{level 2}=[sibling distance=20mm]
\tikzstyle{level 4}=[sibling distance=25mm]
\tikzstyle{level 5}=[sibling distance=25mm]
\tikzstyle{level 6}=[sibling distance=10mm]
\node{$- \nu$}
    child{node {$- \wedge$}
        child{node {$- \rightarrow$}
            child{node {$+ \Box$}
                child{node {$+ \wedge$}
                    child{node {$+ \rightarrow$}
                        child{node {$- q$}}
                        child{node {$+ \bot$}}
                    }
                    child{node {$+ \rightarrow$}
                        child{node {$- p$}}
                        child{node {$+ \bot$}}
                    }
                }
            }
                child{node {$- \bot$}}
        }
        child{node {$- \Box$}
            child{node {$- Y$}}
        }
    };

%ADDING THE BRACES
\draw [decorate,decoration={brace,amplitude=2pt,mirror},xshift=-2pt,yshift=0pt]
(-1.5, 0.1) -- (-1.5,-1.3) node [black,midway,xshift=-0.3cm, yshift=-0.0cm]
{{\footnotesize{$P_2$}}};
\draw [dotted] (-1.5,0.1) -- (-0.2,0.1);

\draw [decorate,decoration={brace,amplitude=2pt,mirror},xshift=-2pt,yshift=0pt]
(-3.7, -1.6) -- (-3.7,-3.2) node [black,midway,xshift=-0.3cm, yshift=-0.0cm]
{{\footnotesize{$P_1$}}};
\draw [dotted] (-3.7,-1.6) -- (-2.1,-1.6);

\end{tikzpicture}
\end{center}
\caption{The generation trees of $+ \Diamond \mu X. (\Diamond X \vee \Box (\Box \Diamond q \vee p ))$ and $- \nu Y. ([ \Box((q \rightarrow \bot) \wedge (p \rightarrow \bot)) \rightarrow \bot] \wedge \Box Y)$ with $\epsilon$-good branches indicated for an order type $\epsilon$ with $\epsilon_p = 1$ and $\epsilon_q = \partial$.}\label{Fig:GoodBranches}
\end{figure}
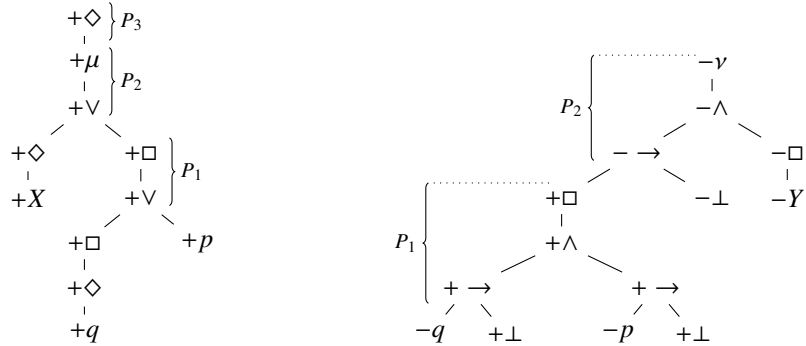

\begin{remark}
\begin{enumerate}
\item Since there is overlap between the signed connectives in the columns of Table \ref{Skeleton:PIA:Node:Table}, the borders between the $P_1$, $P_2$ and $P_3$ parts of a good branch need not be uniquely determined. For strategic reasons, to be discussed below, it is preferable make $P_3$ as long as possible at the cost of $P_2$ and, in turn, make $P_2$ as long as possible at the cost of $P_1$.
\item The abbreviations SLR, SLA, SRA and SRR stand for syntactically left residual, left adjoint, right adjoint and right residual, respectively. Nodes are thus classified according to the order-theoretic properties of their interpretations.
\end{enumerate}
\end{remark}

\begin{definition}\label{def:rstrctd:tm:indctv}
Given an order type $\epsilon$ and a strict partial order $<_{\Omega}$ on the variables $p_1,\ldots p_n$, the signed generation tree $\ast \phi$, $\ast \in \{-, + \}$, of a term $\phi(p_1,\ldots p_n)$ is called

\begin{enumerate}
\item \emph{$\epsilon$-recursive} if every $\epsilon$-critical branch is $\epsilon$-good.
\item \emph{$(\Omega, \epsilon)$-inductive} it is $\epsilon$-recursive and every $\epsilon$-critical branch satisfies ($\Omega$-CONF).
\item \emph{restricted $(\Omega,\epsilon)$-inductive} if it is $(\Omega,\epsilon)$-inductive and
    \begin{enumerate}
    \item every $\epsilon$-critical branch satisfies (NB-PIA) and (NL),
    \item every occurrence of a binder is on an $\epsilon$-critical branch.
    \end{enumerate}
\item \emph{tame $(\Omega,\epsilon)$-inductive} if it is $(\Omega,\epsilon)$-inductive and
    %W: Since there are no binders on critical branches in these formulas, P_2 and hence the need for NL, falls away.
    \begin{enumerate}
    \item $\Omega = \varnothing$,
    \item no binder occurs on any $\epsilon$-critical branch,
    \item the only nodes involving binders which are allowed to occur are $+\nu$ and $-\mu$.
    \end{enumerate}
\end{enumerate}

An inequality $\phi \leq \psi$ is $\epsilon$-recursive (resp., $(\Omega, \epsilon)$-inductive, restricted $(\Omega,\epsilon)$-inductive, tame $(\Omega,\epsilon)$-inductive) if $+\phi$ and $-\psi$ are both $\epsilon$-recursive (resp., $(\Omega, \epsilon)$-inductive, restricted $(\Omega,\epsilon)$-inductive, tame $(\Omega,\epsilon)$-inductive).

An inequality $\phi \leq \psi$ is recursive (resp., inductive, restricted inductive, tame inductive) if $\phi \leq \psi$ is $\epsilon$-recursive (resp., $(\Omega, \epsilon)$-inductive, restricted $(\Omega,\epsilon)$-inductive, tame $(\Omega,\epsilon)$-inductive) for some strict partial order $\Omega$ and order type $\epsilon$.

The corresponding classes of inequalities will be referred to as the \emph{recursive} (resp., \emph{inductive, restricted inductive, tame inductive})
\emph{mu-inequalities}, or the \emph{recursive} (resp., \emph{inductive, restricted inductive, tame inductive}) \emph{mu-formulas}, if the inequality signs have been replaced with implications.
\end{definition}

\begin{remark}
If we were to interpret $\mathcal{L}_1$ formulas classically, the tame inductive formulas would include the \emph{Sahlqvist fixed point formulas} of \cite[Definition 5.1]{Bezhanishvili:Hodkinson:TCS}.
%Although they do \emph{not} constitute a substantial extension, in
In
this setting the tame inductive formulas are, for example, slightly more liberal in terms of what they allow to play the role of a `boxed atom' (see Example \ref{Examp:Tame:Ind}, below).
\end{remark}

\begin{example}\label{Examp:Restricted:Ind}
Consider the inequality $\Diamond \mu X. (\Diamond X \vee \Box (\Box \Diamond q \vee p )) \leq \nu Y. ([ \Box((q \rightarrow \bot) \wedge (p \rightarrow \bot)) \rightarrow \bot] \wedge \Box Y)$. The generation trees for $+ \Diamond \mu X. (\Diamond X \vee \Box (\Box \Diamond q \vee p ))$ and $- \nu Y. ([ \Box((q \rightarrow \bot) \wedge (p \rightarrow \bot)) \rightarrow \bot] \wedge \Box Y)$ were given in Figure \ref{Fig:GoodBranches}.

From the generation trees we can now see that this is a restricted $(\Omega, \epsilon)$-inductive inequality for $\epsilon$ such that $\epsilon_p = 1$ and $\epsilon_q = \partial$ and $\Omega$ such that $q <_{\Omega} p$. Indeed, given this $\epsilon$, there are two $\epsilon$-critical branches: the one ending in $+p$ and the one ending in $-q$. In Example \ref{Examp:Good:Branches} we verified that these branches are $\epsilon$-good.  It is straightforward to check that the branch ending in $+p$ also satisfies (NB-PIA) and (NL). We need to check that the SRR node $+\vee$ satisfies ($\Omega$-CONF): in this case $\Box \Diamond q$ plays the role of $\gamma$, and indeed $\epsilon^{\partial}(+ \Box \Diamond q)$ and $q <_{\Omega} p$.

Now considering the branch ending in $-q$: it is easy to check that is satisfies (NB-PIA) and (NL). To check that is satisfies ($\Omega$-CONF) we need to consider the SRR-node $+ \rightarrow$, but here the role of $\gamma$ is played by $+ \bot$, so clearly $\epsilon^{\partial}(\gamma)$.

The inequality is \emph{not} $(\Omega', \epsilon')$-inductive for any other order type $\epsilon'$ or dependency order $\Omega'$. Indeed, we cannot have $\epsilon'_q = 1$, as the branch ending in $+q$  cannot be correctly divided in  $P_1$, $P_2$ and $P_3$ parts due the presence of $+\Diamond$ in the scope of $+\Box$, and is therefore not good. Also we cannot have $\epsilon'_p = \epsilon'_q$, since in that case there would be a binder in one of the trees the occurrence of which is not on a critical branch, violating Definition \ref{def:rstrctd:tm:indctv}.3(b). So since $\epsilon$ is the only possible order type, the configuration  $+ \Box (\Box \Diamond q \vee p )$ and ($\Omega$-CONF) dictate that $\Omega$ is the only possible dependency order.
\end{example}

\begin{figure}
\begin{center}
\begin{tikzpicture}[level distance=6mm]
%\tikzset{style={align=center,anchor=north}}
\tikzstyle{level 6}=[sibling distance=30mm]
\node{$+ \Diamond$}
    child{node {$+ \wedge$}
        child{node {$+ \Box$}
            child{node {$+\vee$}
                child{node {$+p$}}
                child{node {$+ \Box$}
                    child{node {$+\bot$}}
                }
            }
        }
        child{node {$+ \Box$}
            child{node {$+q$}}
        }
    };

%ADDING THE BRACES
\draw [decorate,decoration={brace,amplitude=2pt, mirror},xshift=-2pt,yshift=0pt]
(-0.3, 0.2) -- (-0.3,-0.7) node [black,midway,xshift=-0.3cm, yshift=0.0cm]
{{\footnotesize{$P_3$}}};

\draw [decorate,decoration={brace,amplitude=2pt, mirror},xshift=-2pt,yshift=0pt]
(-1,-1) -- (-1,-1.9) node [black,midway,xshift=-0.3cm, yshift=0.0cm]
{{\footnotesize{$P_1$}}};
%\draw [dotted] (-1.3,-1) -- (-0.7,-1);
%\draw [dotted] (-1.3,-1.9) -- (-0.8,-1.9);

\draw [decorate,decoration={brace,amplitude=2pt},xshift=-2pt,yshift=0pt]
(0.4, 0.2) -- (0.4,-0.7) node [black,midway,xshift=0.3cm, yshift=0.0cm]
{{\footnotesize{$P_3$}}};

\draw [decorate,decoration={brace,amplitude=2pt},xshift=-2pt,yshift=0pt]
(1.2,-1) -- (1.2,-1.4) node [black,midway,xshift=0.3cm, yshift=0.0cm]
{{\footnotesize{$P_1$}}};
%\draw [dotted] (1.2,-0.4) -- (0.3,-0.4);
%\draw [dotted] (1.2, -1.3) -- (1, -1.3);

\end{tikzpicture}
\hspace{2cm}
\begin{tikzpicture}[level distance=6mm]

%\tikzset{style={align=center,anchor=north}}
\tikzstyle{level 6}=[sibling distance=30mm]
\node{$- \mu$}
    child{node {$-\vee$}
        child{node {$-\Diamond$}
            child{node {$- \wedge$}
                child{node {$-p$}}
                child{node {$-q$}}
            }
        }
        child{node {$-\Box$}
            child{node {$-Y$}}
        }
    };
\end{tikzpicture}
\end{center}
\caption{The generation trees of $+ (\Diamond(\Box \bot \vee p) \wedge \Box q)$ and $- \mu Y. (\Diamond(p \wedge q) \wedge \Box Y)$ with $\epsilon$-good branches indicated for an order type $\epsilon$ with $\epsilon_p = 1 = \epsilon_q$.}\label{Fig:GoodBranches:Inductive}
\end{figure}
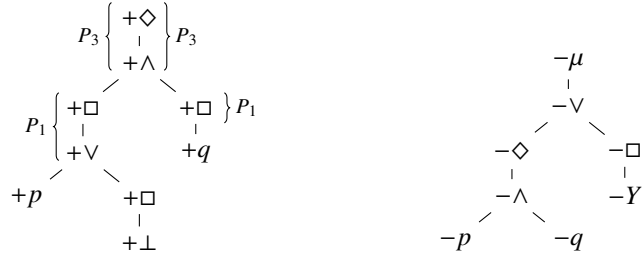

\begin{example}\label{Examp:Rec}
Consider the formula $(\mu X.(p \vee \Diamond X) \wedge \mu X.(q \vee \Diamond X)) \rightarrow \mu X. ( (p \wedge \mu Y.(q \vee \Diamond Y)) \vee \Diamond X)$. This is valid at a point $w$ in a frame iff all point reachable from $w$ are reachable from one another. The corresponding inequality $(\mu X.(p \vee \Diamond X) \wedge \mu X.(q \vee \Diamond X)) \leq \mu X. ( (p \wedge \mu Y.(q \vee \Diamond Y)) \vee \Diamond X)$ is $(\Omega, \epsilon)$-inductive with $\epsilon_p = 1 = \epsilon_q$ (and for no other order type) and any strict partial order $\Omega$. Note that it is not restricted inductive nor tame inductive, since binders occur on both critical and non-critical branches for every possible choice of $\epsilon$.
\end{example}

\begin{example}\label{Examp:Tame:Ind}
The inequality $\Diamond(\Box \bot \vee p) \wedge \Box q \leq \mu Y. (\Diamond(p \wedge q) \wedge \Box Y)$ is tame $(\Omega, \epsilon)$-inductive with $\epsilon_p = 1 = \epsilon_q$ (and $\Omega = \varnothing$), as can be seen from the generation trees in Figure \ref{Fig:GoodBranches:Inductive}. Moreover, the corresponding implication it is \emph{not} a Sahlqvist fixed point formula in the sense of \cite[Definition 5.1]{Bezhanishvili:Hodkinson:TCS}, as that definition (and the classical definition of Sahlqvist formulas) would not allow the configuration $\Box(\Box \bot \vee p)$ where $p$ is a variable to be solved for. It is not tame $(\Omega, \epsilon')$-inductive for any $\epsilon'$ with $\epsilon'_p = \partial$ or $\epsilon'_q = \partial$, as in such cases the binder $\mu$ would appear on an $\epsilon$-critical branch.

Finally note that this inequality is not restricted inductive. Indeed, since the definition of restricted $(\Omega, \epsilon)$-inductive inequalities requires all binders to occur on critical branches, the $- \mu$ would have to be on a critical branch, but then it would have to be on the $P_1$ part of that branch (according to the definition of good branches), yet (NB-PIA) prohibits this for restricted $(\Omega, \epsilon)$-inductive inequalities.
\end{example}

The relationships between the different classes of inequalities or, equivalently, formulas, under consideration are illustrated in the diagram below. The inclusions follow from the definitions. Each of the regions is non-empty: that regions $A$, $B$ and $F$ are non-empty follows from Examples \ref{Examp:Rec}, \ref{Examp:Restricted:Ind} and \ref{Examp:Tame:Ind}, respectively. Region $C$ contains the inductive formulas which are not Sahlqvist and which contain no fixed point binders, e.g., the Frege axiom $(p\to (q\to r)) \to ((p\to q)\to (p\to r))$ (see \cite[Examples 3.16 and 7.5]{DistMLALBA}). Region $D$ contains the binder free (intuitionistic) Sahlqvist formulas, e.g., $p  \rightarrow \Box \Diamond p$, while the Sahlqvist $mu$-formulas which do contain binders are found in region $D$, e.g., $\Diamond p \leq \Box \mu X. (p \vee \Diamond X)$ (cf. \cite[Example 6.4]{Bezhanishvili:Hodkinson:TCS}).

\begin{center}
\def\svgwidth{7cm}
%% Creator: Inkscape 0.48.2, www.inkscape.org
%% Accompanies image file 'Venn.pdf' (pdf, eps, ps)
\begingroup%
  \makeatletter%
  \providecommand\color[2][]{%
    \errmessage{(Inkscape) Color is used for the text in Inkscape, but the package 'color.sty' is not loaded}%
    \renewcommand\color[2][]{}%
  }%
  \providecommand\transparent[1]{%
    \errmessage{(Inkscape) Transparency is used (non-zero) for the text in Inkscape, but the package 'transparent.sty' is not loaded}%
    \renewcommand\transparent[1]{}%
  }%
  \providecommand\rotatebox[2]{#2}%
  \ifx\svgwidth\undefined%
    \setlength{\unitlength}{267.6777832bp}%
    \ifx\svgscale\undefined%
      \relax%
    \else%
      \setlength{\unitlength}{\unitlength * \real{\svgscale}}%
    \fi%
  \else%
    \setlength{\unitlength}{\svgwidth}%
  \fi%
  \global\let\svgwidth\undefined%
  \global\let\svgscale\undefined%
  \makeatother%
  \begin{picture}(1,0.73144984)%
    \put(0,0){\includegraphics[width=\unitlength]{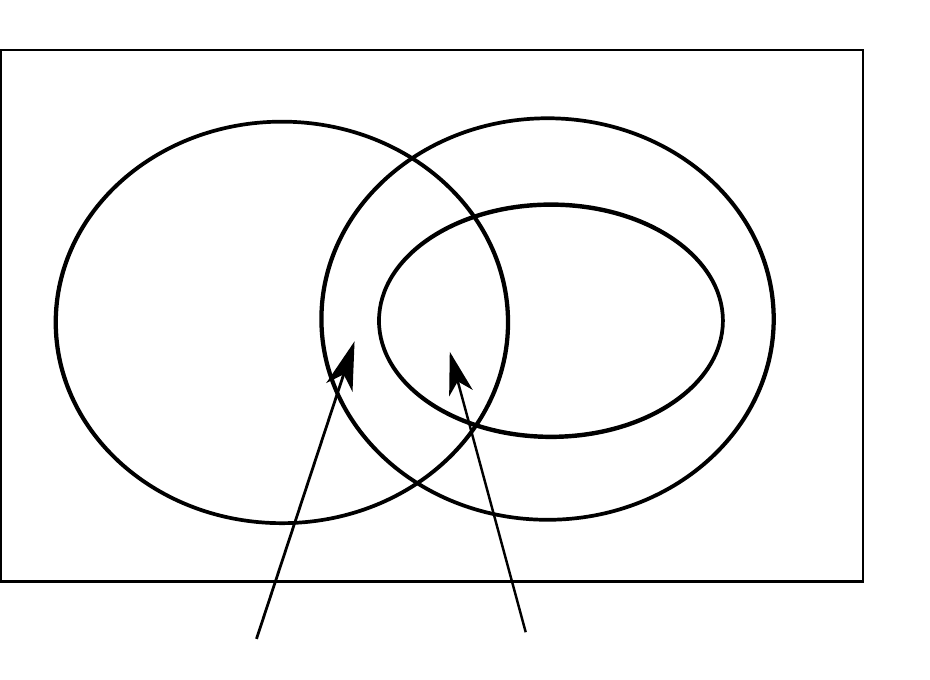}}%
    \put(-0.00257167,0.69406814){\color[rgb]{0,0,0}\makebox(0,0)[lb]{\smash{Recursive $\mu$}}}%
    \put(0.03930334,0.60667681){\color[rgb]{0,0,0}\makebox(0,0)[lb]{\smash{Rest. Ind. $\mu$}}}%
    \put(0.62737382,0.6048562){\color[rgb]{0,0,0}\makebox(0,0)[lb]{\smash{Tame. Ind. $\mu$}}}%
    \put(0.51813476,0.51928556){\color[rgb]{0,0,0}\makebox(0,0)[lb]{\smash{Sahlqvist $\mu$}}}%
    \put(0.20316197,0.00586169){\color[rgb]{0,0,0}\makebox(0,0)[lb]{\smash{Inductive}}}%
    \put(0.51631409,0.00950303){\color[rgb]{0,0,0}\makebox(0,0)[lb]{\smash{Sahlqvist}}}%
    \put(0.08354986,0.1509757){\color[rgb]{0,0,0}\makebox(0,0)[lb]{\smash{$A$}}}%
    \put(0.18550637,0.35853002){\color[rgb]{0,0,0}\makebox(0,0)[lb]{\smash{$B$}}}%
    \put(0.35846827,0.36217136){\color[rgb]{0,0,0}\makebox(0,0)[lb]{\smash{$C$}}}%
    \put(0.45372052,0.36026518){\color[rgb]{0,0,0}\makebox(0,0)[lb]{\smash{$D$}}}%
    \put(0.61335951,0.36217136){\color[rgb]{0,0,0}\makebox(0,0)[lb]{\smash{$E$}}}%
    \put(0.564202,0.19831273){\color[rgb]{0,0,0}\makebox(0,0)[lb]{\smash{$F$}}}%
  \end{picture}%
\endgroup%

\end{center}

\section{The algorithm $\mu^{*}$-ALBA}\label{sec:algomustar}

The algorithm $\mu$-ALBA was introduced in~\cite[Section 1]{muALBA}. Here we present a variant called $\mu^{*}$-ALBA, and which is a restricted version of $\mu$-ALBA. The goal of our algorithm is to \emph{eliminate propositional variables} from inequalities, while maintaining admissible validity. The purpose of this is to make the transition from \emph{admissible validity} to \emph{validity} in the argument for canonicity, as outlined in Section \ref{sec:canon-summary}. Before formally stating the rules which can be applied at each stage of the algorithm, we outline its general strategy.

The final stage of the algorithm, where the actual elimination of propositional variables takes place, is the application of the two Ackermann rules. There are
very specific syntactic requirements of the quasi-inequalities which can be used as input for
the right and left Ackermann rules, (RA) and (LA). The steps before this final elimination stage
are used to get the original inequality into the correct `shape' so that the Ackermann rules can be applied.

\begin{itemize}
\item \textbf{Preprocessing:} operations are distributed according to the signed generation trees (see below),
inequalities are split where possible with ($\vee$LA) and ($\wedge$RA),  and
simple propositional variable elimination takes place using ($\bot$) and ($\top$). Preprocessing may split the original inequality into a number of inequalities, on each of which we proceed separately.
\item \textbf{Conversion of fixed point binders:} all occurrences of $\mu X.\phi(X)$ are converted to $\mu^* X.\phi(X)$
and all occurrences of $\nu X. \psi(X)$ are converted to  $\nu^* X.\psi(X)$. We emphasize that this step is required in both
tame \emph{and} proper runs of $\mu^*$-ALBA (see below).
\item \textbf{First approximation:} an inequality is converted by (FA) into a quasi-inequality consisting of an implication with two inequalities in the
antecedent, and one inequality in the consequent. The inequality in the consequent contains no propositional variables and thus all
steps after this point are aimed only at eliminating propositional variables from the two inequalities in the antecedent.
\item \textbf{Residuation, adjunction and approximation rules:} these rules prepare the antecedents of quasi-inequalities for the application of the Ackermann rules. The residuation and adjunction rules are straightforward applications of properties of the operations on perfect algebras. The approximation rules are more intricate in their formulation. The approximation rules for fixed point binders are the crucial link which enables us to apply this algorithmic approach to the mu-calculus. When dealing with inductive inequalities, a winning strategy (see Section \ref{sec:canon-of-synclasses}) dictating the order in which these various rules are to be applied is determined by the way in which each critical branch is divided into its $P_1,P_2$ and $P_3$ parts as described in the previous section.
\item \textbf{Ackermann rules:} the rules (RA) and (LA) are applied to the quasi-inequalities in the final step to eliminate the propositional variables.
\end{itemize}

Before we introduce the rules of our algorithm we comment briefly on their presentation.
Let $\phi$ and $\psi$ be two $\Tm_1$-terms. Then $\phi \le \psi$ is an \emph{$\Tm_1$-inequality}.
%ac0912 comments about inequalities
We use the symbols $\amp$ and $\Rightarrow$, interpreted as conjunction and implication, respectively, to
combine $\Tm$-inequalities into \emph{quasi-inequalities}.
To show canonicity of an inequality, we must show that the validity of the inequality is preserved when we
move from interpreting it in an algebra $\A$ to interpreting it in the canonical extension $\A^\delta$.
The algorithm $\mu^*$-ALBA restructures the inequality into an equivalent
quasi-inequality (or sometimes a \emph{set of} quasi-inequalities).

As mentioned in the introduction, when formulas from the extended language
$\Tm^+$ are interpreted in $\A^\delta$, an assignment $V$ will have
$V(\NOM)\subseteq \jty(\A^\delta)$ and $V(\CNOM) \subseteq \mty(\A^\delta)$.
Given an assignment $V$, we call $V'$ a $p$-variant (or, potentially a $\nomj$-variant or
$\cnomm$-variant) of $V$ if $V'$ agrees with $V$ on all elements of $\mathsf{PROP}\cup\mathsf{NOM}\cup\mathsf{CNOM}$
except possibly at $p$ (respectively, at $\nomj$ or $\cnomm$). We denote this by writing $V' \sim_p V$.
From this point on we will use the same $V$ to denote both the assignment
$V : \mathsf{PROP} \cup \NOM \cup\CNOM \to \A^\delta$ and the
%ac1408
assignment
%valuation
$V: \Tm^+ \to \A^\delta$.

\paragraph{Preprocessing}
Given an inequality $\phi \le \psi$,
%WEC: EDITED. This made it seem like one type of distribution happened in the one fml and the other type in the other.
%the first action that takes in the pre-processing stage is perform the following operations on
%the positive generation tree of $\phi$ and the negative generation tree of $\psi$:
consider the positive generation tree of $\phi$ and the negative generation tree of $\psi$ and apply the following transformations exhaustively:
\begin{enumerate}
\item distribute positive occurrences of $\Diamond$ and $\wedge$ over positive occurrences of $\vee$;
\item distribute negative occurrences of $\Box$ and $\vee$ over negative occurrences of $\wedge$;
\item apply the \emph{splitting rules} rules ($\vee$LA) and ($\wedge$RA);
\item apply the \emph{monotone variable elimination rules} ($\top$) and ($\bot$).
\end{enumerate}
%The intuition behind the distribution described above is to get the inequalities in the right shape so that the splitting rules described below
%can be applied.

%The  ($\bot$) and ($\top$) rules are in fact just special cases of the Ackermann rules (RA) and (LA), respectively.
%WEC: The polarity conditions were the wrong way around.
%In ($\bot$) we require that $\alpha$ is positive and $\beta$ is negative in $p$, while for ($\top$) we require that $\gamma$ is negative and $\chi$ is positive in $p$.
In ($\bot$) we require that $\alpha$ is negative and $\beta$ is positive in $p$, while for ($\top$) we require that $\gamma$ is positive and $\chi$ is negative in $p$.
\begin{prooftree}
		
    \AxiomC{$\alpha \vee \beta \le \gamma$}
    \RightLabel{($\vee$LA)}
    \UnaryInfC{$\alpha \le \gamma \quad \beta \le \gamma$}

\AxiomC{$ \alpha \le \beta \wedge \gamma$}
    \RightLabel{($\wedge$RA)}
    \UnaryInfC{$\alpha \le \beta \quad \alpha \le \gamma$}

		\noLine\BinaryInfC{}
    \AxiomC{$ \alpha(p) \le \beta(p)$}
    \RightLabel{($\bot$)}
    \UnaryInfC{$ \alpha(\bot) \leq\beta(\bot)$}

		\AxiomC{$\gamma(p)\le \chi(p)$}
    \RightLabel{($\top$)}
    \UnaryInfC{$\gamma(\top) \leq \chi(\top)$}
        \noLine\BinaryInfC{}
    \noLine\BinaryInfC{}
    %\noLine\QuaternaryInfC{}
\end{prooftree}

\paragraph{First approximation rule}

This is the same as for $\mu$-ALBA, viz.:

\begin{prooftree}
    \AxiomC{$\phi \leq \psi $}
    \RightLabel{(FA)}
    \UnaryInfC{ $\forall \nomj\forall\cnomm[(\nomj \leq \phi\ \&\ \psi\leq\cnomm)\Rightarrow \nomj\leq\cnomm]$}
\end{prooftree}
The rule (FA) takes an inequality from $\Tm_*$ (or $\Tm_1$) and replaces it with a quasi-inequality from $\Tm^+$ ($\Tm_1^+$).
%WEC:Removed
%The quantification over $\NOM$ and $\CNOM$ is interpreted in $\A^\delta$ as for the usual semantics of first-order logic.
The soundness of (FA) follows from the set $\jty(\A^\delta)$
(the interpretation of the nominals of $\mathcal{L}^+$) being join-dense in $\A^\delta$, and from
$\mty(\A^\delta)$ (the interpretation of the co-nominals of $\mathcal{L}^+$) being meet-dense in $\A^\delta$
(see the base case in the proof of Proposition~\ref{prop:admissiblesoundness}).

\paragraph{Residuation rules}
These rules are the same as the corresponding group for $\mu$-ALBA.
%The soundness and invertibility of the residuation rules follows from the fact that $\rightarrow$ is both the left and right residual of $\wedge$,
%and $-$ is both the left and right residual of $\vee$.
The soundness and invertibility of the residuation rules follows from the facts that $\vee$ and $\wedge$ are, respectively, the right residual of $-$ and the left residual of $\rightarrow$.

\begin{prooftree}
		\AxiomC{$ \chi - \psi\leq \phi$}
    \RightLabel{($-$LR)}
    \UnaryInfC{$\chi\leq \psi\vee \phi$}

    \AxiomC{$\phi \leq \chi\rightarrow \psi$}
    \RightLabel{($\rightarrow$RR)}
    \UnaryInfC{$\phi\wedge \chi\leq \psi$}
\noLine\BinaryInfC{}
    \AxiomC{$ \chi \wedge \psi\leq \phi$}
    \RightLabel{($\wedge$LR)}
    \UnaryInfC{$ \chi \leq \psi \rightarrow \phi$}

		\AxiomC{$\phi \leq \chi \vee \psi$}
    \RightLabel{($\vee$RR)}
    \UnaryInfC{$\phi - \chi \leq \psi$}
    \noLine\BinaryInfC{}
    \noLine\BinaryInfC{}
    %\noLine\QuaternaryInfC{}
\end{prooftree}

\paragraph{Adjunction rules} The equivalent validity of the formulas above and below the line in
($\vee$LA) and ($\wedge$RA) follows from the fact that $\vee$ is a left adjoint and $\wedge$ a right adjoint of the diagonal map $a \mapsto (a,a)$.
The rules ($\Diamond$LA) and ($\Box$RA) are justified by $\Diamond$ being the left adjoint of $\Bb$, and
$\Box$ being the right adjoint of $\Db$.
\begin{prooftree}
    \AxiomC{$ \phi \vee \chi \leq \psi$}
    \RightLabel{($\vee$LA)}
    \UnaryInfC{$ \phi \leq \psi \quad  \chi \leq \psi$}
    \AxiomC{$ \psi\leq \phi \wedge \chi $}
    \RightLabel{($\wedge$RA)}
    \UnaryInfC{$\psi \leq \phi \quad \psi \leq \chi$}
    \noLine\BinaryInfC{}

    \AxiomC{$\Diamond\phi \leq \psi $}
    \RightLabel{($\Diamond$LA)}
    \UnaryInfC{ $\phi \leq \blacksquare\psi$}

    \AxiomC{$\phi \leq \Box\psi $}
    \RightLabel{($\Box$RA)}
    \UnaryInfC{ $\Diamondblack\phi \leq \psi$}

    \noLine\BinaryInfC{}

    \noLine\BinaryInfC{}
\end{prooftree}

\paragraph{Approximation rules} These four rules have the requirement that the nominals and co-nominals introduced by them need to be \emph{fresh}, i.e., do not occur in the derivation thus far.
%
%s: $\psi$ must be $\cnomn$-free in ($\Box$Appr), $\psi$ must be
%$\nomi$-free in ($\Diamond$Appr), $\chi$ and $\phi$ must be $\nomj,\cnomn$-free in ($\rightarrow$Appr) and in ($-$Appr).
The justification of all four rules uses the join-primeness of elements of $\jty(\A^\delta)$ and the meet-primeness of elements of $\mty(\A^\delta)$
(see the example of ($\rightarrow$Appr) in Proposition~\ref{prop:admissiblesoundness}).

\begin{prooftree}
		\AxiomC{$\Box\psi\leq \cnomm$}
    \RightLabel{($\Box$Appr)}
    \UnaryInfC{$\exists\cnomn(\Box \cnomn \leq \cnomm \ \&\  \psi\leq\cnomn)$}
    \AxiomC{$\nomj \leq \Diamond\psi$}
    \RightLabel{($\Diamond$Appr)}
    \UnaryInfC{$\exists\nomi(\nomj \leq \Diamond\nomi \ \&\  \nomi \leq \psi)$}
    \noLine\BinaryInfC{}
\end{prooftree}

\begin{prooftree}
    \AxiomC{$\chi\rightarrow \phi \leq \cnomm$}
    \RightLabel{($\rightarrow$Appr)}
    \UnaryInfC{$\exists\nomj \exists \cnomn (\nomj\rightarrow \cnomn \leq \cnomm\ \&\ \nomj\leq\chi\ \&\ \phi\leq\cnomn)$}

    \AxiomC{$\nomi\leq \chi- \phi$}
    \RightLabel{($-$Appr)}
    \UnaryInfC{$\exists\nomj \exists\cnomn (\nomi\leq\nomj -\cnomn \ \&\ \nomj\leq\chi  \&\ \phi\leq\cnomn)$}
    \noLine\BinaryInfC{}
\end{prooftree}
We will usually refer to the above group of approximation rules as the \emph{ordinary} approximation rules in
order to distinguish them from the fixed point binder approximation rules that follow.

\paragraph{Approximation rules for fixed point binders}

Our next two rules are used to extract propositional variables that occur within the scope
of a fixed point binder. Let $\psi$ be an $(n+1)$-ary term function, and suppose that there
exists an order type $\tau$ over $n$ such that $\psi(\overline{x},X)$ is completely
$\bigvee$-preserving in $(\overline{x},X) \in
%ac1408
\mathbf{C}^\tau \times \mathbf{C}$ where $\mathbf{C}$
%\mathbb{C}^\tau \times \mathbb{C}$ where $\mathbb{C}$
is any perfect mu-algebra of the second kind. Then (subject to the conditions listed below)
we will be able to define two rules: ($\mu^\tau$-A-R) and ($\nu^\tau$-A-R).
We refer the reader to the definitions at the end of Section~\ref{sec:language} regarding order types and join-irreducible elements
to assist in the reading of these two rules.
\begin{prooftree}
\AxiomC{$\nomi \leq \mu^* X.\psi(\ophi / \overline{x}, X)$}
\RightLabel{($\mu^{\tau}$-A-R)}
\UnaryInfC{ $\bigiamp_{i=1}^{n}( \exists \nomj^{\tau_i}[\nomi \leq \mu^* X.\psi( {\overline{\nomj}_i}^{\tau}/\overline{x}, X)\ \&\ \nomj^{\tau_i} \leq^{\tau_i} \phi_i])$}
\end{prooftree}

\begin{prooftree}
\AxiomC{$ \nu^* X.\varphi(\opsi/\overline{x}, X)\leq \cnomm$}
\RightLabel{($\nu^{\tau}$-A-R)}
\UnaryInfC{ $\bigiamp_{i=1}^{n}( \exists \cnomn^{\tau_i} [\nu^* X.\varphi({\overline{\cnomn}_i}^{\tau}/\overline{x}, X) \leq \cnomm \ \&\ \psi_i \leq^{\tau_i} \cnomn^{\tau_i}$])}
\end{prooftree}
where
\begin{enumerate}
\item in each rule, the variables $\overline{x}\in \mathsf{PHVAR}$ do not occur in any formula in $\overline{\psi}$ or in $\overline{\phi}$;
\item all propositional variables and free fixed point variables in $\psi(\overline{x}, X)$ and $\phi(\overline{x}, X)$ are among $\overline{x}$ and $X$.
\item in ($\mu^{\tau}$-A-R)
the associated term function of $\psi(\overline{x}, X)$ is completely $\bigvee$-preserving in $(\overline{x}, X) \in \Cc^{\tau} \times \Cc$,
for any perfect modal bi-Heyting algebra $\Cc$; in particular we require that $\psi(\overline{x}, X)$ is positive (negative) in $x_i$ if $\tau_i = 1$ ($\tau_i = \partial$);
\item in ($\nu^{\tau}$-A-R)
the associated term function of $\varphi(\overline{x}, X)$ is completely $\bigwedge$-preserving in $(\overline{x}, X) \in \Cc^{\tau} \times \Cc$, for any perfect modal bi-Heyting algebra $\Cc$; in particular we require that $\phi(\overline{x}, X)$ is positive (negative) in $x_i$ if $\tau_i = 1$ ($\tau_i = \partial$).
\end{enumerate}

%\texttt{CHECK!}\\

Note that the difference between these rules and the more general rules, ($\mu^{\tau}$-A) and ($\nu^{\tau}$-A)
(see~\cite[Section 2.3]{muALBA}),
is the absence of the additional tuple, $\overline{z}$, of place holder variables in $\phi$ and $\psi$. This ensures that, in applications of the rule, the resulting inequalities  $\nomi \leq \mu X.\psi( {\overline{\nomj}_i}^{\tau}/\overline{x}, X)$ and $\nu X.\varphi({\overline{\cnomn}_i}^{\tau}/\overline{x}, X) \leq \cnomm$ are pure --- a fact which will be needed in Section \ref{sec:canon-of-synclasses} to justify the applicability of the Ackermann-rules, when we prove that $\mu^*$-ALBA successfully purifies all restricted and tame inductive mu-inequalities.\\

Note that ($\mu^\tau$-A-R) and ($\nu^\tau$-A-R) are the only rules which can result in the overall
quasi-inequality becoming a set of quasi-inequalities. This happens as a result of the introduction
of the disjunction $\bigiamp_{i=1}^{n}$.
%If there have not yet been any applications of ($\mu^\tau$-A-R) or ($\nu^\tau$-A-R),
Given the way we have set up $\mu^*$-ALBA, when the rule is applied, the inequality $\nomi \le \mu^* X.\psi(\varphi/\bar{x},X)$ will occur in the antecedent of an implication as part of a conjunction of inequalities. Once the disjunction $\bigiamp_{i=1}^{n}$ is introduced, this will be distributed over the conjunctions and the implications,
producing a set of $n$ quasi-inequalities.

\begin{example} Consider the mu-inequality $\mu^* X. (p \vee (\Box \bot - q) \vee \Diamond X)) \le \Diamond \Box (p \vee q)$.  The first approximation rule (FA) gives us the quasi-inequality
$\forall \nomi \forall \cnomm \big[ \nomi \le \mu^* X. (p \vee (\Box \bot - q) \vee \Diamond X)  \amp \Diamond \Box (p \vee q)\le \cnomm\Rightarrow \nomi \le \cnomm\big]$. The first inequality from the antecedent gives us the term function $\psi(x,y,X)=x \vee (\Box\bot -y) \vee \Diamond X$ which is $\bigvee$-preserving in $(x,y,X)$ for the order type
$\tau=(1,\partial)$. That is, $\psi(x,y,X)$ is order-preserving as a map from $\Cc^\tau \times \Cc$.
A tuple $\overline{\nomj}_1^\tau$ has the form $(\nomj,\top)$ and a tuple $\overline{\nomj}_2^\tau$ has the form $(\bot,\cnomn)$.
Let $\overline{\phi}=(p,q)$. The rule ($\mu^\tau$-A-R) gives us the following:
\begin{prooftree}
\AxiomC{$\nomi \leq \mu^* X.\psi(\overline{\phi} /\overline{x}, X)$}
%\RightLabel{($\mu^{\tau}$-A-R)}
\UnaryInfC{ $\Big(\exists \nomj^{1} \big[ \nomi \leq \mu^* X.\psi(\overline{\nomj}_1^\tau /\overline{x}, X) \;  \amp  \; (\nomj^{1} \le^{1} p) \big] \Big)
\; \; \iamp \; \; \Big( \exists \nomj^{\partial} \big[\nomi \le \mu^*X.( \overline{\nomj}_2^\tau/\overline{x},X) \; \amp \; (\nomj^\partial \le^\partial q)\big] \Big)
$}
\end{prooftree}
This can be further simplified and written as:
\begin{prooftree}
\AxiomC{$\nomi \leq \mu^* X.\psi(\overline{\phi} /\overline{x}, X)$}
%\RightLabel{($\mu^{\tau}$-A-R)}
\UnaryInfC{ $\Big(\exists \nomj \big[ \nomi \leq \mu^* X.(\nomj \vee (\Box \bot - \top) \vee X) \; \amp \; (\nomj \le p) \big] \Big)
\; \; \iamp \; \; \Big( \exists \cnomn \big[\nomi \le \mu^*X.(\bot \vee (\Box\bot - \cnomn) \vee X) \; \amp \; (\cnomn \geq q)\big] \Big)
$}
\end{prooftree}
\end{example}

%ac1403
\paragraph{Ackermann rules}
The Ackermann rules are used for the crucial task of eliminating propositional variables from
quasi-inequalities. Before we define these rules we need to introduce some terminology and definitions.

Let $\phi \in \mathcal{L}^+_1$ or $\phi \in \mathcal{L}^+_*$. We say that $\phi$ is \emph{positive} (\emph{negative}) in a variable
$p$ if in the generation tree $+\phi$ all $p$-nodes are signed $+$($-$). An inequality $\phi \le \psi$ is \emph{positive} (\emph{negative})
in a variable $p$ if $\phi$ is negative (positive) in $p$ and $\psi$ is positive (negative) in $p$.

\begin{definition}\label{def:synopensynclosed}
The \emph{syntactically open formulas} $\phi$ and \emph{syntactically closed formulas} $\psi$ are defined by simultaneous mutual recursion as follows:
\begin{eqnarray*}
\phi ::= \bot \mid \top \mid p \mid \cnomm \mid \phi_1 \wedge \phi_2 \mid \phi_1 \vee \phi_2 \mid \psi \rightarrow \phi \mid \phi - \psi \mid \Box \phi \mid \Diamond \phi \mid \blacksquare \phi
\mid \nu^{*} X. \phi\\
\psi ::= \bot \mid \top \mid p \mid \nomi \mid \psi_1 \wedge \psi_2 \mid \psi_1 \vee \psi_2 \mid \phi \rightarrow \psi \mid \psi - \phi \mid \Box \psi \mid \Diamond \psi \mid \Diamondblack \psi
\mid \mu^{*} X. \psi
\end{eqnarray*}
where $p \in \mathsf{PROP}$, $\nomi \in \mathsf{NOM}$, and $\cnomm \in \mathsf{CNOM}$.

The \emph{syntactically almost open formulas} and \emph{syntactically almost closed formulas} are defined by adding, respectively,  $\mu^{*} X. \phi$ and  $\nu^{*} X. \psi$ to the recursions above.

Informally, an $\Tm^+_*$-term is \emph{syntactically almost open} if, in it,
all occurrences of nominals and $\Db$ are negative, while all occurrences of co-nominals and $\blacksquare$ are positive. If, in addition, all occurrences of $\mu^*$ are negative and all occurrences of $\nu^*$ positive, the term is syntactically open. Similarly,  an $\Tm^+_*$-term is syntactically almost closed if, in it,
all occurrences of nominals and $\Db$  are positive, while all occurrences of co-nominals and  $\blacksquare$ are negative. If, in addition, all occurrences of $\mu^*$ are positive and all occurrences of $\nu^*$ are negative, the term is syntactically closed.
\end{definition}

Given these definitions we are now able to define our restricted versions of the Ackermann rules:

\begin{prooftree}
    \AxiomC{$\exists p[\bigamp_{i = 1}^{n} \alpha_i \leq p \: \amp \: \bigamp_{j = 1}^{m} \beta_j(p) \leq \gamma_j(p)]$}
    \RightLabel{(RA)}
    \UnaryInfC{$\bigamp_{j = 1}^{m} \beta_j(\bigvee_{i = 1}^{n} \alpha_i / p) \leq \gamma_j(\bigvee_{i = 1}^{n} \alpha_i / p) $}
\end{prooftree}
subject to the restrictions that the $\alpha_i$ are $p$-free and syntactically closed, the $\beta_j$ are positive in $p$ and syntactically closed, while the $\gamma_j$ are negative in $p$ and syntactically open.

\begin{prooftree}
    \AxiomC{$\exists p[\bigamp_{i = 1}^{n} p \leq \alpha_i \: \amp \: \bigamp_{j = 1}^{m} \gamma_j(p) \leq \beta_j(p) ]$}
    \RightLabel{(LA)}
    \UnaryInfC{$\bigamp_{j = 1}^{m} \gamma_j(\bigwedge_{i = 1}^{n} \alpha_i / p) \leq \beta_j(\bigwedge_{i = 1}^{n} \alpha_i / p) $}
\end{prooftree}
subject to the restrictions that the $\alpha_i$ are $p$-free and syntactically open, the $\beta_j$ are positive in $p$ and syntactically open, while the $\gamma_j$ are negative in $p$ and syntactically closed.

A \emph{tame run} of $\mu^*$-ALBA is one during which there are no applications of either ($\mu^\tau$-A-R) or ($\nu^\tau$-A-R).
%This is because there are no occurrences of fixed point binders on critical branches.
By contrast, a \emph{proper run} of $\mu^*$-ALBA is one during which \emph{all} occurrences of fixed point binders are handled
by ($\mu^\tau$-A-R) and ($\nu^\tau$-A-R).
%We say that a tame/proper run of the algorithm $\mu^*$-ALBA \emph{succeeds} if all propositional variables are eliminated from the input inequality
%and we denote the resulting quasi-inequality by $\texttt{pure}(\phi^*\le \psi^*)$.
We say that a run of the algorithm $\mu^*$-ALBA \emph{succeeds} if all propositional variables are eliminated from the input inequality, $\phi \leq \psi$, and we denote the resulting set of pure quasi-inequalities by $\texttt{pure}(\phi^*\le \psi^*)$. An inequality on which some run of $\mu^*$-ALBA succeeds is called a \emph{$\mu^*$-ALBA inequality}.
%It is called a \emph{tame} (resp., \emph{proper}) \emph{$\mu^*$-ALBA inequality} if a tame (resp., proper) run succeeds on it.

\section{Syntactic conditions for meet and join preservation: inner formulas}\label{sec:innerformulas}

In the formulation of the approximation rules ($\mu^{\tau}$-A-R) and ($\nu^{\tau}$-A-R) we required the term functions $\psi(\overline{x},X)$ and $\phi(\overline{x},X)$ to be, respectively, completely $\bigvee$ and $\bigwedge$-preserving as maps from $\A^{\tau} \times \A$ to $\A$. The inner formulas (introduced in \cite[Section 4]{muALBA}) are a syntactically specified classes of formulas, the term functions of which satisfy these properties. Being an inner formula is thus an effectively checkable sufficient condition for the applicability of the  rules ($\mu^{\tau}$-A-R) and ($\nu^{\tau}$-A-R).
%Here is the inductive definition, which makes use of \emph{placeholder variables} from a set $\mathsf{PHVAR}$:

\begin{definition}\label{IF:Definition}
Let $\oy, \oz \subseteq \mathsf{PHVAR}$ and $\oX \subseteq \mathsf{FVAR}$ be tuples of variables which are pairwise different in the union of the their underlying sets. Let $\tau$ be an order-type on $\ox =  \oy \oplus \oX$. The {\em $\tau$-$\Box$ and $\tau$-$\Diamond$ $(\ox, \oz)$-inner formulas}  (\emph{$(\ox, \oz)$-\ifBox{\tau}} and \emph{$(\ox, \oz)$-\ifDia{\tau}}), the free variables of which are contained in $(\ox, \oz)$, are given by the following simultaneous recursion (for the sake of readability, the parameters $\ox$ and $\oz$ are omitted):

\begin{center}
\begin{tabular}{lcccccccccccccc}
\ifBox{\tau}$\ni \phi$ &$\!\!::=\!\!$ &$x_i$ &$\!\!\mid\!\!$ &$\Box \phi$  &$\!\!|\!\!$ &$\phi_1 \wedge \phi_2$ &$\!\!|\!\!$ &$\nu^* Y. \phi' $ &$\!\!|\!\!$ &$\pi \rightarrow \phi$ &$\!\!|\!\!$ &$\pi \vee \phi$ &$\!\!|\!\!$ &$\psi^c \rightarrow \pi$\\
\ifDia{\tau}$\ni \psi$ &$\!\!::=\!\!$ &$x_i$ &$\!\!\mid\!\!$ &$\Diamond \psi$  &$\!\!|\!\!$ &$\psi_1 \vee \psi_2$ &$\!\!|\!\!$ &$\mu^* Y. \psi' $ &$\!\!|\!\!$ &$\psi - \pi$ &$\!\!|\!\!$ &$\pi \wedge \psi$ &$\!\!|\!\!$ &$\pi - \phi^c$\\
%%
%\ifRhd$\ni \xi$ &$\!\!::=\!\!$ & &$\!\!\mid\!\!$ &$\Box \xi$  &$\!\!|\!\!$ &$\xi_1 \wedge \xi_2$ &$\!\!|\!\!$ & &$\!\!|\!\!$ &$\pi %\rightarrow \xi$ &$\!\!|\!\!$ &$\pi \vee \xi$ &$\!\!|\!\!$ &$\psi \rightarrow \pi$\\
%%
%\ifLhd$\ni \chi$ &$\!\!::=\!\!$ & &$\!\!\mid\!\!$ &$\Diamond \chi$  &$\!\!|\!\!$ &$\chi_1 \vee \chi_2$ &$\!\!|\!\!$ & &$\!\!|\!\!$ %&$\chi - \pi$ &$\!\!|\!\!$ &$\pi \wedge \chi$ &$\!\!|\!\!$ &$\pi - \phi$\\
%%
\end{tabular}
\end{center}
\noindent where
\begin{enumerate}
\item $\tau_i = 1$ in the base of the recursion,\label{IF:Base}
\item $\pi$ is $\pi(\oz)\in \mathcal{L}_*$ (specifically, $\pi(\oz)$ contains none of the variables in $\ox$ or $\oX$),\label{IF:Pi}
\item $\phi' = \phi'(\oy \oplus \oX',\oz)$ and $\psi' = \psi'(\oy \oplus \oX',\oz)$ are \ifBox{\tau'} and \ifDia{\tau'}, respectively, with $\oX' = \oX \oplus Y$ and $\tau' = \tau \oplus 1$,\label{IF:FixPoint}
\item $\psi^c\in (\ox, \oz)$-\ifDia{\tau^{\partial}} and $\phi^c\in (\ox, \oz)$-\ifBox{\tau^{\partial}}.\label{IF:Reverse}
\item All other formulas have their free variables among $(\ox, \oz)$.
\end{enumerate}
\end{definition}

\noindent The key fact about $(\ox, \oz)$-\ifBox{\tau} and $(\ox, \oz)$-\ifDia{\tau} formulas is the following:

\begin{lemma}[{\cite[Lemma 4.3]{muALBA}}]\label{IF:Are:Adjoint}
For any perfect modal bi-Heyting algebra $\A$, the term function associated with any \ifBox{\tau} formula $\phi (\ox, \oz)$ (resp., \ifDia{\tau} formula $\psi (\ox, \oz)$) is completely meet-preserving (resp., join-preserving) as a map $\A^{\tau} \rightarrow \A$, fixing the variables $\oz$.

In particular, if the $\oga$ are constant $\mathcal{L}_*^+$ sentences, then the term function associated with $\phi (\ox, \oga/ \oz)$ (resp., $\psi (\ox, \oga/ \oz)$) is completely meet-preserving (resp., join-preserving) as a map $\A^{\tau} \rightarrow \A$.
\end{lemma}

For our purposes then, the most important consequence of this lemma is the fact that ($\nu^{\tau}$-A-R) and ($\mu^{\tau}$-A-R) are respectively applicable to formulas of the form $\psi (\ophi/\ox, \oga/ \oz)$ and $\phi (\opsi / \ox, \oga/ \oz)$,  where $\psi (\ox, \oz)$ is \ifBox{\tau}, $\psi (\ox, \oz)$ is \ifDia{\tau} and the $\oga$ are constant $\mathcal{L}_*^+$ sentences. In particular, ($\nu^{\tau}$-A-R) and ($\mu^{\tau}$-A-R) are applicable to formulas $\psi (\ophi/\ox)$ and $\phi (\opsi / \ox)$,  where $\psi (\ox, \oz)$ is \ifBox{\tau} and $\psi (\ox, \oz)$ is \ifDia{\tau} with $\oz$ the empty tuple.

\section{Soundness of the fixed point approximation rules}\label{sec:sound-fp-approx}

Before we prove the soundness of the Approximation Rules ($\mu^\tau$-A-R) and ($\nu^\tau$-A-R), we require some lemmas regarding
preservation properties of operations.

\begin{lemma}\label{lem:heart} Let\, $\mathbf{L}$ be a complete lattice. If $f : \mathbf{L}^m \times \mathbf{L} \to \mathbf{L}$
is completely join-preserving, then the function $f(x_1,\ldots,x_m,\bot) : \mathbf{L}^m \to \mathbf{L}$ is completely
join-preserving.
\end{lemma}
\begin{proof}
Let $S \subseteq L^m$. We want to show that $f(\TJ S,\bot) = \bigvee\{\, f(s,\bot)\mid s \in S\,\}$. Let $S'=\{\,(s,\bot) \mid s \in S\,\} \subseteq \mathbf{L} \times \mathbf{L}^m$. Then
$\bigvee S'=(\bigvee S,\bot)$ and
$f(\TJ S')=\bigvee\{\, f(s') \mid s' \in S' \,\}=\bigvee \{\, f(s,\bot) \mid s \in S\,\}$.
\end{proof}
We note that this result would also hold if we replaced $\bot$ with any $a \in L$, so long as $S$ is non-empty. To accommodate the case
that $S=\emptyset$, we must have $a = \bot$.

\begin{lemma}\label{lem:2.1equiv} Let $\A$ be a mu-algebra of the second kind and let $\tau$ be an order type over $n$.
Let  $\tau'$ be the order type over $n+1$ defined by $\tau'=\tau \oplus \{1\}$.
%Suppose $f \colon \A^{n+1} \to \A$ is a composition of operations of $\A$ such that
Suppose $f \colon \A^{n+1} \to \A$ is an $\mathcal{L}_{*}$ term function such that
$f^{\A^\delta} : (\A^\delta)^{\tau'} \to \A^\delta$ is completely $\bigvee$-preserving.
By the assumption that $\A$ is of the second kind,
$\mu_2 x. f(a_1,a_2,\ldots,a_n,x)$ exists in $\A$ for all $a_1,a_2, \ldots a_n \in A$.

Let $S \subseteq (\A^\delta)^{\tau}$ such that $\bigvee S \in \A^{\tau}$. Then
$\LFPstar x. f^{\A^\delta}(\TJ S,x) = \bigvee \{\, \LFPstar x.f^{\A^\delta}(s,x)\mid s \in S\,\}$.
\end{lemma}
\proof
%ac1408
We have the following sequence of equalities:
%We will write $f^{\A^\delta}$ for $(f)^{\A^\delta}$. That is, $f^{\A^\delta}$ denotes $f$ interpreted
%as a term function in the algebra $\A^\delta$. Now
\begin{align}
\mu^* x.f^{\A^\delta}(\TJ S,x) & = \bigwedge \{\, a \in A \mid f^{\A^\delta}(\TJ S,a) \le a \,\} \\
& = \bigwedge \{\,a \in A \mid f^{\A}(\TJ S,a) \le a \,\} \\
& = \mu x. f^{\A}(\TJ S,x)\\
& = \LFPtwo x.f^{\A}(\TJ S,x) \\
& = \bigvee_{\alpha \ge 0} (f^{\A})^\alpha(\TJ S,\bot) \\
& = \bigvee_{\alpha \ge 0} (f^{\A^\delta})^\alpha(\TJ S,\bot).
\end{align}
The equivalence of (1) and (2), as well as (5) and (6), follows from the fact that all of the arguments of $f^{\A^\delta}$
are in $\A$. The equivalence of (2) and (3) is the definition of how $\mu x.\phi(x)$ is interpreted,
and the equivalence of (3) and (4) follows from the fact that $\A$ is a mu-algebra of the second kind.

By induction on $\alpha$ we will show that $(f^{\A^\delta})^\alpha(\TJ S,\bot) = \bigvee\{\, (f^{\A^\delta})^\alpha(s,\bot) \mid s \in S \,\}$
and this will be sufficient to prove the overall result.
\begin{itemize}
\item Case $\alpha=0$: \quad $(f^{\A^\delta})^0(\TJ S,\bot) =\bot$.

\item Case $\alpha=1$: \quad $(f^{\A^\delta})^1(\TJ S,\bot) = f^{\A^\delta}( \TJ S,\bot) =
\bigvee\{\, f^{\A^\delta}(s,\bot) \mid s \in S\,\}$. The second equivalence follows from Lemma~\ref{lem:heart}.

\item Successor ordinals: $(f^{\A^\delta})^{\alpha +1}( \TJ S,\bot) = f^{\A^\delta}\Big(\TJ S, (f^{\A^\delta})^\alpha(\TJ S,\bot) \Big)
= f^{\A^\delta} \Big( \TJ S, \bigvee \{\, (f^{\A^\delta})^\alpha(s,\bot) \mid s \in S \,\}\Big)$ by the inductive hypothesis. Since $f^{\A^\delta}$ is completely join-preserving we have
$$ f^{\A^\delta} \Big( \TJ S, \bigvee_{s \in S} (f^{\A^\delta})^\alpha(s,\bot) \Big)= \bigvee_{s \in S} f^{\A^\delta} \Big( s, (f^{\A^\delta})^\alpha(s,\bot) \Big)=
\bigvee_{s \in S} (f^{\A^\delta})^{\alpha+1}(s,\bot) .$$

\item Limit ordinals:
\begin{align*}
\qquad\qquad\qquad(f^{\A^\delta})^\gamma(\TJ S,\bot) & = \bigvee_{\beta < \gamma}(f^{\A^\delta})^\beta(\TJ S,\bot) = \bigvee_{\beta < \gamma} \bigvee_{s \in S} (f^{\A^\delta})^\beta (s,\bot) \qquad \text{(by IH)}\\
\qquad\qquad\qquad& = \bigvee_{s \in S} \bigvee_{\beta < \gamma}  (f^{\A^\delta})^\beta (s,\bot) = \bigvee_{s \in S} (f^{\A^\delta})^\gamma(s,\bot).
\qquad \qquad\qquad \qquad\qquad\qquad \qed
\end{align*}
\end{itemize}

Before demonstrating that the Approximation Rule is sound we should point out that this rule is only ever
applied to a \emph{quasi-inequality} that is the result of an application of the First Approximation Rule (FA).
That is, ($\mu^\tau$-A-R) and ($\nu^\tau$-A-R) are each applied to an inequality which forms \emph{part of}
the antecedent of an implication.
When demonstrating the soundness of the rule ($\mu^\tau$-A-R) it is therefore sufficient
to show that the inequality above the line and the inequality below the line are valid under
%ac1408
assignments
%valuations
which agree everywhere except at some nominal which does not occur in the consequent of the quasi-inequality.

\begin{prop} {\upshape\textbf{(Soundness of ($\mu^\tau$-A-R))}}\label{prop:approxsound}
Let
%ac1408
$\mathbf{C}$
%$\mathbb{C}$
be a
%perfect
mu-algebra of the second kind.
Let $\psi(\overline{x},X)$ and $\overline{\phi}$ be terms in $\mathcal{L}_*$, with
$\psi(\overline{x},X)$ completely $\bigvee$-preserving in $(\overline{x},X) \in (\Cc^\delta)^\tau \times \Cc^\delta$ for an order type
$\tau$. Let $V$ be an admissible assignment on $\Cc^\delta$.

Then $\Cc^\delta, V \models \nomi \le \mu^* X.\psi(\overline{\phi}/\overline{x},X)$ if and only if there exists $i$ ($1 \le i \le n$)
and a $\nomj^{\tau_i}$-variant $V'$ of $V$ such that
$$\Cc^\delta,V' \models \nomi \le \mu^* X. \psi(\overline{\nomj_i}^\tau/\overline{x},X) \quad \text{and} \quad
\Cc^\delta,V' \models \nomj^{\tau_i}\le^{\tau_i} \varphi_i.$$
\end{prop}
\begin{proof}
Suppose that
%ac1408
$\mathbf{C}^\delta,V
%$\mathbb{C},V
\models \nomi \le \mu^* X.\psi(\overline{\phi}/\overline{x},X)$ for some admissible assignment $V$.
By Lemma~\ref{lem:2.1equiv} we have that the term function
$\mu^* X.\psi(\overline{x},X)$ is completely $\bigvee$-preserving in
%ac1408
$(\mathbf{C}^\delta)^\tau$.
%$\mathbb{C}^\tau$.

At this point we will not distinguish between formulas and their interpretations under $V$. Since
%ac1408
$\mathbf{C}^\delta$
%$\mathbb{C}$
is a perfect modal bi-Heyting algebra we have that
$\overline{\phi} = \bigvee\{\, \overline{j} \in
\jty((\mathbf{C}^\delta)^\tau) \mid \overline{j} \le \overline{\phi} \,\}$. Thus we have
\[
\mu^* X.\psi(\overline{\phi},X) = \mu^* X.\psi(\TJ \overline{j}, X) = \TJ \{\, \mu^* X \psi(\overline{j},X) \mid \overline{j} \le \overline{\phi} \,\}.
\]
Now $V(\nomi) \le \bigvee \{\, \mu^* X \psi(\overline{j},X) \mid \overline{j} \le \overline{\phi} \,\}$. Since $V(\nomi)$ is completely join-irreducible and hence completely join-prime we have that there exists $\overline{j_0} \in \jty((\mathbf{C}^\delta)^\tau)$ with
$\overline{j_0} \le \overline{\phi}$ such that $V(\nomi) \le \mu^* X.\psi(\overline{j_0},X)$.
Recall from the end of Section~\ref{sec:language} that $\overline{j_0}$ is $\bot^{\tau_i}$ at every
$1 \le i \le n$ except at one coordinate, say $k$. There we have $(\overline{j_0})_k \in \jty((\mathbf{C}^\delta)^{\tau_k})$
and also $(\overline{j_0})_k \le^{\tau_k} \overline{\phi}_k$.
Let $\nomj_0$ be some nominal for which
$\mu^* X.\psi(\overline{\phi}/\overline{x},X)$ is $\nomj_0$-free.
Now let $V'$ be the $\nomj_0$-variant of $V$ such that $V'(\nomj_0)=(\overline{j_0})_k$.

For the converse, suppose that there exists $i \in \{1,\ldots,n\}$ and $\nomj^{\tau_i}$ such that
$ \Cc^\delta,V' \models \nomi \le \mu^*X.\psi(\overline{\nomj_i}^{\tau}/\overline{x},X)$
and $\Cc^\delta,V' \models \nomj^{\tau_i} \le^{\tau_i} \varphi_i$
where $V'$ is an admissible $\nomj^{\tau_i}$-variant of $V$.
Let us consider the $n$-tuple $\overline{\nomj_i}^\tau$. Given $k\neq i$, if  $\tau_k=1$, then the $k$-th coordinate of
$\overline{\nomj_i}^\tau$ is $\bot$ and hence $(\overline{\nomj_i}^\tau)_k \le \phi_k$. Again for
$k \neq i$, if
$\tau_k=\partial$, then the $k$-th coordinate of $\overline{\nomj_i}^\tau$ is $\top$ and thus
$(\overline{\nomj_i}^\tau)_k \le^\partial \phi_k$.
Thus we have \emph{for all $k$} that $(\overline{\nomj_i}^\tau)_k \le^{\tau_k} \phi_k$.
The fact that $\nomi \le \mu^* X.\psi(\overline{\phi}/\overline{x},X)$ follows
from the fact that $\psi$ is completely join-preserving (and hence monotone) in $\overline{x} \in (\mathbf{C}^\delta)^\tau$.
Thus we have
\[
V'(\nomi) \le V'(\mu^* X. \psi(\overline{\nomj_i}^\tau/\overline{x},X)) \le V'(\mu^* X.\psi(\overline{\phi}/\overline{x},X))
\]
and hence $\mathbf{C}^\delta, V \models \nomi\le  \mu^* X.\psi(\overline{\phi}/\overline{x},X)$.
\end{proof}

The statements and proofs of Lemma~\ref{lem:heart}, Lemma~\ref{lem:2.1equiv} and Proposition~\ref{prop:approxsound} can easily be dualised
and hence we can prove the soundness of the rule ($\nu^\tau$-A-R). \\

The reason that we need to use $\mu^*$  and $\nu^*$ in the rules ($\mu^\tau$-A-R) and ($\nu^\tau$-A-R), respectively, is shown by
the proof of Lemma~\ref{lem:2.1equiv}. If we were to use $\mu x.f^{\A^\delta}(\TJ S,x)=\bigwedge \{\,a \in A^\delta \mid f^{\A^\delta}(\bigvee S,a) \le a\,\}$
in line (1),
we would then have only $(1) \le (2)$ (as $A \subseteq A^\delta$). Thus if we formulated ($\mu^\tau$-A-R) and Lemma~\ref{lem:2.1equiv} with
$\mu$ instead of $\mu^*$, we would not have the equality in Lemma~\ref{lem:2.1equiv} and thus would not be able to show the invariance of admissible validity under ($\mu^\tau$-A-R).

\section{Soundness of the Ackermann rules}\label{sec:soundAckermann}
%\section{All inequalities on which $\mu^*$-ALBA succeeds are canonical}

In this section we prove the soundness of the Ackermann rules, (RA) and (LA). Once we have shown the soundness of (RA) and (LA), we will be able, in the next section, to prove Proposition~\ref{prop:admissiblesoundness}. This proposition declares the soundness of $\mu^*$-ALBA derivations with respect to admissible validity.

We will need the next two technical lemmas, the proofs of which will make extensive use of the algebraic results presented in the appendix. Our strategy closely follows that in~\cite{DistMLALBA}. However, we work algebraically whereas~\cite{DistMLALBA} proceeds in the setting of general frames. Moreover, we need to accommodate fixed point binders which are absent in~\cite{DistMLALBA}.

Fix a modal bi-Heyting algebra $\mathbf{A}$ of the first kind. The set of open elements of $\mathbf{A}^{\delta}$, denoted $\Open(\A^\delta)$, is defined as $\{\bigvee S \mid S \subseteq  A \}$, i.e., as all those elements of $\mathbf{A}^{\delta}$ that can be obtained as arbitrary joins of elements of $\A$. Dually, the set of closed elements of $\mathbf{A}^{\delta}$, denoted $\Clos(\A^\delta)$, is defined as $\{\bigwedge S \mid S \subseteq  A \}$, i.e., as all those elements of $\mathbf{A}^{\delta}$ that can be obtained as arbitrary meets of elements of $\A$. The intention behind the definition of syntactically open and closed $\mathcal{L}^+_*$ formulas is that admissible assignments will always interpret them as open and closed elements of $\A^\delta$, respectively. 

\begin{lemma}\label{Syn:Opn:Clsd:Appld:ClsdUp:Lemma}
Let $\phi(p, \overline{q}, \overline{\nomi}, \overline{\cnomm})$ be syntactically closed and $\psi(p, \overline{q}, \overline{\nomi}, \overline{\cnomm})$ syntactically open. Let $\overline{b} \in A$, $\overline{c} \in \jty(\A^{\delta})$ and $\overline{d} \in \mty(\A^{\delta})$.
Let $k \in \Clos(\A^\delta)$ and $u \in \Open(\A^\delta)$. Then,
\begin{enumerate}
\item
    \begin{enumerate}
    \item If $\phi(p, \overline{q}, \overline{\nomi}, \overline{\cnomm})$ is positive in $p$, then $\phi(k, \overline{b}, \overline{c}, \overline{d}) \in \Clos(\A^\delta)$ %is closed,
    and
    \item if $\psi(p, \overline{q}, \overline{\nomi}, \overline{\cnomm})$ is negative in $p$, then $\psi(k, \overline{b}, \overline{c}, \overline{d}) \in \Open(\A^\delta)$. %is open.
    \end{enumerate}
\item
    \begin{enumerate}
    \item If $\phi(p, \overline{q}, \overline{\nomi}, \overline{\cnomm})$ is negative in $p$, then $\phi(u, \overline{b}, \overline{c}, \overline{d}) \in \Clos(\A^\delta)$,% is a closed up-set
    and
    \item if $\psi(p, \overline{q}, \overline{\nomi}, \overline{\cnomm})$ is positive in $p$, then $\psi(u, \overline{b}, \overline{c}, \overline{d}) \in \Open(\A^\delta)$.
    % is an open up-set.
    %
    \end{enumerate}
\end{enumerate}
\end{lemma}

\begin{proof}
We proceed by simultaneous structural induction on $\phi$ and $\psi$. We show (1). Assume that $\phipqim$ is positive in $p$ and
$\psipqim$ is negative in $p$.
As they do not impact the overall result, we will omit  the parameters $\overline{q}$, $\overline{\nomi}$ and $\overline{\cnomm}$ and simply write $\phi(p)$ and $\psi(p)$ for $\phipqim$ and $\psipqim$ respectively.
The base cases are when $\phi$ is of the form $\top,\bot,p,q$ (where $q$ is a propositional variable different from $p$), or $\phi=\mathbf{i}$, and when $\psi$ is of the form $\top,\bot,q$ (where $q$ is a propositional variable different from $p$), or $\psi=\mathbf{m}$.
The $\phi$ cannot be a co-nominal $\mathbf{m}$ since $\phi$ is syntactically closed but $\mathbf{m}$ syntactically open. Similarly $\psi$ cannot be $p$ or a nominal $\mathbf{i}$ since $\psi$ must be negative in $p$ and any occurrence of a nominal must be negative.

Now clearly $\top,\bot,q$ are all interpreted as clopen elements of $\A^\delta$, as is $p$. Furthermore, the claims follow for $\phi$ any nominal $\mathbf{i}$ and $\psi$ any co-nominal $\mathbf{m}$ since $\jty(\A^{\delta}) \subseteq \Clos(\A^\delta)$ and $\mty(\A^{\delta}) \subseteq \Open(\A^\delta)$.

If $\phi(p)=\phi_1(p) \wedge \phi_2(p)$ or $\phi(p)=\phi_1(p) \vee \phi_2(p)$ then both $\phi_1(p)$ and $\phi_2(p)$ must be syntactically closed
and positive in $p$. Thus by the inductive hypothesis we have that $\phi_1(k) \in \Clos(\A^\delta)$ and $\phi_2(k) \in \Clos(\A^\delta)$
and both their meet and join are in $\Clos(\A^\delta)$ as this is closed under the lattice operations.

If $\psi(p)=\psi_1(p) \wedge \psi_2(p)$ or $\psi(p)=\psi_1(p) \vee \psi_2(p)$ then both
$\psi_1(p)$ and $\psi_2(p)$ must be syntactically open and negative in $p$. By the inductive hypothesis, we have
$\psi_1(k) \in \Open(\A^\delta)$ and $\psi_2(k) \in \Open(\A^\delta)$ and both their meet and join will be in $\Open(\A^\delta)$ as this
is also closed under the lattice operations.

We note that there cannot be occurrences of $\mu$ or $\nu$ in either $\phi(p)$ or $\psi(p)$. If $\phi(p)$ is of the form $\mu^*X.\phi_1(p,X)$, then $\phi(k)=\mu^*X.\phi_1(k,X)=\bigwedge\{\,a \in A \mid \phi_1(k,a)\leq a \,\}$ and so $\phi(k) \in \Clos(\A^\delta)$. If $\psi(p)$ is of the form $\nu^*X.\psi_1(p,X)$, then $\psi(k)=\nu^*X.\psi_1(k,X)=\bigvee \{\,a \in A \mid a \le \psi_1(k,a)\,\}$ and so $\psi(k) \in \Open(\A^\delta)$.

If $\phi(p)= \phi_1(p) - \phi_2(p)$ then $\phi_1(p)$ is syntactically closed and positive in $p$ and $\phi_2(p)$ is syntactically
open and negative in $p$. By the inductive hypothesis, $\phi_1(k) \in \Clos(\A^\delta)$ and $\phi_2(k) \in \Open(\A^\delta)$.
Now by
Lemma~\ref{lem:openclosedextraops}(4)
we have that $\phi_1(k)  - \phi_2(k)=\phi(k) \in \Clos(\A^\delta)$.

If $\psi(p)=\psi_1(p) \rightarrow \psi_2(p)$ then $\psi_1(p)$ is syntactically closed and positive in $p$ while $\phi_2(p)$ is
syntactically open and negative in $p$. By the inductive hypothesis we have
$\psi_1(k) \in \Clos(\A^\delta)$ and $\psi_2(k) \in \Open(\A^\delta)$.
Using Lemma~\ref{lem:openclosedextraops}(3) we see that $\psi_1(k) \rightarrow \psi_2(k) = \psi(k) \in \Open(\A^\delta)$.

Now we look at the cases for the unary connectives. We note that $\phi$ cannot be of the form $\Bb \phi_1$
%or $\btr \phi_1$ as these are
as this is
not syntactically closed. Likewise, $\psi$ cannot be of the form $\Db \psi_1$
%or $\btl \psi_1$ as these are
as this is
not syntactically open.

If $\phi(p)$ is of the form $\Box \phi_1(p)$, $\Diamond\phi_1(p)$ or $\Db\phi_1(p)$ then $\phi_1(p)$ must be syntactically closed and
positive in $p$. By the inductive hypothesis we have that $\phi_1(k) \in \Clos(\A^\delta)$ and then using
Lemma~\ref{lem:ops-usual-preservation}(1), Corollary~\ref{cor:whiteopspreservation}(2) and Lemma~\ref{lem:openclosedextraops}(2) respectively we see
that $\phi(k) \in \Clos(\A^\delta)$.

If $\psi(p)$ is of the form $\Box \psi_1(p)$, $\Diamond \psi_1(p)$ or $\Bb \psi_1(p)$ then $\psi_1(p)$ must be syntactically open and negative in $p$.
We then use the inductive hypothesis to see that $\psi_1(k) \in \Open(\A^\delta)$ and then use
Corollary~\ref{cor:whiteopspreservation}(1), Lemma~\ref{lem:ops-usual-preservation}(2), and Lemma~\ref{lem:openclosedextraops}(1)
respectively to see that $\psi(k) \in \Open(\A^\delta)$.
%
%If $\phi(p)$ is of the form $\lhd \phi_1(p)$, $\rhd \phi_1(p)$ or $\btl \phi_1(p)$, then we must have that $\phi_1(p)$ is syntactically \emph{open}
%and \emph{negative} in $p$. By the inductive hypothesis we have that $\phi_1(k) \in \Open(\A^\sigma)$. Now by
%Corollary~\ref{cor:whiteopspreservation}(4), Lemma~\ref{lem:ops-usual-preservation}(3) and Lemma~\ref{lem:Db-closed-Bb-open}(4) respectively,
%we see that $\phi(k) \in \Clos(\A^\sigma)$.
%
%If $\psi(p)$ is of the form $\lhd \psi_1(p)$, $\rhd \psi_1(p)$, or $\btr \psi_1(p)$ then we must have that
%$\psi_1(p)$ is syntactically \emph{closed} and \emph{positive} in $p$. By the inductive hypothesis $\psi_1(k) \in \Clos(\A^\sigma)$.
%Now by Lemma~\ref{lem:ops-usual-preservation}(4), Corollary~\ref{cor:whiteopspreservation}(3) and
%Lemma~\ref{lem:Db-closed-Bb-open}(3) respectively, we have that $\psi(k) \in \Open(\A^\sigma)$.
\end{proof}

\begin{lemma}\label{lem:A10}
Let $\varphi(p,\overline{q},\overline{\nomi},\overline{\cnomm})$ be syntactically closed and $\psi(p,\overline{q},\overline{\nomi},\overline{\cnomm})$ be
syntactically open. Let $D \subseteq \Clos(\A^\delta)$ be down-directed and let $U \subseteq \Open(\A^\delta)$ be up-directed, let $\overline{b} \in A$, $\overline{c} \in \jty(\A^{\delta})$ and $\overline{o} \in \mty(\A^{\delta})$. Then
\begin{enumerate}
\item
\begin{enumerate}
\item if $\phipqim$ is positive in $p$, then $\phi(\bigwedge D,\overline{b},\overline{c},\overline{o}) =
\bigwedge \{\, \phi(d,\overline{b},\overline{c},\overline{o}) \mid d \in D\,\}$, and
\item if $\psipqim$ is negative in $p$, then $\psi(\bigwedge D,\overline{b},\overline{c},\overline{o}) = \bigvee \{\,\psi(d,\overline{b},\overline{c},\overline{o}) \mid d \in D\,\}$;
\end{enumerate}
\item
\begin{enumerate}
\item if $\phipqim$ is negative in $p$, then $\phi(\bigvee U,\overline{b},\overline{c},\overline{o})= \bigwedge \{\, \phi(u,\overline{b},\overline{c},\overline{o}) \mid u \in U\,\}$, and
\item if $\psipqim$ is positive in $p$, then $\psi(\bigvee U,\overline{b},\overline{c},\overline{o}) =
\bigvee \{\, \psi(u,\overline{b},\overline{c},\overline{o}) \mid u \in U \,\}$.
\end{enumerate}
\end{enumerate}
\end{lemma}
\begin{proof} We prove (1) by simultaneous induction on $\phi$ and $\psi$. As before, we will write $\phi(p)$ for $\phipqim$ and $\psi(p)$ for $\psipqim$.

The base cases of the induction for $\phi$ are when $\phi$ is of the form $\top$, $\bot$, $p$, a propositional variable $q$ other than $p$, or $\nomi$.
The base cases for $\psi$ are those when $\psi$ is of the form $\top$, $\bot$, a propositional variable $q$ other than $p$, or $\cnomm$. In each case
the claim is trivially true.

If $\phi(p)=\phi_1(p) \vee \phi_2(p)$ or $\phi(p)=\phi_1(p) \wedge \phi_2(p)$ then $\phi_1$ and $\phi_2$ are syntactically closed and positive in $p$. Similarly, if $\psi(p)=\psi_1(p) \vee \psi_2(p)$ or $\psi(p)=\psi_1(p) \wedge \psi_2(p)$ then $\psi_1$ and $\psi_2$ are syntactically open and negative in $p$.

Thus when $\phi(p)=\phi_1(p) \wedge \phi_2(p)$ the claim follows by the inductive hypothesis and the
associativity of the meet operation, and when $\psi(p)=\psi_1(p) \vee \psi_2(p)$ the claim follows by the inductive hypothesis and
the associativity of the join operation.

Now suppose that $\phi(p)=\phi_1(p) \vee \phi_2(p)$. By the inductive hypothesis,
$$\phi(\TM D) = \phi_1(\TM D) \vee \phi_2(\TM D) = \Big( \bigwedge_{d \in D} \phi_1(d) \Big) \vee \Big( \bigwedge_{d \in D} \phi_2(d) \Big).$$
Since $\bigwedge \{\, \phi_1(d) \mid d \in D\,\} \le \phi_1(e)$ and $\bigwedge\{\, \phi_2(d)\mid d \in D\,\} \le \phi_2(e)$ for all $e \in D$ we have that
$\phi(\bigwedge D) \le \bigwedge \{\, \phi_1(d) \vee \phi_2(d) \mid d \in D\,\}$. For the reverse inequality, suppose that
$x \in \jty(\A^\delta)$ and $x \nleq \phi(\bigwedge D)$. This implies that
$x \nleq \bigwedge \{\, \phi_1(d)\mid d \in D\,\}$ and $x \nleq \bigwedge \{\, \phi_2(d)\mid d \in D \,\}$. Hence there exists
$d_1,d_2 \in D$ such that $x \nleq \phi_1(d_1)$ and $x \nleq \phi_2(d_2)$. Since $D$ is down-directed, there exists
$d_3 \in D$ such that $d_3 \le d_1$ and $d_3 \le d_2$. Since $\phi_1$ and $\phi_2$ are monotone we have that
$\phi_1(d_3) \le \phi_1(d_1)$ and $\phi_2(d_3) \le \phi_2(d_2)$. This implies that
$x \nleq \phi_1(d_3)$ and $x \nleq \phi_2(d_3)$. By Lemma~\ref{lem:xckappa}(3),
$\phi_1(d_3) \le \kappa(x)$ and $\phi_2(d_3) \le \kappa(x)$ and so
$\phi_1(d_3) \vee \phi_2(d_3) \le \kappa(x)$. Hence $x \nleq \phi_1(d_3) \vee \phi_2(d_3)$ and so
$x \nleq \bigwedge \{\, \phi_1(d) \vee \phi_2(d) \mid d \in D \,\}$. Finally, by Lemma~\ref{lem:JMforleq}(1),
$\bigwedge \{\, \phi(d) \mid d \in D\,\} \le \phi(\bigwedge d)$.

If $\psi(p)=\psi_1(p) \wedge \psi_2(p)$ then by the inductive hypothesis,
$$\psi(\TM D) = \psi_1(\TM D) \wedge \psi_2(\TM D) =
\Big( \bigvee_{d \in D} \psi_1(d) \Big) \wedge \Big( \bigvee_{d \in D} \psi_2(d) \Big).$$
Since $\bigvee\{\, \psi_1(d)\mid d \in D\,\} \ge \psi_1(e)$ and
$\bigvee \{\, \psi_2(d)\mid d \in D\,\} \ge \psi_2(e)$ for all $e \in D$ we see that
$\bigvee\{\, \psi(d)\mid d \in D \,\} \le \psi(\bigwedge D)$.
For the reverse inequality, let $x \in \jty(\A^\delta)$ be such that $x \le \big(\bigvee \{\, \psi_1(d)\mid d \in D \,\} \big)
\wedge \big( \bigvee\{\, \psi_2(d)\mid d \in D\,\} \big)$.
Now $x \le \bigvee \{\, \psi_1(d)\mid d \in D \,\}$ and $x \le \bigvee \{\, \psi_2(d)\mid d \in D \,\}$ and so by Lemma~\ref{lem:xckappa}(3)
we have $\bigvee\{\, \psi_1(d) \mid d \in D \,\} \nleq \kappa(x)$ and $\bigvee \{\, \psi_2(d)\mid d \in D\,\} \nleq \kappa(x)$. This implies that
there exist $d_1,d_2 \in D$ such that $\psi_1(d_1) \nleq \kappa(x)$ and $\psi_2(d_2) \nleq \kappa(x)$. Now by Lemma~\ref{lem:xckappa}(3)
we have $x \leq \psi_1(d_1)$ and $x \leq \psi_2(d_2)$. Furthermore, by the down-directedness of $D$, there exists
$d_3 \in D$ such that $d_3 \le d_1,d_2$. Since $\psi_1$ and $\psi_2$ are antitone in $p$ we have that
$x \leq \psi_1(d_3)$ and $x \leq \psi_2(d_3)$. This gives us that $x \leq \psi_1(d_3) \wedge \psi_2(d_3)$ and hence
$x \leq \bigvee\{\, \psi_1(d) \wedge \psi_2(d) \mid d \in D \,\}$.

Suppose $\phi(p)=\mu^*X.\phi_1(p,X)$ where $\phi_1(p)$ is closed and positive in $p$. For each $d \in D$ we have
$\{\, a \in A \mid \phi_1(d,a) \le a \,\} \subseteq \{\, a \in A \mid \phi_1(\bigwedge D,a) \le a \,\}$ and hence
%ac2811 \TM used to give a smaller \bigwedge
$$\phi(\TM D) = \bigwedge \big\{\, a \in A \mid \phi_1(\TM D,a) \le a \, \big\} \le \bigwedge_{d \in D}
\Big( \bigwedge \big\{\, a \in A \mid \phi_1(d,a) \le a \,\big\}  \Big).$$
Now suppose that $\phi_1(\bigwedge D, a) \le a$. By the inductive hypothesis, $\phi_1(\bigwedge D,a)=\bigwedge \{\, \phi_1(d,a)\mid d \in D\,\}$ and
so $\bigwedge \{\,\phi_1(d,a) \mid d \in D\,\} \le a$. By Lemma~\ref{Syn:Opn:Clsd:Appld:ClsdUp:Lemma}(1a),
$\bigwedge \{\, \phi_1(d,a)\mid d \in D\,\}$ is a meet of closed elements and hence closed. Thus we can apply compactness to get a finite
set $\{d_i\}_{i=1}^n \subseteq D$ such that $\bigwedge_{i=1}^n \phi_1(d_i,a) \le a$. It follows that
$\phi_1(d_1\wedge \ldots \wedge d_n, a) \le a$ and since $D$ is down-directed, there exists $d_a \in D$ such that
$d_a$ is a lower bound for $\{d_i\}_{i=1}^n$. Now $\phi_1(d_a,a) \le a$.

Thus $\mu^*X.\phi_1(d_a,X)= \bigwedge \{\, b \in A \mid \phi_1(d_a,b) \le b \,\} \le a$ for each $a \in A$ with
$\phi_1(\bigwedge D,a) \le a$. Finally this gives
\[
\bigwedge_{d \in D} \phi(d) = \bigwedge_{d \in D} \mu^*X.\phi_1(d,X) \le \bigwedge \{ a \in A \mid \phi_1(\TM D, a) \leq a\} = \phi(\TM D).
\]

If $\psi(p)$ is of the form $\nu^*X.\psi_1(p,X)$ then $\psi_1(p)$ must be open and negative in $p$. Now
$ \psi_1(d,a) \le \psi_1(\bigwedge D,a)$ for all $d \in D$. Hence
$\{\, a \in A \mid a \leq \psi_1(d,a)\,\}\subseteq \{\,a \in A \mid a \leq \psi_1(\bigwedge D,a)\,\}$ for
all $d \in D$. This gives us
$$\bigvee_{d \in D} \psi(d) = \bigvee_{d \in D} \Big\{ \bigvee\{\, a \in A \mid a \le \psi_1(d,a) \} \Big\}
\le   \bigvee\{\, a \in A \mid a \leq \psi_1(\TM D,a)\,\} =\psi(\TM D).$$
Now suppose that $a \le \psi_1(\bigwedge D,a)$.  By the inductive hypothesis, we have that $\psi_1(\bigwedge D,a) = \bigvee \{\, \psi_1(d,a) \mid d \in D\,\}$, and by Lemma~\ref{Syn:Opn:Clsd:Appld:ClsdUp:Lemma}(1b) we have that $\psi_1(d,a)$ is open for each $d \in D$. Hence we can
apply compactness to obtain a finite set $\{d_i\}_{i=1}^n \subseteq D$ such that
$a \le \bigvee_{i=1}^n \psi_1(d_i,a)$.
Now since $D$ is down-directed, there exists $d_a \in D$ such that
$d_a$ is a lower bound for the set $\{d_i \mid 1 \le i \le n\}$.
Now $\psi_1(d_a,a) \ge \bigvee_{i=1}^n \psi_1(d_i,a)$ and hence $a \le \psi_1(d_a,a)$. Now
$a \le \nu^*X.\psi_1(d_a,X)$ and so
$$\psi(\TM D) = \bigvee \{a \in A \mid a \le \psi_1(\TM D,a) \} \le \bigvee_{d \in D} \nu^*X.\psi_1(d,X) = \bigvee_{d \in D} \psi(d).$$

If $\phi(p)$ is of the form $\odot \phi_1(p)$ for $\odot \in \{\Box, \Diamond, \Db\}$, then $\phi_1(p)$ must be
syntactically closed and positive in $p$. By the inductive hypothesis, $\phi_1(\bigwedge D) = \bigwedge \{\, \phi_1(d)\mid d \in D\,\}$. By Lemma~\ref{Syn:Opn:Clsd:Appld:ClsdUp:Lemma}(1a) we have that $\phi_1(d) \in \Clos(\A^\delta)$ for all $d \in D$. As $\phi_1$ is monotone,
and since $D$ is down-directed, we have that $\{\,\phi_1(d)\mid d \in D\,\}$ is a down-directed subset of $\Clos(\A^\delta)$.
Thus we can use the fact that $\Box$ is $\bigwedge$-preserving, Lemma~\ref{lem:whiteEsakia}(2), and Lemma~\ref{lem:extraopspreservation}(2)
to conclude that
$\odot \big(\bigwedge \{\, \phi_1(d) \mid d \in D\,\} \big)= \bigwedge \{\, \odot \phi_1(d)\mid d \in D \,\}$ for $\odot \in \{\Box, \Diamond,\Db\}$.

If $\psi(p)$ is of the form $\odot \psi_1(p)$ for $\odot \in \{\Box,\Diamond,\Bb\}$, then $\psi_1(p)$ must be syntactically open and negative in $p$.
By the inductive hypothesis, $\psi_1(\bigwedge D) = \bigvee_{d \in D} \psi_1(d)$. By Lemma~\ref{Syn:Opn:Clsd:Appld:ClsdUp:Lemma}(1b), 	
each $\psi_1(d) \in \Open(\A^\delta)$. Since $D$ is down-directed, we have that $\{\,\psi_1(d) \mid d \in D \,\}$ is an up-directed set of open elements.
We apply the fact that $\Box$ is $\bigwedge$-preserving, Lemma~\ref{lem:whiteEsakia}(2), and Lemma~\ref{lem:extraopspreservation}(2) to conclude that
$\psi(\bigwedge D) = \odot \psi_1(\bigwedge D) = \odot \Big( \bigvee \{\, \psi_1(d) \mid d \in D \,\} \Big)= \bigvee \{\, \odot \psi_1(d) \mid d \in D\,\}=\bigvee\{\, \psi(d) \mid d \in D \,\}$.

%ac0711 removed
%If $\phi(p)$ is of the form $\odot\phi_1(p)$ for $\odot \in \{\lhd,\rhd,\btl\}$ then $\phi_1(p)$ must be syntactically open and negative in $p$.
%
%If $\psi(p)$ is of the form $\odot\psi_1(p)$ for $\odot \in \{\lhd,\rhd,\btr\}$ then $\psi_1(p)$ must be syntactically closed and positive in $p$.

If $\phi(p)$ is of the form $\phi_1(p) - \phi_2(p)$ then $\phi_1(p)$ is syntactically closed and positive in $p$,
and $\phi_2(p)$ is syntactically open and negative in $p$. By the inductive hypothesis we have
$\phi(\bigwedge D) = \phi_1(\bigwedge D) - \phi_2(\bigwedge D) = \bigwedge \{\, \phi_1(d) \mid d \in D \,\} - \bigvee \{\, \phi_2(e) \mid e \in D \,\}$.
Now by Lemma~\ref{Syn:Opn:Clsd:Appld:ClsdUp:Lemma}(1a) and (1b), for all $d \in D$, we have $\phi_1(d) \in \Clos(\A^\delta)$ and $\phi_2(d) \in \Open(\A^\delta)$. Since
$\phi_1$ is positive in $p$ we have $\{\, \phi_1(d) \mid d \in D \,\}$ is a down-directed set (of closed elements). Similarly, since
$\phi_2$ is negative in $p$ we have $\{\,\phi_2(e) \mid e \in D\,\}$ is an up-directed set (of open elements). Thus we can now apply
Lemma~\ref{lem:extraopspreservation}(3) to get $\phi(\bigwedge D) = \bigwedge  \{\, \phi_1(d) - \phi_2(e) \mid d,e \in D \,\}$. Hence
$\phi(\bigwedge D) \le \bigwedge \{\, \phi_1(d) - \phi_2(d) \mid d \in D \,\}=\bigwedge \{\, \phi(d) \mid d \in D \,\}$.
Now let $d, e \in D$. Since $D$ is down-directed, there exists $f$ such that $f \le d$ and $f \le e$. This gives us
$\phi_1(f) \le \phi_1(d)$ and $\phi_2(e) \le \phi_2(f)$. By the fact that $-$ is order-preserving in its
first argument and order-reversing in its second argument, we get that $\phi_1(f) - \phi_2(f) \le \phi_1(d) - \phi_2(e)$. Thus
we have $\phi(\bigwedge D) = \bigwedge \{\, \phi(d) \mid d \in D \,\}$.

If $\psi(p)$ is of the form $\psi_1(p) \rightarrow \psi_2(p)$ then $\psi_1(p)$ is syntactically closed
and positive in $p$ and $\psi_2(p)$ is syntactically open and negative in $p$. By the inductive hypothesis we have
$$\psi(\TM D) = \psi_1(\TM D) \rightarrow \psi_2(\TM D) = \bigwedge \{\, \psi_1(d) \mid d \in D \,\} \rightarrow
\bigvee \{\, \psi_2(e) \mid e \in D \,\}.$$
Again, $\{\,\psi_1(d) \mid d \in D \,\}$ is a down-directed subset of $\Clos(\A^\delta)$ and
$\{\,\psi_2(e) \mid e \in D \,\}$ is an up-directed subset of $\Open(\A^\delta)$. We apply Lemma~\ref{lem:extraopspreservation}(4) to get
$\psi(\bigwedge D) = \bigvee \{\, \psi_1(d) \rightarrow \psi_2(e) \mid d,e \in D \,\}$.
It is then clear that $\psi(\bigwedge D) \ge \bigvee\{\, \psi_1(d) \rightarrow \psi_2(d) \mid d \in D \,\}$.
Since $D$ is down-directed, for any $d,e \in D$, there exists a lower bound $f$ for $d,e$.
Now $\psi_1(f) \le \psi_2(d)$ and $\psi_2(e) \le \psi_2(f)$ and so $\psi_1(d) \rightarrow \psi_2(d) \le \psi_1(f) \rightarrow \psi_2(f)$.
Hence we conclude that $\psi(\bigwedge D) = \bigvee \{\, \psi_1(f) \rightarrow \psi_2(f) \mid f \in D \,\} = \bigvee \{\, \psi(d) \mid d \in D \,\}$.
\end{proof}

We are now ready to prove the Ackermann lemmas which will justify the rules (RA) and (LA).

\begin{lemma}[Righthanded Ackermann lemma for mu-algebras]\label{Ackermann:Mu:Alg:Right:Lemma}
Let $\mathbf{A}$ be a mu-algebra of the first kind. Let $\alpha(\overline{q}, \overline{\nomi}, \overline{\cnomn})$, $\beta(p, \overline{q}, \overline{\nomi}, \overline{\cnomn})$ and $\gamma(p, \overline{q}, \overline{\nomi}, \overline{\cnomn})$ be
%ac2811
$\mathcal{L}^{+}_*$-formulas such that
%$\mathcal{L}^{+}_{\mu}$-formulas such that
\begin{enumerate}
\item[(i)] $\alpha(\overline{q}, \overline{\nomi}, \overline{\cnomn})$ is syntactically closed and does not contain any occurrences of $p$,
\item[(ii)] $\beta(p, \overline{q}, \overline{\nomi}, \overline{\cnomn})$ is syntactically closed and positive in $p$, and
\item[(iii)] $\gamma(p, \overline{q}, \overline{\nomi}, \overline{\cnomn})$ is syntactically open and negative in $p$.
\end{enumerate}
Then for any $\overline{b} \in A$,  $\overline{c} \in \jty(\mathbf{A}^{\delta})$ and $\overline{d} \in \mty(\mathbf{A}^{\delta})$, the following are equivalent
\begin{enumerate}
\item there exists $a \in A$ such that $\alpha(\overline{b}, \overline{c}, \overline{d}) \leq a \,\textrm{ and }\, \beta(a, \overline{b}, \overline{c}, \overline{d}) \leq \gamma(a, \overline{b}, \overline{c}, \overline{d})$, 

\item $\beta(\alpha(\overline{b}, \overline{c}, \overline{d}), \overline{b}, \overline{c}, \overline{d}) \leq \gamma(\alpha(\overline{b}, \overline{c}, \overline{d}), \overline{b}, \overline{c}, \overline{d})$.
\end{enumerate}
\end{lemma}

%\texttt{Do we need to take $(\cdot)^{*}$ translations to make the above work?}

\begin{proof}
For the implication from top to bottom, it suffices to appeal to the monotonicity of  $\beta$ in $p$ and the antitonicity of $\gamma$ in $p$.

For the sake of the converse implication, assume that $\beta(\alpha(\overline{b}, \overline{c}, \overline{d}), \overline{b}, \overline{c}, \overline{d}) \leq \gamma(\alpha(\overline{b}, \overline{c}, \overline{d}), \overline{b}, \overline{c}, \overline{d})$. By Lemma \ref{Syn:Opn:Clsd:Appld:ClsdUp:Lemma}, $\alpha(\overline{b}, \overline{c}, \overline{d})$ is closed. Hence,  $\alpha(\overline{b}, \overline{c}, \overline{d}) = \bigwedge \{\,u \in A \mid \alpha(\overline{b}, \overline{c}, \overline{d}) \leq u \,\}$, making it the meet of the down-directed family $U = \{\,u \in A \mid \alpha(\overline{b}, \overline{c}, \overline{d}) \leq u \,\}$ of clopen elements.  Thus we have
$$\beta(\TM U, \overline{b}, \overline{c}, \overline{d}) \leq \gamma(\TM U, \overline{b}, \overline{c}, \overline{d}).$$
Since $\beta$ is syntactically closed and positive in $p$, and $\gamma$ is syntactically open and negative in $p$, we may apply Lemma \ref{lem:A10} to obtain
\[
\bigwedge \{\,\beta (u, \overline{b}, \overline{c}, \overline{d}) \mid u \in U \,\} \leq \bigvee \{\,\gamma(u, \overline{b}, \overline{c}, \overline{d}) \mid u \in U \,\}.
\]
By Lemma \ref{Syn:Opn:Clsd:Appld:ClsdUp:Lemma}, $\beta (u, \overline{b}, \overline{c}, \overline{d})$ is closed and $\gamma (u, \overline{b}, \overline{c}, \overline{d})$ is open for each $u \in A$. Hence, by compactness,
\begin{equation}
\label{eq:compactness:topol:Ackermann}
\bigwedge_{i = 1}^{m} \beta (u_i, \overline{b}, \overline{c}, \overline{d}) \leq \bigvee_{j = 1}^{n} \gamma(u'_j, \overline{b}, \overline{c}, \overline{d}),
\end{equation}
for some $u_1, \ldots, u_m \in A$ with $\alpha(\overline{b}, \overline{c}, \overline{d}) \leq u_i$ for $1 \leq i \leq m$, and some $u'_1, \ldots, u'_{n} \in A$ with $\alpha(\overline{b}, \overline{c}, \overline{d}) \leq u'_j$ for all $1 \leq j \leq n$.

Let $a = u_1 \wedge \cdots \wedge u_{m} \wedge u'_1 \wedge \cdots \wedge u'_n$. Then $\alpha(\overline{b}, \overline{c}, \overline{d}) \leq a \in A$. By the monotonicity of $\beta$ in $p$, the antitonicity of $\gamma$ in $p$, and \eqref{eq:compactness:topol:Ackermann}, it follows that
\begin{equation}
\label{eq:compactness:topol:Ackermann:2}
\beta (a, \overline{b}, \overline{c}, \overline{d}) \leq \gamma(a, \overline{b}, \overline{c}, \overline{d}).
\end{equation}
\end{proof}

\begin{lemma}[Lefthanded Ackermann lemma for mu-algebras]\label{Ackermann:Mu:Alg:Left:Lemma}
Let $\mathbf{A}$ be a mu-algebra of the first kind.
Let $\alpha(\overline{q}, \overline{\nomi}, \overline{\cnomn})$, $\beta(p, \overline{q}, \overline{\nomi}, \overline{\cnomn})$ and $\gamma(p, \overline{q}, \overline{\nomi}, \overline{\cnomn})$ be
%ac2811
$\mathcal{L}^{+}$-formulas such that
%$\mathcal{L}^{+}_{\mu}$-formulas such that
\begin{enumerate}
\item[(i)] $\alpha(\overline{q}, \overline{\nomi}, \overline{\cnomn})$ is syntactically open and does not contain any occurrences of $p$,
\item[(ii)] $\beta(p, \overline{q}, \overline{\nomi}, \overline{\cnomn})$ is syntactically closed and negative in $p$, and
\item[(iii)] $\gamma(p, \overline{q}, \overline{\nomi}, \overline{\cnomn})$ is syntactically open and positive in $p$.
\end{enumerate}
Then for any $\overline{b} \in A$,  $\overline{c} \in \jty(\mathbf{A}^{\delta})$ and $\overline{d} \in \mty(\mathbf{A}^{\delta})$,
 the following are equivalent
 
\begin{enumerate}
\item there exists $a \in A$ such that $a \leq \alpha(\overline{b}, \overline{c}, \overline{d}) \,\textrm{ and }\, \beta(a, \overline{b}, \overline{c}, \overline{d}) \leq \gamma(a, \overline{b}, \overline{c}, \overline{d})$,
\item $\beta(\alpha(\overline{b}, \overline{c}, \overline{d}), \overline{b}, \overline{c}, \overline{d}) \leq \gamma(\alpha(\overline{b}, \overline{c}, \overline{d}), \overline{b}, \overline{c}, \overline{d})$.
\end{enumerate}
\end{lemma}

\section{Canonicity of $\mu^*$-ALBA inequalities}\label{sec:canonicity}

We now have all of the machinery in place to prove that admissible validity for an inequality $\phi^* \le \psi^*$
is preserved under any number of applications of rules from $\mu^*$-ALBA. This will be crucial in showing both
canonicity and tame canonicity for classes of inequalities on which the algorithm succeeds.

In this section we will assume that each of the inequalities that we are working with has already undergone preprocessing
(see Section~\ref{sec:algomustar}). The fact that both admissible validity and ordinary validity is preserved throughout the
preprocessing stage is straightforward to prove.

\begin{prop}{\rm (Soundness of $\mu^*$-ALBA rules w.r.t. admissible validity)}\label{prop:admissiblesoundness}
Let $\mathbf{A}$ be a mu-algebra of the second kind, and $\phi \leq \psi$ an $\mathcal{L}_1$-inequality.
Let\, $\mathsf{FA}(\phi^* \le \psi^*):= \forall \nomi \forall \cnomm ( \nomi \le \varphi^* \: \amp \: \psi^* \le \cnomm \Rightarrow \nomi \le \cnomm)$
and let $\{\,\mathsf{QIneq}_k \mid 1 \le k \le n \,\}$ be a set of quasi-inequalities obtained from $\mathsf{FA}(\phi^* \leq \psi^*)$
through the application of $\mu^*$-ALBA rules.
Then\,
$\mathbf{A}^{\delta} \models_{\mathbf{A}} \phi^* \le \psi^*$ iff\, $\mathbf{A}^{\delta} \models_{\mathbf{A}} \{\,\mathsf{QIneq}_k \mid 1 \le k \le n\,\}$.
\end{prop}

\begin{proof}
We proceed by induction on the number of rule applications. For the base case, suppose that there are no applications of rules
to $\mathsf{FA}(\phi^* \le \psi^*)$. Assume $\A^\delta \models_{\A} \phi^* \leq \psi^*$. Let $V$ be an admissible assignment such that $V(\nomi) \leq V(\phi^*)$ and $V(\psi^*) \leq V(\cnomm)$. By assumption we have that $V(\phi^*) \leq V(\psi^*)$ and hence by transitivity we have $V(\nomi) \leq V(\cnomm)$.

Conversely, assume that $\A^\delta \models_{\A} \nomi \leq \phi^* \: \&  \: \psi^* \leq \cnomm \Rightarrow \nomi \leq \cnomm$. Let $V$ be any admissible assignment. We need to show that $V(\phi^*) \leq V(\psi^*)$. By the join-density of $\jty(\A^\delta)$ and the meet-density of $\mty(\A^\delta)$, it is sufficient to show that for every $i \in \jty(\A^\delta)$ and $m \in \mty(\A^\delta)$, if $i \leq V(\phi^*)$ and $V(\psi^*) \leq m$, then $i \leq m$. Accordingly, suppose that $i \in \jty(\A^\delta)$ and $m \in \mty(\A^\delta)$ such that $i \leq V(\phi^*)$ and $V(\psi^*) \leq m$. Let  $V'\sim_{\nomi,\cnomm} V$ such that $V'(\nomi) = i$ and $V'(\cnomm) = m$. Then, $V'(\nomi) \leq V'(\phi^*)$ and $V'(\psi^*) \leq V'(\cnomm)$, so
by our assumption $i \leq m$.

Now suppose that $\mathcal{Q}_1=\{\,\mathsf{QIneq}_i \mid 1 \le i \le n \,\}$ is a set of quasi-inequalities obtained from $\mathsf{FA}(\phi^* \le \psi^*)$ via $k$ rule applications.
Let $\mathcal{Q}_2=\{\,\mathsf{QIneq}_i \mid i \neq m, 1 \le m \le n\,\} \cup \{\,\mathsf{QIneq}_{m_j} \mid 1 \le j \le \ell\,\}$
be a set of quasi-inequalities obtained from $\mathcal{Q}_1$ via a single rule application.
(To be clear,  $\mathcal{Q}_2$ is the same as $\mathcal{Q}_1$ except for the quasi-inequality $\mathsf{QIneq}_m \in \mathcal{Q}_1$
which has been changed via the single rule application.)

We now cover the various possible cases for the $(k+1)$-th rule application that gives us
$\{\,\mathsf{QIneq}_{m_j} \mid 1 \le j \le \ell \,\}$ from $\mathsf{QIneq}_m$.
The cases of the residuation and adjunction rules are clear from the
properties of the residuals and adjoints.

We prove the case of ($\rightarrow$Appr).
Let $V$ be an admissible assignment such that $V(\chi \rightarrow \phi) \le V(\cnomm)$. We observe that
$\big( \TJ\{\, j \in \jty(\A^\delta) \mid j \le V(\chi)\,\}\big) \rightarrow \big(\TM\{\, n \in \mty(\A^\delta) \mid V(\phi) \le n \,\}\big) \le V(\cnomm)$.
Since $\rightarrow$ is completely join-reversing in its first coordinate and completely meet-preserving in its second
coordinate, we get $ \bigwedge\{\, j \rightarrow n \mid j \le V(\chi), V(\phi) \le n\,\} \le V(\cnomm)$. As
$V(\cnomm) \in \mty(\A^\delta)$ and hence completely meet-prime, there must exist $j_0 \in \jty(\A^\delta)$ with
$j_0 \le V(\chi)$ and $n_0 \in \mty(\A^\delta)$ with $V(\phi) \le n_0$ such that $j_0 \rightarrow n_0 \le V(\cnomm)$. Now choose
$\nomj \in \NOM, \cnomn \in \CNOM$ such that $\chi \rightarrow \phi\le \cnomm$ is $(\nomj,\cnomn)$-free and construct
a $(\nomj,\cnomn)$-variant $V'$ of $V$ such that $V'(\nomj)=j_0$ and $V'(\cnomn)=n_0$.
The opposite direction only requires the observation that $\rightarrow$ is order-reversing in the first
coordinate and order-preserving in the second coordinate.

The soundness and invertibility of the remaining ordinary approximation rules will follow using properties of canonical extension $\A^\delta$, especially the join-density and join-primeness of $\jty(\A^\delta)$ and the meet-density and meet-primeness of $\mty(\A^\delta)$.

The case of the rule application being that of the Approximation Rules ($\mu^\tau$-A-R) and ($\nu^\tau$-A-R) follows from
Proposition~\ref{prop:approxsound}. We observe that these are the only rules which could possibly result in $\ell > 1$.

The case for the Ackermann rules (RA) and (LA) follow from Lemmas~\ref{Ackermann:Mu:Alg:Right:Lemma} and~\ref{Ackermann:Mu:Alg:Left:Lemma}, respectively.
\end{proof}

\begin{cor}\label{cor:type1soundness} Let $\A$ be a mu-algebra of the first kind,
and $\phi \leq \psi$ an $\Tm_1$-inequality. Let $\mathsf{QIneq}$ be a quasi-inequality
obtained from $\mathsf{FA}(\phi^* \le \psi^*)$ through the application of $\mu^*$-ALBA rules
none of which is ($\mu^\tau$-A-R) or ($\nu^\tau$-A-R). Then
$\A^\delta \models_{\A} \phi^* \le \psi^*$ if and only if $\A^\delta \models_{\A} \mathsf{QIneq}$.
\end{cor}
\begin{proof} The only part of the proof of Proposition \ref{prop:admissiblesoundness} above that requires
$\A$ to be a mu-algebra of the second kind is the case of ($\mu^\tau$-A-R) and ($\nu^\tau$-A-R).
Those cases rely on Proposition~\ref{prop:approxsound}, which is proved using Lemma~\ref{lem:2.1equiv}. The proof of that lemma requires
$\A$ be of the second kind. The soundness and invertibility of the remaining rules are
proved exactly as in Proposition~\ref{prop:admissiblesoundness} and do not require
$\A$ to be of the second kind.
\end{proof}

In order to complete the right-hand side of the U-shaped argument, we need to know that
the rules can be `undone' when using arbitrary assignments on the algebras that are the canonical extensions of
mu-algebras of the first kind.

\begin{prop}
\label{prop:ordinarysoundness}
Let $\A$ be a mu-algebra of the first kind.
For any $\Tm_1$-inequality $\phi \le \psi$ on which a tame run of $\mu^*$-ALBA succeeds,
we have that $\A^\delta \models$ {\upshape\texttt{pure}}$(\varphi^* \le \psi^*)$
if and only if  $\A^\delta \models \varphi^* \le \psi^*$.
\end{prop}
\begin{proof}
Since $\texttt{pure}(\phi^* \le \psi^*)$ is the result of a successful \emph{tame} run of $\mu^*$-ALBA, we know that
$\texttt{pure}(\phi^* \le \psi^*)$ is a single quasi-inequality obtained via some finite
number of rule applications from $\mu^*$-ALBA, not including ($\mu^\tau$-A-R) or ($\nu^\tau$-A-R). Not including (RA) and (LA),
the soundness and invertibility
of each of the rules under arbitrary (i.e., not necessarily admissible)
%ac1408
assignments
%valuations
is clear from the fact that the admissibility of the
%ac1408
assignments
%assignments/valuations
was not used in any of the proofs in Proposition~\ref{prop:admissiblesoundness}.

%Although the Ackermann Lemmas do make use of the fact that we are interpreting the formulas using
%admissible assignments on $\A^\delta$, the proofs are easily generalized to the case of arbitrary assignments.
When working on perfect algebras, the Ackermann lemmas are easy to prove, and depend only on the monotonicity and antitonicity of the formulas involved. For (RA), the direction from top to bottom can be proved
using only the fact that $\beta$ is positive in $p$ and $\gamma$ is negative in $p$. The direction from
bottom to top can be shown by taking $V'$ to be the $p$-variant of $V$ such that $V'(p)=V(\alpha)$.
\end{proof}

\begin{figure}
\begin{center}
\begin{tikzpicture}
\path (0,0) node(a) {$\A \models \phi \le \psi$}
			(0,-0.5) node(b) [rotate=90] {$\Leftrightarrow$}
			(0,-1) node(c) {$\A^{\delta} \models_\A \phi^* \le \psi^*$}
			(0,-1.5) node(i) [rotate=90] {$\Leftrightarrow$}
			(0,-2) node(d) {$\A^{\delta} \models_\A \texttt{pure}(\phi^* \le \psi^*)$}
			(2.5,-2) node(j) {$\Longleftrightarrow$}
			(5,-2) node(f) {$\A^\delta \models \texttt{pure}(\phi^* \le \psi^*)$}
			(5,-1) node(g) [rotate=90] {$\Longrightarrow$}
			(5,0) node(k) {$\A^\delta \models \phi^* \le \psi^*$};
\end{tikzpicture}
\caption{The proof of tame canonicity for $\Tm_1$-inequalities on which a tame run of $\mu^*$-ALBA succeeds.}
\label{fig:final-U}
\end{center}
\end{figure}

\begin{theorem}\label{thm:tamecanonicity}{\upshape\textbf{(Tame Canonicity)}}
All $\mathcal{L}_1$-inequalities on which a tame run of $\mu^*$-ALBA succeeds are tame canonical.
\end{theorem}
\begin{proof}
Suppose that a tame run of $\mu^*$-ALBA succeeds on the inequality $\varphi \le \psi$.
%That is, $\texttt{pure}(\phi\le \psi)$ contains no propositional variables.
Then we have the following sequence of equivalences:
\begin{eqnarray}
&&\A \models \phi \leq \psi \label{eq:tamecanonicity:1}\\
&\Longleftrightarrow  &\A^\delta \models_{\A} \phi^{*} \leq \psi^{*}\label{eq:tamecanonicity:2}\\
&\Longleftrightarrow  &\A^{\delta} \models_{\A} \texttt{pure}(\phi^{*} \leq \psi^{*})\label{eq:tamecanonicity:3}\\
&\Longleftrightarrow  &\A^{\delta} \models \texttt{pure}(\phi^{*} \leq \psi^{*})\label{eq:tamecanonicity:4}\\
&\Longleftrightarrow  &\A^{\delta} \models \phi^{*} \leq \psi^{*}\label{eq:tamecanonicity:5}
\end{eqnarray}
The equivalence of \eqref{eq:tamecanonicity:1} and \eqref{eq:tamecanonicity:2} follows from the fact that the admissible validity of formulas without $\mu$ or $\nu$ on $\mathbf{A}^\delta$
will agree with ordinary admissible validity of formulas without $\mu$ or $\nu$. For formulas with $\mu$ or $\nu$, the
definition of $\mu^*$ and $\nu^*$ assures that the interpretation of such formulas will be identical on
$\mathbf{A}$ and $\mathbf{A}^\delta$.
For the equivalence of \eqref{eq:tamecanonicity:2} and \eqref{eq:tamecanonicity:3} we use
Corollary~\ref{cor:type1soundness}.
The equivalence of \eqref{eq:tamecanonicity:3} and \eqref{eq:tamecanonicity:4} follows from the fact that $\texttt{pure}(\varphi^* \le \psi^*)$ contains no propositional variables
and hence that its admissible validity and validity coincide. The final equivalence is the statement of Proposition~\ref{prop:ordinarysoundness}.
\end{proof}

The next two lemmas and the proposition which follows are needed to convert formulas of the form  $\mu^*X.\phi(X)$ back to $\mu X.\phi(X)$.
The first of the lemmas is proved in~\cite{muALBA}.

\begin{lemma}\label{lem:2.1orig}{\upshape \cite[Lemma 2.1]{muALBA}} Let $\mathbf{L}$ and $\mathbf{M}$ be complete lattices and
$G : M \times L \to L$. Let $\mu y.G(-,y) : \mathbf{M} \to \mathbf{L}$ be given for $a \in M$ by
$a \mapsto \bigwedge \{\, x \in L \mid G(a,x) \le x\,\}$.

If $G$ is completely join-preserving, then $\mu y.G(-,y) : \mathbf{M} \to \mathbf{L}$ is defined everywhere on $M$ and is
completely join-preserving.
\end{lemma}

\begin{lemma}\label{lem:droppingstar} Let $\A$ be a mu-algebra of the second kind and let $\tau$ be an order type over $n$. Let
$\psi(\overline{x},X) \in \Tm_1$ such that $\psi(\overline{x},X)$ is completely
$\bigvee$-preserving in $(\overline{x},X) \in (\A^\delta)^\tau \times \A^\delta$.
If $\A^\delta \models [(\nomi \le \mu^* X. \psi(\overline{\nomj_i}^\tau,X)
\amp \nomj^{\tau_i} \le^{\tau_i} \overline{\phi}_i)\Rightarrow \nomi \le \cnomm]$, then
$\A^\delta \models [\nomi \le \mu X. \psi(\overline{\phi},X) \Rightarrow \nomi \le \cnomm]$.
\end{lemma}
\begin{proof}
Assume that $\A^\delta \models [(\nomi \le \mu^* X. \psi(\overline{\nomj_i}^\tau,X) \amp \nomj^{\tau_i} \le^{\tau_i} \overline{\phi}_i)\Rightarrow \nomi \le \cnomm]$ and suppose $V(\nomi) \le V(\mu X.\psi(\overline{\phi},X))$. %That is, $V(i) \le \bigwedge \{\, a \in A^\delta \mid \psi(\varphi,a) \le a \,\}$.
By the join-density of $\jty(\A^\delta)$ we have that
$\overline{\varphi}=\bigvee\{\, \overline{j} \mid \overline{j} \in \jty((\A^\delta)^\tau), \overline{j} \le \overline{\phi}\,\}$. Thus
$\mu X.\psi(\overline{\phi},X)=\mu X. \psi( \bigvee\{\,\overline{j}\mid \overline{j}\le \overline{\phi}\,\},X)$.
Now by Lemma~\ref{lem:2.1orig} with $M=(\A^\delta)^\tau$ and $L=\A^\delta$ we get
$\mu X.\psi(\overline{\phi},X)=\bigvee\{\, \mu X.\psi(\overline{j},X) \mid \overline{j} \le \overline{\phi}\,\}$.
Since $V(\nomi) \in \jty(\A^\delta)$, it is completely join-prime and so
$V(\nomi) \le \mu X.\psi(\overline{\phi},X)=\bigvee\{\, \mu X.\psi(\overline{j},X) \mid \overline{j} \le \overline{\phi}\,\}$
implies that there exists some $\overline{j_0}$ with $\overline{j_0} \le \overline{\phi}$ such that
$V(\nomi)  \le \mu X.\psi(\overline{j_0},X) =\bigwedge \{\,a \in A^\delta \mid \psi(\overline{j_0},a)\le a \,\}
\le \bigwedge \{\, a \in A \mid \psi(\overline{j_0},a)\,\}= \mu^* X.\psi(\overline{j_0},X)$.
Thus by the assumption, $V(\nomi) \le V(\cnomm)$.
\end{proof}

\begin{prop}\label{prop:droppingstar}
Let $\A$ be a mu-algebra of the second kind and $\phi \le \psi$ an $\Tm_1$-inequality. Suppose
that a proper run of $\mu^*$-ALBA succeeds on $\phi \le\psi$, producing {\upshape\texttt{pure}}$(\phi^* \le \psi^*)$.
If $\A^\delta \models$ {\upshape\texttt{pure}}$(\phi^*\le \psi^*)$ then $\A^\delta \models \phi \le \psi$.
\end{prop}
\begin{proof}
The set of quasi-inequalities $\texttt{pure}(\phi^* \le \psi^*)$ is obtained through a finite number of rule applications of $\mu^*$-ALBA.
As detailed by Proposition~\ref{prop:ordinarysoundness}, all of the residuation, adjunction, ordinary approximation and Ackermann rules can be reversed while preserving validity on $\A^\delta$. Suppose that at some stage during the proper run of $\mu^*$-ALBA,
the inequality $\nomi \le \mu^* X.\psi(\overline{\phi}/\overline{x},X)$ is converted into
$\bigiamp_{i=1}^{n}( \exists \nomj^{\tau_i}[\nomi \leq \mu^* X.\psi( {\overline{\nomj}_i}^{\tau}/\overline{x}, X)\ \&\ \nomj^{\tau_i} \leq^{\tau_i} \phi_i])$. By Lemma~\ref{lem:droppingstar}, the disjunction $\bigiamp_{i=1}^{n}( \exists \nomj^{\tau_i}[\nomi \leq \mu^* X.\psi( {\overline{\nomj}_i}^{\tau}/\overline{x}, X)\ \&\ \nomj^{\tau_i} \leq^{\tau_i} \phi_i])$ can be converted into
$\nomi \le \mu X.\psi(\overline{\phi},X)$.
\end{proof}

We are now ready to present the final canonicity result.

\begin{theorem}{\upshape\textbf{(Canonicity)}} \label{thm:truecanonicity}
Let $\A$ be a mu-algebra of the second kind and let $\phi \le \psi$ be an $\Tm_1$-inequality
on which a proper run of $\mu^*$-ALBA succeeds.
If $\A \models \phi \le \psi$ then $\A^\delta \models \phi \le \psi$.
\end{theorem}
\begin{proof} The proof is similar to that of Theorem~\ref{thm:tamecanonicity}.
\begin{eqnarray}
&&\A \models \phi \leq \psi\label{eq:truecanonicity:1}\\
&\Longleftrightarrow  &\A^\delta \models_{\A} \phi^{*} \leq \psi^{*}\label{eq:truecanonicity:2}\\
&\Longleftrightarrow  &\A^{\delta} \models_{\A} \texttt{pure}(\phi^{*} \leq \psi^{*})\label{eq:truecanonicity:3}\\
&\Longleftrightarrow  &\A^{\delta} \models \texttt{pure}(\phi^{*} \leq \psi^{*})\label{eq:truecanonicity:4}\\
&\Longrightarrow  &\A^{\delta} \models \phi\leq \psi\label{eq:truecanonicity:5}
\end{eqnarray}
The equivalence of \eqref{eq:truecanonicity:1} and \eqref{eq:truecanonicity:2} follows as for Theorem~\ref{thm:tamecanonicity}.
For the equivalence of \eqref{eq:truecanonicity:2} and \eqref{eq:truecanonicity:3} we  apply the inductive argument from
Proposition~\ref{prop:admissiblesoundness}. 	
The equivalence of \eqref{eq:truecanonicity:3} and \eqref{eq:truecanonicity:4} follows from the fact that $\texttt{pure}(\varphi^* \le \psi^*)$ contains no propositional variables
and hence admissible validity and validity coincide for it.
The final implication is the statement of Proposition~\ref{prop:droppingstar}.
\end{proof}

%\marginpar{\raggedright \tiny AC: should we make this a numbered remark?}
Lastly, we make an observation regarding mu-inequalities with no occurrences of fixed point binders. If $\phi \le \psi$ does not
contain any fixed point binders, and a tame run of $\mu^*$-ALBA succeeds on $\phi \le \psi$, then $\phi \le \psi$ will be tame
canonical. That is, $\A \models \phi \le \psi$ if and only if $\A^\delta \models \phi^* \le \psi^*$.
However, we have that $\phi = \phi^*$ and $\psi=\psi^*$ and so $\A \models \phi \le \psi$ if and only if $\A^\delta \models \phi \le \psi$.
Likewise, if a proper run of $\mu^*$-ALBA succeeds on $\alpha \le \beta$ where this mu-inequality has no occurrences of fixed point binders,
Proposition~\ref{prop:ordinarysoundness} can be applied to show the converse of (21) $\Rightarrow$ (22) in
Theorem~\ref{thm:truecanonicity} and hence $\A \models \alpha \le \beta$ if and only if $\A^\delta \models \alpha \le \beta$.

\section{Canonicity of the restricted and tame inductive mu-inequalities}\label{sec:canon-of-synclasses}

In this section we argue that $\mu^*$-ALBA successfully purifies all restricted inductive mu-inequalities by means of proper runs, and all tame inductive mu-inequalities by means of tame runs. Once we have established these claims, the next theorem will then follow from Theorems \ref{thm:tamecanonicity} and \ref{thm:truecanonicity}:

\begin{theorem}\label{thm:rstrctd:tm:ind:canonical}
All restricted inductive mu-inequalities are canonical and all tame inductive mu-inequalities are tame canonical.
\end{theorem}

%habe to show that the weakenings imposed on $\mu$-ALBA to obtain $\mu^*$-ALBA together with the constraints on runs which make them proper and pure, correspond to the restrictions imposed on recursive %$\mu$-inequalities to make them firstly inductive and then, respectively, restricted and tame.

As far as tame inductive inequalities are concerned, the proof is almost verbatim the same as in \cite[Section 10.1]{DistMLALBA}. Indeed, tame inductive inequalities are just (a subset of the) inductive inequalities from the language of intuitionistic modal logic with some fixed point binders `along for the ride' on the non-critical branches. Thus the binders never need to be handled by the algorithm, so the runs are guaranteed to be tame and exactly the same strategy employed in \cite{DistMLALBA} is sufficient. The only point that needs to be checked is that the side conditions of the Ackermann rules pertaining to syntactic openness and closure are met when these rules need to be applied.  This follows from the next two lemmas:

\begin{lemma}\label{lem:syn:lmst:opn:clsd}
If $\mu^*$-ALBA is applied to any $\mathcal{L}_1$-inequality then, during the whole run, every inequality in every antecedent of every produced quasi-inequality is either pure, or has a syntactically almost closed left-hand side and a syntactically almost open right-hand side. Consequently, if a non-pure inequality contains no fixed point binders it has a syntactically closed left-hand side and a syntactically open right-hand side.
\end{lemma}
\begin{proof}
The proof is by induction on the application of the rules of $\mu^*$-ALBA. Suppose $\mu^*$-ALBA is run on an $\mathcal{L}_1$-inequality. As discussed above, preprocessing turns this into a finite number of inequalities $\eta \leq \beta$. Starring  and first approximation produce $(\nomi \leq \eta^*\ \&\ \beta^* \leq \cnomm)\Rightarrow \nomi\leq \cnomm$. We claim that $\nomi\leq \eta^*$ and $\beta^*\leq \cnomm$ satisfy the statement of the lemma. Indeed, $\nomi$ and $\cnomm$ are, respectively, syntactically closed and open. As $\eta^*, \beta^* \in \mathcal{L}_*$, they do not contain any nominals, co-nominals, $\Diamondblack$ or $\blacksquare$, and are therefore both syntactically almost open and syntactically almost closed.

The induction now proceeds by showing that the desired properties are invariant under the application of the rules of $\mu^*$-ALBA. The most interesting cases are those for ($\mu^{\tau}$-A-R) and ($\nu^{\tau}$-A-R). We verify ($\mu^{\tau}$-A-R). Applied to an inequality $\nomi \leq \mu^* X.\psi(\ophi / \overline{x}, X, \oga/\oz)$, the rule ($\mu^{\tau}$-A-R) produces a disjunction of inequalities of the form
$\nomi \leq \mu^* X.\psi( {\overline{\nomj}_i}^{\tau}/\overline{x}, X, \oga/\oz)$ and $\nomj^{\tau_i} \leq^{\tau_i} \phi_i$. Since the $\oga$ are constant sentences, the first of these inequalities is pure. If $\tau_i = 1$, the second inequality is $\nomj \leq \phi_i$ where $\nomj$ is syntactically closed and $\phi_i$ is syntactically almost open since it occurs positively as a subformula of the syntactically almost open formula  $\mu^* X.\psi(\ophi / \overline{x}, X, \oga/\oz)$. On the other hand, if $\tau_i = \partial$, the second inequality is $ \phi_i \leq \cnomn$ where $\cnomn$ is syntactically open and $\phi_i$ is syntactically almost closed since it occurs negatively as a subformula of the syntactically almost open formula  $\mu^* X.\psi(\ophi / \overline{x}, X, \oga/\oz)$.
\end{proof}

\begin{lemma}
If $\mu^*$-ALBA is applied to any tame inductive $\mathcal{L}_1$ inequality, then, during the whole run, every inequality in every antecedent of every quasi-inequality is either pure or has a left-hand side (resp., right-hand side) in which all occurrences of $\mu^*$ are positive (resp., negative) and all occurrences of $\nu^*$ are negative (resp., positive).
\end{lemma}

\begin{proof}
Preprocessing turns any tame inductive $\mathcal{L}_1$ inequality into a finite number of tame inductive inequalities $\eta \leq \beta$. By the definition of tame inductive inequalities, the only occurrences of binder nodes in $+\eta$ and $-\beta$ are $+\nu$ and $-\mu$. Thus starring and first approximation yields a quasi-inequality
\[
(\nomi\leq \eta^* \ \&\ \beta^* \leq \cnomm)\Rightarrow \nomi\leq \cnomm,
\]
where in $\eta^*$ (resp., $\beta^*$) all occurrences of  all occurrences of $\mu^*$ are negative (resp., positive) and all occurrences $\nu^*$ are positive (resp., negative). Thus this quasi-inequality satisfies the claim of the lemma. It is now sufficient to show that these conditions are invariant under the application of all $\mu^*$-ALBA rules. For the sake of the adjunction rules, note that all occurrences of binders are untouched and left on the same sides of inequalities.  The residuation rules move subformulas across the inequality but with an accompanying change in polarity, and the same consideration also deals with the Ackermann rules. The fixed point approximation rules ($\mu^{\tau}$-A-R) and ($\nu^{\tau}$-A-R)
%\marginpar{\raggedright \tiny W: Might need to adjust name of order type.}
are not applied, since tame inductive inequalities contain no binders on critical branches. The cases for the other approximation rules are easy to verify --- let us look explicitly only at ($\rightarrow$Appr). By assumption, in the premise $\chi \rightarrow \phi \leq \cnomm$, all occurrences of $\mu^*$ are positive and all occurrences of $\nu^*$ are negative. Thus, in $\nomj \rightarrow \cnomn \leq \cnomm \; \amp \; \nomj \leq \chi \; \amp \; \phi \leq \cnomm$, the conclusion of the rule, the inequality $\nomj \rightarrow \cnomn$  is pure, the righthand side of $\nomj \leq \chi$ contains  $\mu^*$ only negatively and $\nu^*$  only positively, while the lefthand side of $\phi \leq \cnomm$  contains  $\mu^*$ only positively and $\nu^*$  only negatively.
\end{proof}

Now that we have established Theorem \ref{thm:rstrctd:tm:ind:canonical} for tame inductive mu-inequalities, we will focus on the restricted inductive mu-inequalities for the remainder of this section. In \cite[Section 9]{muALBA} it was proven that $\mu$-ALBA successfully purifies all recursive mu-inequalities. The restricted inductive mu-inequalities form a subclass of these, but $\mu^*$-ALBA is also a restricted version of $\mu$-ALBA. Consequently our argument will follow that in \cite{muALBA} very closely, but we will be at pains to show how the constraints built into the restricted inductive mu-inequalities allow us to still succeed with the restricted resources of $\mu^*$-ALBA. The reader might find it useful to refer to Example \ref{Examp:ALBA:Restr:Ind} while reading the following argument.

%\subsection{Preprocess, first approximation and approximation}
Let $\eta \leq \beta$ be an $(\Omega, \epsilon)$-inductive inequality. We proceed as in ALBA and preprocess this inequality by applying splitting and ($\top$) and ($\bot$) exhaustively. This might produce multiple inequalities, on each of which we proceed separately. On each such inequality, denoted again $\eta \leq \beta$,  we proceed to first approximation, which yields the following quasi-inequality:
\begin{equation}\label{First:Approx:Eq}
\forall \op\forall \nomi\forall \cnomm[(\nomi\leq \eta\ \&\ \beta\leq \cnomm)\Rightarrow \nomi\leq \cnomm].
\end{equation}
Because its consequent is always pure, we only concentrate on its antecedent. Since the outer skeletons of $\beta$ and $\eta$ are built exactly as the outer part of an inductive modal formula, the ordinary approximation rules can be applied so as to surface the inner skeleton.  So we can equivalently rewrite $\nomi\leq \eta\ \&\ \beta\leq \cnomm$ as the conjunction of a set of inequalities which, whenever they contain critical variables in the scope of fixed point binders occurring as skeleton nodes, are of the form
\begin{equation}
\nomi \leq \mu X.\psi'(\overline{p}) \quad \text{ \ and  \ } \quad \nu X.\phi'(\overline{p})\leq \cnomm,\label{Eq:Approx}
\end{equation}
where $\mu X.\psi'(\overline{p})$ and $\nu X.\phi'(\overline{p})$ are sentences. (For the critical branches which do not contain such fixed point binders, we further proceed by exhaustively applying the approximation rules as in ALBA in order to surface the PIA parts.)

In what follows, we call a generation tree \emph{non-trivially restricted $(\Omega, \epsilon)$-inductive} if it is restricted $(\Omega, \epsilon)$-inductive and contains at least one $\epsilon$-critical branch.

\begin{prop}\label{Approx:Rules:Applic:Prop}%${}$\\
\begin{enumerate}
\item The inequality $\nomi \leq \mu X.\psi'$ in (\ref{Eq:Approx}) is of the form $\nomi \leq \mu X.\psi(\ophi / \oy,X,  \oga /\oz)$, where $\mu X.\psi(\oy, X, \oz)$ is an
%ac1308
%$(\oy, \oz)$-\ifDia{\delta}
$(\oy, \oz)$-\ifDia{\tau}
formula for some order-type $\tau$ over $\oy$, and the $\oga$ are constant sentences;

\item the inequality $\nu X.\phi' \leq \cnomm$ in (\ref{Eq:Approx}) is of the form $\nu X.\phi(\opsi / \oy, X,  \oga / \oz)\leq \cnomm$, where $\nu X.\phi(\oy, X, \oz)$ is an
%ac1308
%$(\oy, \oz)$-\ifBox{\delta}
$(\oy, \oz)$-\ifBox{\tau}
formula for some order-type $\tau$ over $\oy$, and the $\oga$ are constant sentences.
\end{enumerate}
\end{prop}
\begin{proof}
Notice that preprocessing, first approximation and ordinary approximation rules do not involve fixed points. Hence a proof very similar to that of \cite[Lemma 10.6]{DistMLALBA} proves that $+ \mu X.\psi'$ and $- \nu X.\phi'$ are non-trivially restricted $(\Omega, \epsilon)$-inductive. Hence the statement immediately follows from Lemma \ref{Skeleton:Are:IFDia:Lemma} below.
\end{proof}

\begin{lemma}\label{Skeleton:Are:IFDia:Lemma}
\begin{enumerate}
\item Let $\psi'$ be such that $+\psi'$ is non-trivially restricted $(\Omega, \epsilon)$-inductive, and the $P_3$-paths of all critical branches are of length 0. Then $\psi'$ is of the form $\psi(\ophi / \oy, \oX, \oga / \oz)$ where $\psi(\ox, \oz)$ is an $(\ox, \oz)$-\ifDia{\tau} formula, for $\ox = \oy \oplus \oX$ and some order-type $\tau$ over $\ox$, the $\ophi$ are sentences and the $\oga$ are constant sentences.  Moreover, if $\oy = (y_1, \ldots, y_n)$ then, for each $1 \leq i \leq n$, $+\phi_i$ is $\epsilon$-PIA if $\tau_i = 1$ and $-\phi_i$ is $\epsilon$-PIA if $\tau_i = \partial$. Finally $\epsilon^{\partial}(+ \psi(\ox, \oga / \oz))$.

\item Let $\phi'$ be such that $-\phi'$ is non-trivially restricted $(\Omega, \epsilon)$-inductive, and the $P_3$-paths of all critical branches are of length 0. Then $\phi'$ is of the form $\phi(\opsi / \oy, \oX, \oga / \oz)$ where $\phi(\ox, \oz)$ is an $(\ox, \oz)$-\ifBox{\tau} formula, for $\ox = \oy \oplus \oX$ and some order-type $\tau$ over $\ox$, the $\ophi$ are sentences and the $\oga$ are constant sentences. Moreover, if $\oy = (y_1, \ldots, y_n)$ then, for each $1 \leq i \leq n$, $-\psi_i$ is $\epsilon$-PIA if $\tau_i = 1$ and $+\psi_i$ is $\epsilon$-PIA if $\tau_i = \partial$. Finally $\epsilon^{\partial}(- \phi(\ox, \oga / \oz))$.
\end{enumerate}
\end{lemma}

\begin{proof}
The proof is virtually identical to that of \cite[Lemma 9.2]{muALBA} and \cite[Remark 9.3]{muALBA}, the only difference being that we need to ensure that the $\oga$ are \emph{constant}, i.e., contain no propositional variables. We will do so by appealing to the condition (NL) at the appropriate point. The proof is by simultaneous induction on the {\em skeleton depths} of $+\psi'$ and $-\phi'$, where the skeleton depth of an $\epsilon$-recursive (or inductive) generation tree $\ast \xi$, with $\ast \in \{ +, -\}$, is the maximum length of the $P_2$ parts of $\epsilon$-critical branches in $\ast \xi$.

Consider the base case when the skeleton depth of $+\psi'$ is 0, i.e., the critical branches consist only of PIA nodes. Then, by (GB1), $\psi'$ is a sentence. Hence we let $\psi= x_1$ which is \ifDia{\tau} with $\tau = (1)$, $\phi_1 = \psi'$ (which is $\epsilon$-PIA and a sentence). Moreover, (vacuously) all $\oga$ are constant sentences.

We assume that the claim holds for all $+ \psi'$ and $-\phi'$ of skeleton depth at most $k$. The inductive step consists of checking the possible cases (corresponding to the possible main connectives of $\psi'$ and $\phi'$). We verify the case when  $- \phi'$ is of the form $- (\psi' \rightarrow \chi)$. By (GB3) and (NL), $\chi$ is a constant sentence (and $\epsilon^{\partial}(- \chi)$). Then $+ \psi'$ is non-trivially restricted $(\Omega, \epsilon)$-inductive, and hence, by the induction hypothesis, $\psi'$ is of the form $\psi(\ophi / \oy, \oX, \oga / \oz)$, where $\psi(\ox, \oz)$ is an
%ac1308
$(\ox, \oz)$-\ifDia{\tau^{\partial}}
%$(\ox, \oz)$-\ifDia{\delta^{\partial}}
formula for some order-type $\tau$ over $\ox = \oy \oplus \oX$, and the $\ophi$ are sentences and the $\oga$ are constant sentences. Moreover, $\oy = (y_1, \ldots, y_n)$ and for every $1 \leq i \leq n$, the generation tree $-\phi_i$ is non-trivially $(\Omega,\epsilon)$-PIA if $(\tau^{\partial})_i = 1$ (i.e., $\tau_i = \partial$) and $+\phi_i$ is non-trivially $(\Omega, \epsilon)$-PIA if $\tau^{\partial}_i = \partial$ (i.e., $\tau_i = 1$). Then we let $\phi = \psi(\ox, \oz) \rightarrow z$, which is
%ac1308
$(\ox, \oz \oplus z)$-\ifBox{\tau},
%$(\ox, \oz \oplus z)$-\ifBox{\delta},
where $z$ is a fresh variable. Moreover, for $1 \leq i \leq n$ we let $\psi_i = \psi_i$. Hence $\phi'$  is of the form $(\psi(\opsi / \oy, \oX, \oga / \oz) \rightarrow z)[\chi / z]$, with $\phi_1 \ldots \phi_n$ (obtained above from the induction hypothesis) playing the role of $\psi_1 \ldots \psi_n$. Finally, $\epsilon^{\partial}(\psi(\ox, \oga / \oz) \rightarrow \chi)$, since $\epsilon^{\partial}(- \chi)$, and the induction hypothesis implies that $\epsilon^{\partial}( + \psi(\ox, \oga / \oz))$.
\end{proof}

Proposition \ref{Approx:Rules:Applic:Prop} implies that the approximation rules ($\mu^{\tau}$-A-R) and ($\nu^{\tau}$-A-R) can be applied to the inequalities (\ref{Eq:Approx}), respectively.\footnote{Applying one of these approximation rules within the antecedent of a quasi-inequality may split that quasi-inequality into the conjunction of several quasi-inequalities, on each of which we proceed separately.} In addition to this, we can assume w.l.o.g.\ that every inequality sitting in an antecedents of any quasi-inequality produced by these rule applications and containing a critical branch is of the form
\begin{equation}\label{Epsilon:PIA:Displayed:Eq}
\nomj \leq \phi \quad \text{ or  } \quad \psi \leq \cnomn,
\end{equation}
where $+\phi$ and $-\psi$ are non-trivially $(\Omega, \epsilon)$-PIA (i.e., non-trivially $(\Omega, \epsilon)$-inductive with all critical branches consisting only of PIA-nodes, i.e., of $P_1$-nodes) and, by (NB-PIA) and (GB1) $\phi$ and $\psi$ are sentences containing no binders, i.e., they are formulas of $\mathcal{L}$. In other words $\phi$ and $\psi$ are $(\Omega, \epsilon)$-inductive formulas of intuitionistic modal logic. This means that we can proceed as in ALBA to eliminate all propositional variables from the quasi-inequality. However, before the Ackermann rules can be applied we need to be sure that their side conditions are met. This is guaranteed by the following considerations:
%\marginpar{\raggedright \tiny Check order type here: should be one of tau or delta.}
after the application of the approximation rules  ($\mu^{\tau}$-A-R) and ($\nu^{\tau}$-A-R), all fixed point binders in the current quasi-inequalities occur only in pure inequalities of the form  $\nomi \leq \mu^* X.\psi( {\overline{\nomj}_i}^{\epsilon}/\overline{x}, X, \oga/\oz)$ and $\nu^* X.\varphi({\overline{\cnomn}_i}^{\epsilon}/\overline{x}, X, \oga/\oz) \leq \cnomm$. The only inequalities involved in the application of an Ackermann rule are non-pure, and therefore cannot contain any fixed point binders. So by Lemma \ref{lem:syn:lmst:opn:clsd} above, their left hand sides will be syntactically closed and their right and sides syntactically open, as desired.

\paragraph{In summary:} In dealing with the outer skeleton we proceed as in ALBA. The inner skeleton is processed as in $\mu$-ALBA, but with the applicability of the restricted approximation rules ($\mu^{\tau}$-A-R) and ($\nu^{\tau}$-A-R) guaranteed by Proposition \ref{Approx:Rules:Applic:Prop} via Lemma \ref{Skeleton:Are:IFDia:Lemma}. This proposition and lemma are provable thanks to the condition (NL). The condition (NB-PIA) ensures that the PIA parts contain no fixed point binders or fixed point variables and can hence be treated as in ALBA, without the need of the adjunction rules for $\mu$ and $\nu$ used in $\mu$-ALBA. Next, the condition ($\Omega$-CONF) guarantees that we are working with inductive rather than recursive inequalities, and hence that the ordinary (non-recursive) Ackermann rule can be used to eliminate the propositional variables. Lastly, the requirement that every occurrence of a binder is on an $\epsilon$-critical branch (Definition \ref{def:rstrctd:tm:indctv}) guarantees that every binder must have one of the approximation rules ($\mu^{\tau}$-A-R) and ($\nu^{\tau}$-A-R) applied to it, i.e., that the run of the algorithm is proper.

\section{Examples}\label{sec:examples}

%\texttt{[THIS SECTION MIGHT FALL AWAY OR BE REPLACED WITH SOME COMMENTS ABOUT HOW THE EARLIER EXAMPLES ON WHICH $\mu^*$-ALBA SUCCEEDED ARE
%INCLUDED IN THE CLASS WE HAVE DESCRIBED]}\\

%Here we will include some examples of inequalities on which on our algorithm $\mu^*$-ALBA succeeds. That is, the algorithm is successful in
%removing all propositional variables from the inequality. This provides the equivalence of admissible valuations on the canonical extension
%and valuations which are not necessarily admissible.

\begin{example}\label{Examp:ALBA:Restr:Ind}
Let us consider the restricted inductive inequality $\Diamond \mu X. (\Diamond X \vee \Box (\Box \Diamond q \vee p )) \leq \nu Y. ([ \Box((q \rightarrow \bot) \wedge (p \rightarrow \bot)) \rightarrow \bot] \wedge \Box Y)$ from Example \ref{Examp:Restricted:Ind}. No preprocessing is possible, so starring the binders and applying the first approximation rule produces:
\begin{align}
\forall \nomi \forall \cnomm[\nomi \leq  \Diamond \mu^* X. (\Diamond X \vee \Box (\Box \Diamond q \vee p )) \amp \nu^* Y. ([ \Box((q \rightarrow \bot) \wedge (p \rightarrow \bot)) \rightarrow \bot] \wedge \Box Y) \leq \cnomm \Rightarrow \nomi \leq \cnomm]\label{Examp:ALBA:Restr:Ind:EQ:1}
\end{align}
We begin by applying approximation rules to surface the PIA parts. The ($\Diamond$Appr) rule transforms the first inequality in the antecedent into $\nomi \leq  \Diamond \nomj$ and  $\nomj \leq \mu^* X. (\Diamond X \vee \Box (\Box \Diamond q \vee p ))$. We have to apply ($\mu^{\tau}$-A-R) to the latter inequality. To that end, consider the subformula $(\Diamond X \vee \Box (\Box \Diamond q \vee p ))$ --- it is of the form $\psi(X/x_1, \Box (\Box \Diamond q \vee p )/x_2)$ where $\psi(x_1, x_2) = \Diamond x_1 \vee x_2$, which is completely $\bigvee$-preserving as a map from $\mathbf{C} \times \mathbf{C}$ to $\mathbf{C}$ (indeed, $\psi(x_1, x_2)$ is an $(x_1, x_2)$-\ifDia{(1,1)} formula). So we may apply ($\mu^{\tau}$-A-R) to produce $\nomj \leq \mu^* X. (\Diamond X \vee \nomk)$ and $\nomk \leq \Box (\Box \Diamond q \vee p )$.

Next we have to apply ($\nu^{\tau}$-A-R) to $\nu^* Y. ([ \Box((q \rightarrow \bot) \wedge (p \rightarrow \bot)) \rightarrow \bot] \wedge \Box Y) \leq \cnomm$. The subformula $[ \Box((q \rightarrow \bot) \wedge (p \rightarrow \bot)) \rightarrow \bot] \wedge \Box Y$ is of the form $\phi(\Box((q \rightarrow \bot) \wedge (p \rightarrow \bot))/x_1, Y/x_2)$ where $\phi(x_1, x_2) = (x_1 \rightarrow \bot) \wedge \Box x_2$. Now $\phi(x_1, x_2)$ is completely $\bigwedge$-preserving as a map from $\mathbf{C}^{\partial} \times \mathbf{C}$ to $\mathbf{C}$ --- one way to see this is to note that it is an $(x_1, x_2)$-\ifBox{(\partial,1)} formula. Applying ($\nu^{\tau}$-A-R) yields  $\nu^* Y. ([ \noml \rightarrow \bot] \wedge \Box Y) \leq \cnomm$ and $\noml \leq \Box((q \rightarrow \bot) \wedge (p \rightarrow \bot))$.

Thus the PIA parts have been surfaced and the quasi-inequality has been transformed into:

\begin{align*}
%\forall \nomi_0 \cnomm_0 \nomj \nomk \noml
\left[ \bigamp \left(%
\begin{array}{l}
\nomi \leq  \Diamond \nomj\\
\nomj \leq \mu^* X. (\Diamond X \vee \nomk) \qquad\qquad\:\:\: %\\
\nomk \leq \Box (\Box \Diamond q \vee p )\\
\nu^* Y. ([ \noml \rightarrow \bot] \wedge \Box Y) \leq \cnomm \qquad %\\
\noml \leq \Box((q \rightarrow \bot) \wedge (p \rightarrow \bot))
\end{array}
\right)\Rightarrow \nomi \leq \cnomm \right].%\label{Examp:ALBA:Restr:Ind:EQ:2}
\end{align*}

We now proceed to apply adjunction and residuation rules to get the quasi-inequality into the right shape for the application of the Ackermann rules. Applying ($\Box$RA), ($\vee$RR) and ($\wedge$RA) produces

\begin{align*}
%\forall \nomi_0 \cnomm_0 \nomj \nomk \noml
\left[ \bigamp\left(%
\begin{array}{l}
\nomi \leq  \Diamond \nomj\\
\nomj \leq \mu^* X. (\Diamond X \vee \nomk) \qquad \qquad\:\:\: %\\
\Diamondblack \nomk - \Box \Diamond q  \leq  p\\
\nu^* Y. ([ \noml \rightarrow \bot] \wedge \Box Y) \leq \cnomm \qquad%\\
\Diamondblack \noml \leq q \rightarrow \bot  \qquad%\\
\Diamondblack \noml \leq p \rightarrow \bot
\end{array}
\right)\Rightarrow \nomi \leq \cnomm \right].%\label{Examp:ALBA:Restr:Ind:EQ:3}
\end{align*}

Now applying ($\rightarrow$RA) and then ($\wedge$LR) to $\Diamondblack \noml \leq q \rightarrow \bot$ transforms it into $q \leq \Diamondblack \noml \rightarrow \bot$. This yields the following quasi-inequality, which has now been solved for $+p$ and $-q$ and is therefore ready for the application of
the right and lefthanded Ackermann rules (RA) and (LA) to eliminate $p$ and $q$, respectively:
\begin{align*}
%\forall \nomi_0 \cnomm_0 \nomj \nomk \noml
\left[\bigamp\left(%
\begin{array}{l}
\nomi \leq  \Diamond \nomj\\
\nomj \leq \mu^* X. (\Diamond X \vee \nomk) \qquad \qquad\:\:\:%\\
\Diamondblack \nomk - \Box \Diamond q  \leq  p\\
\nu^* Y. ([ \noml \rightarrow \bot] \wedge \Box Y) \leq \cnomm  \qquad%\\
q \leq \Diamondblack \noml \rightarrow \bot   \qquad%\\
\Diamondblack \noml \leq p \rightarrow \bot
\end{array}
\right)\Rightarrow \nomi \leq \cnomm \right].%\label{Examp:ALBA:Restr:Ind:EQ:4}
\end{align*}
Now applying (LA) gives
\begin{align*}
%\forall \nomi_0 \cnomm_0 \nomj \nomk \noml
\left[\bigamp\left(%
\begin{array}{l}
\nomi \leq  \Diamond \nomj\\
\nomj \leq \mu^* X. (\Diamond X \vee \nomk)\qquad \qquad\:\:\:%\\
\Diamondblack \nomk - \Box \Diamond (\Diamondblack \noml \rightarrow \bot)  \leq  p\\
\nu^* Y. ([ \noml \rightarrow \bot] \wedge \Box Y) \leq \cnomm  \qquad%\\
\Diamondblack \noml \leq p \rightarrow \bot
\end{array}
\right)\Rightarrow \nomi \leq \cnomm \right],%\label{Examp:ALBA:Restr:Ind:EQ:5}
\end{align*}
and then applying (RA) gives the pure quasi-inequality
\begin{align*}
%\forall \nomi_0 \cnomm_0 \nomj \nomk \noml
\left[\bigamp\left(%
\begin{array}{l}
\nomi \leq  \Diamond \nomj\\
\nomj \leq \mu^* X. (\Diamond X \vee \nomk)\\
\nu^* Y. ([ \noml \rightarrow \bot] \wedge \Box Y) \leq \cnomm\\
\Diamondblack \noml \leq (\Diamondblack \nomk - \Box \Diamond (\Diamondblack \noml \rightarrow \bot)) \rightarrow \bot
\end{array}
\right)\Rightarrow \nomi \leq \cnomm \right].%\label{Examp:ALBA:Restr:Ind:EQ:6}
\end{align*}

\end{example}

\begin{example}\label{Examp:ALBA:Tame:Ind}
Consider the tame inductive inequality $\Diamond(\Box \bot \vee p) \wedge \Box q \leq \mu Y. (\Diamond(p \wedge q) \wedge \Box Y)$ from
Example~\ref{Examp:Tame:Ind}. Preprocessing distributes $\Diamond$ and $\wedge$ over $\vee$ in the antecedent producing $(\Diamond\Box \bot \wedge \Box q) \vee (\Diamond p \wedge \Box q) \leq \mu Y. (\Diamond(p \wedge q) \wedge \Box Y)$, to which we can apply ($\vee$LA) to split into two inequalities: $(\Diamond\Box \bot \wedge \Box q)  \leq \mu Y. (\Diamond(p \wedge q) \wedge \Box Y)$ and $(\Diamond p \wedge \Box q) \leq \mu Y. (\Diamond(p \wedge q) \wedge \Box Y)$. As $p$ appears only positively in the first, we may eliminate it by applying the ($\bot$) rule to obtain $(\Diamond\Box \bot \wedge \Box q)  \leq \mu Y. (\Diamond(\bot \wedge q) \wedge \Box Y)$. The algorithm $\mu^*$-ALBA now proceeds separately on these two inequalities. We will describe this process only for the second one, the execution on the first being very similar but less interesting. Starring and first approximation transforms the second inequality into the quasi-inequality
%\begin{align*}
$$\nomi \leq \Diamond p \wedge \Box q \;\, \amp \;\,  \mu^* Y. (\Diamond(p \wedge q) \wedge \Box Y) \leq \cnomm \;\,\Rightarrow\;\, \nomi \leq \cnomm.$$
%\end{align*}

%\begin{align*}
%\nomi_0 \leq \Diamond p \wedge \Box q \; \amp \;  \mu Y. (\Diamond(p \wedge q) \wedge \Box Y) \leq \cnomm_0 \Rightarrow \nomi_0 \leq \cnomm_0.
%\end{align*}

Applying ($\wedge$RA) to $\nomi \leq \Diamond p \wedge \Box q$ produces $\nomi \leq \Diamond p$ and  $\nomi \leq \Box q$. To the first of these we apply ($\Diamond$Appr) to produce $\nomi \leq \Diamond \nomj$ and $\nomj \leq p$, while ($\Box$RA) turns the second into $\Diamondblack \nomi \leq q$. Thus we have the quasi-inequality
$$\nomi \leq \Diamond \nomj \;\, \amp\, \; \nomj \leq p \; \amp \; \Diamondblack \nomi \leq q \;\, \amp\, \; \mu^* Y. (\Diamond(p \wedge q) \wedge \Box Y) \leq \cnomm \;\,\Rightarrow\;\, \nomi \leq \cnomm.$$

%\begin{align*}
%\nomi_0 \leq \Diamond \nomj \; \amp \; \nomj \leq p \; \amp \; \Diamondblack \nomi_0 \leq q \; \amp \; \mu Y. (\Diamond(p \wedge q) \wedge \Box Y) \leq \cnomm_0 \Rightarrow \nomi_0 \leq \cnomm_0.
%\end{align*}

This is now ready for the Ackermann rule (RA) to be applied twice to eliminate both $p$ and $q$, producing
$$\nomi \leq \Diamond \nomj \;\, \amp\, \; \mu^* Y. (\Diamond(\nomj \wedge \Diamondblack \nomi) \wedge \Box Y) \leq \cnomm \;\,\Rightarrow\,\; \nomi \leq \cnomm.$$
%\begin{align*}
%\nomi_0 \leq \Diamond \nomj \; \amp \; \mu Y. (\Diamond(\nomj \wedge \Diamondblack \nomi_0) \wedge \Box Y) \leq \cnomm_0 \Rightarrow \nomi_0 \leq \cnomm_0.
%\end{align*}
\end{example}

\appendix
\section*{Appendix: Algebraic properties of additional operations on perfect distributive lattices}\label{app:alg-lemmas}

Let $\C$ be a perfect distributive lattice.
The map $\kappa \colon J^\infty(\C) \to M^\infty(\C)$ is defined as
by $\kappa(x) = \bigvee (C \setminus {\uparrow}x)$.
This map is an order isomorphism between $J^\infty(\C)$ and $M^\infty(\C)$ with
$\kappa^{-1}$ defined for $m \in M^\infty(\C)$ by  $\kappa^{-1}(m) = \bigwedge (C \setminus {\downarrow}m)$.
Note that if $\C$ is an atomic Boolean algebra, then $\kappa$ is simply the negation $\neg$.

The results at the end of this section are algebraic versions of results presented in Appendix A of~\cite{DistMLALBA}.
We acknowledge that some of these results have been shown independently by Zhao~\cite{Zhao-thesis}.

\begin{lemma}\label{lem:xckappa} Let $\C$ be a completely distributive complete lattice. For any
$x \in J^\infty(\C)$, $m \in M^\infty(\C)$ and any $c \in C$
\begin{enumerate}
\item $x \nleqslant \kappa(x)$;
\item $\kappa^{-1}(m) \nleqslant m$;
\item $x \nleqslant c$ if and only if $c \le \kappa(x)$;
\item $c \nleqslant m$ if and only if $\kappa^{-1}(m) \le c$.
\end{enumerate}
\end{lemma}
\begin{proof} We prove (1) and (3). Suppose that $x \le \kappa(x)=\bigvee (C \setminus {\uparrow}x)$. Since
$x$ is completely join-prime we get that there exists $p \in (C \setminus {\uparrow}x)$ such
that $x \le p$, a contradiction. For (3), if $x \nleqslant c$, then $c \in (C \setminus {\uparrow}x)$
and by the definition of $\kappa$ we see that
$c \le \bigvee (C \setminus {\uparrow}x) = \kappa(x)$.
Now suppose that $c \le \kappa(x)$. If $x \le c$ then by transitivity we would get
$x \le \kappa(x)$, contradicting (1).
\end{proof}

We will often make use of the following approach in our proofs.

\begin{lemma}\label{lem:JMforleq} Let $p,q \in A^\delta$. 

\begin{enumerate}
\item If  $(x \nleq q \Rightarrow x \nleq p)$ for all $x \in \jty(\A^\delta)$, then $p \leq q$;
\item If $(p \nleq m \Rightarrow q \nleq m)$ for all $m \in \mty(\A^\delta)$, then $p \leq q$.
\end{enumerate}
\end{lemma}
\begin{proof} 1. We use the fact $\jty(\A^\delta)$ is join-dense in $\A^\delta$. We have
$p = \bigvee \{\,x \in \jty(\A^\delta) \mid x \le p\,\}$ and
$q = \bigvee \{\,x \in \jty(\A^\delta) \mid x \le q\,\}$. If
$x \nleq q \Rightarrow x \nleq p$ then $\{\, x \in \jty(\A^\delta) \mid x \le p \,\} \subseteq \{\,x \in \jty(\A^\delta) \mid x \le q \,\}$
and hence $p \le q$.
Part 2 follows using the meet-density of $\mty(\A^\delta)$ in $\A^\delta$.
\end{proof}

%\texttt{Do we have the order dual results? i.e. do we have $u\rightarrow k \in \Clos(\A^\sigma)$ and $u - k \in \Open(\A^\sigma)$? Do we need these at a later stage?}

%\newpage
Let $f$ be any map $f \colon \A \to \mathbf{B}$ and let
$c \in A^\delta$. The extensions $f^\sigma$ and $f^\pi$ of $f$ are defined by:
$$f^\sigma(c)=\bigvee \Big\{ \bigwedge\{\,f(a) \mid k \le a \le u \,\} \mid k \le c \le u, k \in \Clos(\A^\delta),u \in \Open(\A^\delta)\,\Big\}$$
and
$$f^\pi(c) = \bigwedge \Big\{ \bigvee\{\, f(a) \mid k \le a \le u \,\} \mid k \le c \le u, k\in \Clos(\A^\delta), u \in \Open(\A^\delta)\,\Big\}.$$

The following result holds for general lattices.

\begin{lemma}\label{lem:fsigmafpi}{\upshape\cite[Lemma 4.3]{GH2001}}
Let\, $\Lat$ and\, $\M$ be lattices, and let $f \colon \Lat \to \M$ be an order-preserving map.
\begin{enumerate}
\item $f^\sigma(k) = \bigwedge \{\,f(a) \mid a \in L \text{ and } k \le a\,\}$ for all $k \in \Clos(\Lat^\delta)$.
\item $f^\pi(u) = \bigvee \{\,f(a) \mid a \in L \text{ and } a \leq u \,\}$ for all $u \in \Open(\Lat^\delta)$.
%ac0412 unneeded statements taken out
%\item $f^\sigma(c) = \bigvee \{\,f^\sigma(k) \mid k \in \Clos(\Lat^\sigma) \text{ and } k \le c\,\} $ for all $c \in L^\sigma$.
%\item $f^\pi(c) =\bigwedge \{\,f^\pi(u) \mid u \in \Open(\Lat^\sigma) \text{ and } c \le u \,\} $ for all $c \in L^\sigma$.
%\item $f^\sigma$ and $f^\pi$ agree on\, $\Clos(\Lat^\sigma)\cup \Open(\Lat^\sigma)$.
\end{enumerate}
\end{lemma}

%ac0412 taken out, not needed
%The corresponding results for those above can be obtained for $g \colon \Lat \to \M$ an order-reversing map by
%using the fact that $(\Lat^\sigma)^\partial=(\Lat^\partial)^\sigma$. We state them here for convenience.

%\begin{lemma}\label{lem:gsigmagpi}
%Let\, $\Lat$ and\, $\M$ be lattices, and let $g \colon \Lat \to \M$ be an order-reversing map.
%\begin{enumerate}
%\item $g^\sigma(u) = \bigwedge \{\,g(a) \mid a \in L \text{ and } a \le u\,\}$ for all $u \in \Open(\Lat^\sigma)$.
%\item $g^\pi(k) = \bigvee \{\,g(a) \mid a \in L \text{ and } k \le a \,\}$ for all $k \in \Clos(\Lat^\sigma)$.
%\item $g^\sigma(c) = \bigvee \{\,g^\sigma(u) \mid u \in \Open(\Lat^\sigma) \text{ and } c \le u\,\} $ for all $c \in L^\sigma$.
%\item $g^\pi(c) =\bigwedge \{\,g^\pi(k) \mid k \in \Clos(\Lat^\sigma) \text{ and } k \le c \,\} $ for all $c \in L^\sigma$.
%\item $g^\sigma$ and $g^\pi$ agree on\, $\Clos(\Lat^\sigma)\cup \Open(\Lat^\sigma)$.
%\end{enumerate}
%\end{lemma}

%Lemma 5.3 from \cite{GH2001} will almost certainly be useful at a later stage.

From this point onwards we will be working with a distributive lattice $\A$ and its canonical extension $\A^\delta$.

\begin{lemma}\label{lem:forEsakia}
Let $f \colon \A \to \A$ be an order-preserving map and
%$g \colon \A \to \A$ an order-reversing map. Let
Let
$k \in \Clos(\A^\delta)$ and $u \in \Open(\A^\delta)$.
%Then for any $x \in \jty(\A^\sigma)$ and $m \in\mty(\A^\sigma)$
Then for any $c \in A^\delta$,
\begin{enumerate}
\item if $c \nleq f^\sigma(k)$, there exists $a \in A$ such that $k \leq a$ and $x \nleq f^\sigma(a)=f(a)$;
\item if $f^\pi(u) \nleq c$, there exists $a \in A$ such that $a \leq u$ and $f(a)=f^\pi(a) \nleq c$;
%\item if $g^\pi(k) \nleq c$, there exists $a \in A$ such that $k \leq a$ and $g(a)=g^\pi(a) \nleq c$;
%\item if $c \nleq g^\sigma(u)$, there exists $a \in A$ such that $a \leq u$ and $c \nleq g^\sigma(a)=g(a)$.
\end{enumerate}
\end{lemma}
\begin{proof} The proofs of (1) and (2) follow from parts (1) and (2) respectively of Lemma~\ref{lem:fsigmafpi}.
%while the proofs for (3) and (4) follow, respectively, from parts (2) and (1) of Lemma~\ref{lem:gsigmagpi}.
\end{proof}

%ac0607 removed because
%\begin{lemma}\label{lem:usualpreswhiteops} Let $M \subseteq A^\delta$. Then
%\begin{enumerate}
%\item $\Box (\bigwedge M) = \bigwedge \{\, \Box m \mid m \in M\,\}$;
%\item $\Diamond (\bigvee M) = \bigvee \{\, \Diamond m \mid m \in M\,\}$.
%%ac0711 removed as left and right diamonds no longer in signature
%%\item $\rhd (\bigvee M) = \bigwedge \{\, \rhd m \mid m \in M\,\}$;
%%\item $\lhd (\bigwedge M) = \bigvee \{\, \lhd m \mid m \in M\,\}$.
%\end{enumerate}
%\end{lemma}

The Lemma below is an algebraic version of Corollary A.4 from~\cite{DistMLALBA}.
\begin{lemma} \label{lem:ops-usual-preservation} Consider $\A^\delta$ and let $k \in \Clos(\A^\delta)$ and
$u \in \Open(\A^\delta)$. Then
\begin{enumerate}
\item $\Box k \in \Clos(\A^\delta)$;
\item $\Diamond u \in \Open(\A^\delta)$;
%%\item $\Bb k \in \Clos(\A^\sigma)$;
%%\item $\Db u \in \Open(\A^\sigma)$;
%ac0711
%\item $\rhd u \in \Clos(\A^\sigma)$;
%\item $\lhd k \in \Open(\A^\sigma)$;
%%\item $\btr u \in \Clos(\A^\sigma)$;
%%\item $\btl k \in \Open(\A^\sigma)$;
\end{enumerate}
\end{lemma}
\begin{proof}
1. Since $k$ is closed we have $\Box k= \Box \big(\bigwedge \{\, a \in A \mid k \le a\,\}\big)$.
As $\Box$ is completely meet-preserving we see that
$\Box k = \bigwedge \{\, \Box a \mid a \in A, k \le a \,\}$ which is a closed element of $\A^\delta$ since
$A$ is closed under $\Box$. Item (2) follows from the fact that $\Diamond$ is completely join-preserving.
\end{proof}

The following lemma, a modified version of~\cite[Lemma A.5]{DistMLALBA}, is an algebraic version of the Esakia lemma.

\begin{lemma} \label{lem:whiteEsakia} For every up-directed set $U$ of open elements, and every down-directed set
$D$ of closed elements:
\begin{enumerate}
\item $\Box (\bigvee U) = \bigvee \{\, \Box u \mid u \in U \,\}$;
\item $\Diamond (\bigwedge D) = \bigwedge \{\, \Diamond d \mid d \in D \,\}$.
%ac0711
%\item $\rhd (\bigwedge D) = \bigvee \{\,\rhd d \mid d \in D \,\}$;
%\item $\lhd (\bigvee U)= \bigwedge \{\,\lhd u \mid u \in U\,\}$.
\end{enumerate}
\end{lemma}
\begin{proof}
1. The fact that $\Box (\bigvee U) \geq \bigvee \{\, \Box u \mid u \in U \,\}$ follows from the fact that
$\Box$ is order-preserving. Now let $m \in \mty(\A^\delta)$ such that $\Box (\bigvee U) \nleq m$.
By Lemma~\ref{lem:forEsakia}(2), there exists $a \in A$ such that $a \le \bigvee U$ and $\Box a \nleq m$. By compactness,
there exists a finite set $F \subseteq U$ such that $a \le \bigvee F$. Since $U$ is up-directed, there exists $v \in U$ such that
$\bigvee F \le v$ and hence $a \le v$. By the fact that $\Box$ is order-preserving, we get $\Box a \leq \Box v$ and hence
$\Box v \nleq m$. Thus $\bigvee \{\,\Box u \mid u \in U\,\} \nleq m$ and by Lemma~\ref{lem:JMforleq}(2)
we have that $\Box(\bigvee U) \leq \bigvee\{\,\Box u \mid u \in U\,\}$.

2. The inequality $\Diamond (\bigwedge D) \le \bigwedge\{\,\Diamond d \mid d \in D\,\}$ follows from the
fact that $\Diamond$ is order-preserving. Now suppose that $x \in \jty(\A^\delta)$ such that
$x \nleq \Diamond (\bigwedge D)$. Since $\bigwedge D$ is closed, we can use Lemma~\ref{lem:forEsakia}(1) to get
$a \in A$ such that $\bigwedge D \le a$ and $x \nleq \Diamond(a)$. Again using the fact that
$\bigwedge D$ is closed, we use the compactness of $\A^\delta$ to get a finite subset $F \subseteq D$ such
that $\bigwedge F \le a$. Since $D$ is down-directed, the set $F$ has a lower bound $e \in D$. Now
$x \nleq \Diamond(e)$ and hence $x \nleq \bigwedge \{\,\Diamond d \mid d \in D\,\}$. Using Lemma~\ref{lem:JMforleq}(1)
now gives us the required inequality.
%ac0711
%3. The fact that $\bigvee\{\,\rhd d \mid d \in D\,\} \le \rhd(\bigwedge D)$ follows from $\rhd$ being order reversing.
%Now suppose that $\rhd(\bigwedge D) \nleq m$. By Lemma~\ref{lem:forEsakia}(3) there exists $a \in A$ such that
%$\bigwedge D \leq a$ and $\rhd a \nleq m$. Using compactness we obtain a finite set $F \subseteq D$ such that
%$\bigwedge F \le a$. Since $D$ is down-directed, there exists $e \in D$ such that $e \leq \bigwedge F \leq a$.
%Since $\rhd$ is order reversing, we get
%that $\rhd a \leq \rhd e$ and hence $\rhd e \nleq m$. Thus $\bigvee \{\,\rhd d \mid d \in D\,\} \nleq m$ and by
%Lemma~\ref{lem:JMforleq}(2) we see that $\rhd (\bigwedge D) \leq \bigvee \{\,\rhd d \mid d \in D \,\}$.
%
%4. Since $\lhd$ is order reversing we see that $\lhd(\bigvee U) \le \bigwedge\{\,\lhd u \mid u \in U \,\}$.
%Suppose that $x \nleq \lhd(\bigvee U)$. By Lemma~\ref{lem:forEsakia}(4) there exists $a \in A$ such that $a \leq \bigvee U$ and
%$x \nleq \lhd a$. By compactness there exists a finite set $F \subseteq U$ with $a \leq \bigvee F$ and since $U$ is up-directed there
%exists $v \in U$ such that $a \leq \bigvee F \leq v$. Since $\lhd$ is order reversing we see that $\lhd v \leq \lhd a$ and thus
%$x \nleq \lhd v$. This implies that $x \nleq \bigwedge\{\, \lhd u \mid u \in U\,\}$ and so by Lemma~\ref{lem:JMforleq}(1)
%$\bigwedge\{\,\lhd u \mid u \in U\,\}\leq \lhd(\bigvee U)$.
\end{proof}

\begin{cor}\label{cor:whiteopspreservation}
Let $u \in \Open(\A^\delta)$ and $k \in \Clos(\A^\delta)$. Then
\begin{enumerate}
\item $\Box u \in \Open(\A^\delta)$;
\item $\Diamond k \in \Clos(\A^\delta)$.
%ac0711
%\item $\rhd k \in \Open(\A^\sigma)$;
%\item $\lhd u \in \Clos(\A^\sigma)$.
\end{enumerate}
\end{cor}
\begin{proof} We show (1). The set $\{\,a \in A \mid a \leq u\,\}$ is up-directed
and hence %by
Lemma~\ref{lem:whiteEsakia}(1)
gives us the equality
%we have
$\Box u = \Box \Big( \bigvee \{\, a \in A \mid a \leq u\,\} \Big) = \bigvee \{\,\Box a \mid a \in A, a \leq u\,\}$.
%$$\Box u = \Box \Big( \bigvee \{\, a \in A \mid a \leq u\,\} \Big) = \bigvee \{\,\Box a \mid a \in A, a \leq u\,\}.$$
Thus $\Box u$ is open. Part (2) follows similarly using the corresponding statement from
Lemma~\ref{lem:whiteEsakia}.
\end{proof}

The result below is an Esakia-type lemma for the adjoint operations and the implications. Parts (1) and (2) are adaptations of Lemma A.7 from~\cite{DistMLALBA}.

\begin{lemma}\label{lem:extraopspreservation}
Let $U \subseteq \Open(\A^\delta)$ be an up-directed set and let $D \subseteq \Clos(\A^\delta)$ be a down-directed set.
Then
\begin{enumerate}
\item $\Bb (\bigvee U) = \bigvee\{\,\Bb u \mid u \in U\,\}$;
\item $\Db (\bigwedge D) = \bigwedge \{\,\Db d \mid d\in D\,\}$;
%ac added
\item $\bigwedge D - \bigvee U = \bigwedge \{\, d - u \mid d \in D, u \in U \,\}$;
\item $\bigwedge D \rightarrow \bigvee U = \bigvee \{\, d \rightarrow u \mid d \in D, u \in U\,\}$.
%ac0711
%\item $\btr (\bigwedge D) = \bigvee \{\,\btr d \mid d \in D\,\}$;
%\item $\btl (\bigvee U) = \bigwedge \{\,\btl u \mid u \in U\,\}$.
\end{enumerate}
\end{lemma}
\begin{proof} 1. The fact that $\bigvee \{\,\Bb u \mid u \in U\,\} \leq \Bb (\bigvee U)$ follows
from $\Bb$ being order-preserving. Now suppose that $\Bb(\bigvee U) \nleq m$. Thus $\kappa^{-1}(m) \leq \Bb(\bigvee U)$ and
by the adjunction, $\Diamond (\kappa^{-1}(m)) \leq \bigvee U$. By Corollary~\ref{cor:whiteopspreservation}(2) and the fact that
$\kappa^{-1}(m)$ is closed, we can apply compactness to get a finite subset $F \subseteq U$ such that
$\Diamond (\kappa^{-1}(m)) \leq \bigvee F$. Now since $U$ is up-directed, there exists $v \in U$ such that
$\Diamond (\kappa^{-1}(m)) \leq v$ and hence $\kappa^{-1}(m) \leq \Bb v$ and, furthermore, $\Bb v \nleq m$.
Thus $\bigvee \{\,\Bb u \mid u \in U\,\} \nleq m$ and by Lemma~\ref{lem:JMforleq}(2) we have that
$\Bb (\bigvee U) \leq \bigvee \{\,\Bb u \mid u \in U\,\}$.

2. Since $\Db$ is order-preserving, we have that $\Db(\bigwedge D) \le \bigwedge \{\,\Db d \mid d \in D\,\}$. Now suppose
that $x \nleq \Db(\bigwedge D)$ for $x \in \jty(\A^\delta)$. By Lemma~\ref{lem:xckappa}(3) we have that
$\Db(\bigwedge D) \le \kappa(x)$ and by the adjunction we have $\bigwedge D \le \Box (\kappa(x))$. By Corollary~\ref{cor:whiteopspreservation}(1)
and the fact that $\kappa(x)$ is open, we can apply compactness to get a finite subset $F \subseteq D$ such that
$\bigwedge F \le \Box (\kappa(x))$. Furthermore, since $D$ is down-directed, there exists $e \in D$ with
$e \le \bigwedge F$. Hence $e \le \Box(\kappa(x))$ and $\Db e \le \kappa(x)$ and so
$x \nleq \Db(e)$. Thus $x$ is not a lower bound for $\{\,\Db d \mid d \in D\,\}$ and so by Lemma~\ref{lem:JMforleq}(1) we
have the required inequality.
%ac0711
%3. As $\btr$ is order reversing, we have $\bigvee\{\,\btr d \mid d \in D\,\} \leq\, \btr\!(\bigwedge D)$.
%Now suppose that $\btr(\bigwedge D) \nleq m$ for $m \in\mty(\A^\sigma)$. Then $\kappa^{-1}(m) \le \btr(\bigwedge D)$ and hence
%$\bigwedge D \le \rhd (\kappa^{-1}(m))$. Noting that $\kappa^{-1}(m)$ is closed, we then apply Corollary~\ref{cor:whiteopspreservation}(3) to see
%that $\rhd (\kappa^{-1}(m))$ is open. Now we can apply compactness to get a finite subset $F \subseteq D$ such
%that $\bigwedge F \le \rhd(\kappa^{-1}(m))$. Since $D$ is down-directed we know there exists $e \in D$ such that
%$e \le \bigwedge F \le \rhd(\kappa^{-1}(m))$. Now $\kappa^{-1}(m) \le\, \btr\!e$ and so $\btr\! e \nleq m$. Thus $\bigvee \{\,\btr d \mid d \in D\,\} \nleq m$
%and so $\btr (\bigwedge D) \le \bigvee \{\, \btr d \mid d \in D\,\}$.
%
%4. Since $\btl$ is order reversing, we have that $\btl (\bigvee U) \le\bigwedge\{\,\btl u \mid u \in U \,\}$. Now suppose that
%$x \nleq \btl(\bigvee U)$ for $x \in \jty(\A^\sigma)$. Then ${\btl}(\bigvee U) \le \kappa(x)$ and so $\lhd \kappa(x) \le \bigvee U$.
%Since $\lhd(\kappa(x))$ is closed (Corollary~\ref{cor:whiteopspreservation}(4)) and $\bigvee U$ is open, there exists a finite subset $F \subseteq U$ such that
%$\lhd(\kappa(x)) \le \bigvee F \leq \bigvee U$. Furthermore, as $U$ is directed, there exists $v \in U$ which is an upper bound for $F$.
%Thus $\lhd(\kappa(x)) \le \bigvee F \le v \le \bigvee U$ and so $\btl v \le \kappa(x)$ and so $x \nleq \btl v$. Thus
%$x \nleq \bigwedge\{\, \btl u \mid u \in U\,\}$.

%ac added
3. For any $d \in D$ we have $\bigwedge D \le d$ and hence $\bigwedge D - \bigvee U \le d - \bigvee U$. For any $u \in U$ we have
$u \le \bigvee U$ and hence for arbitrary $d$ we have $d - \bigvee U \le d - u$. Hence $\bigwedge D - \bigvee U \le \bigwedge \{\, d - u \mid d \in D, u \in U \,\}$. For the opposite inequality, note that by density and the definition of $-$ we have
$$\bigwedge D - \bigvee U = \bigwedge \big\{\, w \in \Open(\A^\delta) \mid \TM D - \TJ U \le w\,\Big\} =\bigwedge \Big\{\,w \in \Open(\A^\delta) \mid \TM D \le \TJ U \vee w \,\big\}.$$
Now we see that $\bigvee U \vee w = \bigvee\{u \vee w \mid u \in U \,\}$. Since $w \in \Open(\A^\delta)$, the set $\{\,u \vee w \mid u \in U \,\}$ is a set of open elements. Thus by compactness, for every $w \in \Open(\A^\delta)$ such that $\bigwedge D \le \bigvee U \vee w$, there exists some
$d_w,u_w$ such that $d_w \le u_w \vee w$, or, equivalently, $d_w -u_w \le w$. This gives us
$\bigwedge \{\, w \in \Open(\A^\delta) \mid \exists\, d \in D, u \in U \text{ such that } d-u \le w \,\} \ge \bigwedge \{\, d-u \mid d \in D, u \in U\,\}$, and hence
$\bigwedge D - \bigvee U \ge \bigwedge \{\, d-u \mid d \in D, u \in U\,\}$.

4. For all $d \in D$ and $u \in U$ we have $\bigwedge D \le d$ and $u \le \bigvee U$ and hence $d \rightarrow u \le d \rightarrow \bigvee U \le \bigwedge D \rightarrow \bigvee U$. Thus $\bigvee \{\, d \rightarrow u \mid d \in D, u \in U \,\} \le \bigwedge D \rightarrow \bigvee U$. For the reverse inequality, first observe that by density
$$\bigwedge D \rightarrow \bigvee U = \bigvee \big\{\, k \in \Clos(\A^\delta) \mid k \le \TM D \rightarrow \TJ U \,\big\}= \bigvee\big\{\, k \in \Clos(\A^\delta) \mid k \wedge \TM D \le \TJ U\,\big\}.$$
Using compactness we
%ac1308
get
%see that
$\{\,k \in \Clos(\A^\delta) \mid k \wedge \bigwedge D \le \bigvee U \,\} = \{\,k \in \Clos(\A^\delta) \mid \exists\, d\in D, u \in U \text{ s.t. } k \wedge d \le u \,\}$. Thus
%ac1308
$\bigwedge D \rightarrow \bigvee U = \bigvee \{\, k \in \Clos(\A^\delta) \mid \exists d \in D, u \in U \text{ s.t. } k \le d \rightarrow u \,\}
\le  \bigvee \{\, d \rightarrow u \mid d \in D, u \in U \,\}$.
%$$\bigwedge D \rightarrow \bigvee U = \bigvee \{\, k \in \Clos(\A^\delta) \mid \exists d \in D, u \in U \text{ such that } k \le d \rightarrow u \,\}
%\le  \bigvee \{\, d \rightarrow u \mid d \in D, u \in U \,\}.$$
\end{proof}

%ac1408
The next lemma
%This
is an adaptation of Proposition A.8 from~\cite{DistMLALBA}.

\begin{lemma}\label{lem:openclosedextraops} Let $k \in \Clos(\A^\delta)$ and $u \in \Open(\A^\delta)$. Then
\begin{enumerate}
\item $\Bb u \in \Open(\A^\delta)$;
\item $\Db k \in \Clos(\A^\delta)$;
%\item $\btr k \in \Open(\A^\sigma)$;
%\item $\btl u \in \Clos(\A^\sigma)$.
\item $k \rightarrow u  \in \Open(\A^\delta)$;
\item $k - u \in \Clos(\A^\delta)$.
\end{enumerate}
\end{lemma}
\begin{proof}
1. It is clear that $\bigvee\{a \in A \mid a \le \Bb u\} \le \Bb u$. Suppose $x \in \jty(\A^\delta)$ and
that $x \le \Bb u$. By the adjunction this gives $\Diamond x \le u$. Since $x \in \Clos(\A^\delta)$ we
have that $\Diamond x = \Diamond \big( \bigwedge\{a \in A \mid x \le a\}\big) \le u$. Since
$\{a \in A \mid x \le a\}$ is down-directed, by Lemma~\ref{lem:whiteEsakia}(2) we
have that $\Diamond x = \bigwedge \{ \Diamond a \mid x \le a\} \le u$. Since $\Diamond x$ is closed, we can use
compactness to get a finite set $\{\Diamond a_i\}_i=1^n$ such that $\bigwedge\{\Diamond a_i \mid 1 \le i \le n\} \le u$.
Since $\Diamond$ is order-preserving, we have
$ \Diamond \Big( \bigwedge \{a_i \mid 1 \le i \le n\}\Big) \le \bigwedge\{\Diamond a_i \mid 1 \le i \le n\} \le u$.
Now $b=\bigwedge\{a_i \mid 1 \le i \le n\} \in A$ and $x \le b$. Thus we have $b \in A$ such that
$\Diamond b \le u$ (and hence $b \le \Bb u$) with $x \le b$. This gives us that
$\Bb u \le \bigvee\{a \in A \mid a \le \Bb u\}$.

2. Clearly $\Db k \le \bigwedge\{a \in A \mid \Db k \le a\}$. Suppose $m \in \mty(\A^\delta)$ and that
$\Db k \le m$. By the adjunction, we have $k \le \Box m$, which, by Lemma~\ref{lem:whiteEsakia}(1) gives us
$k \le \bigvee\{\Box a \mid a \in A \text{ and } a \le m\} = \Box m$.
By compactness we can get a finite subset $\{\Box a_i\}_{i=1}^n$ such that
$k \le \bigvee \{\Box a_i \mid 1 \le i \le n\}$ and since $\Box$ is order-preserving, we have that
$k \le \Box \bigvee \{a_i \mid 1 \le i \le n\}$. Clearly $b = \bigvee a_i \in A$ and $b \le m$. Thus
there exists $b \in A$ such that $k \le \Box b$, and hence $\Db k \le b$ with $b \le m$. This implies that
$\bigwedge\{a \in A \mid \Db k \le a\} \le \Db k$.
%ac0711
%3. Clearly $\bigvee\{\,a \in A \mid a \le \btr k\,\} \le \btr k$. Suppose that $\btr k \nleq m$ for $m \in \mty(\A^\sigma)$. Then
%$k \leq \rhd \kappa^{-1}(m)$. Since $\kappa^{-1}(m)$ is closed, and by Lemma~\ref{lem:whiteEsakia}(3) we have
%$$k \leq \rhd \big( \bigwedge\{\,b \in A \mid \kappa^{-1}(m) \le b\,\}\big) = \bigvee \{\,\rhd b \mid \kappa^{-1}(m) \le b \,\}.$$
%By compactness there exists a finite set $\{b_i\}_{i=1}^n$ such that $k \leq \bigvee \{\,\rhd b_i\mid 1 \le i \le n\,\}$. Let
%$c = \bigwedge_{i=1}^n b_i$. Since $\rhd$ is order reversing, we have that $k \leq \rhd c$ and hence $c \leq \btr k$. Lastly,
%$\kappa^{-1}(m) \le c$ and hence $c \nleq m$. The result follows by Lemma~\ref{lem:xkx}(2).
%
%4. Clearly $\btl u \leq \bigwedge\{\,a \in A \mid \btl u \leq a\,\}$. Suppose that $x \nleq \btl u$ where $x \in \jty(\A^\sigma)$. By
%Lemma~\ref{lem:xkx}(1) and the adjunction, we have $\lhd \kappa(x) \le u$. Since $\kappa(x)$ is open, we use Lemma~\ref{lem:whiteEsakia}(4)
%to see that
%$$
%\lhd \kappa(x) = \lhd \big( \bigvee \{\, b \in A \mid b \leq \kappa(x) \,\} = \bigwedge \{\, \lhd b \mid b \in A, b \leq \kappa(x)\,\} \le u.
%$$
%By compactness there exists a finite set $\{b_i\}_{i=1}^n$ such that $\bigwedge_{i=1}^n \lhd b_i \le u$. Let $c = \bigvee_{i=1}^n b_i$. By the fact that
%$\lhd$ is order reversing, $\lhd c \le \bigwedge_{i=1}^n \lhd b_i \le u$ and hence $\btl u \leq a$. Also, $c \le \kappa(x)$ and so $x \nleq c$ and thus the result
%follows by Lemma~\ref{lem:JMforleq}(1).

%ac joined into one lemma
3. By the join-density of $\Clos(\A^\delta)$
%the closed elements
%in $\A^\delta$
we get
%that
$k \!\rightarrow\! u=\!\bigvee\{\, c \in \Clos(\A^\delta) \mid c \le k \!\rightarrow\! u \,\} = \!\bigvee\{\, c\in \Clos(\A^\delta) \mid c \wedge k \le u \,\}$.%$$
Since every $a \in A$ is an element of $\Clos(\A^\delta)$ we have
$\bigvee \{\,a \in A \mid a \le k \rightarrow u \,\} = \bigvee\{\, a \in A \mid a \wedge k \le u \,\} \le k \rightarrow u$.
Now suppose that $c \in \Clos(\A^\delta)$ and that $c \wedge k \le u$.
Now $c = \bigwedge \{\, b_i \mid b_i \in A, c \le b_i\,\}$
and $c \wedge k = \bigwedge \{\, b_i \wedge k \mid c \le b_i, b_i \in A \,\} \in \Clos(\A^\delta)$. By compactness there exists a finite
subset of the $b_i$ such that $\bigwedge^n_{j=1} (b_j \wedge k) = (\bigwedge^n_{j=1} b_j) \wedge k \le u$.
Now $\bigwedge^n_{j=1} b_i = a_c \in A$ and so for every $c \in \Clos(\A^\delta)$ such that
$c \wedge k \le u$, there exists $a_c \in A$ such that $c \le a$ and $a \wedge k \le u$. Thus
$\bigvee\{\, a \in A \mid a \wedge k \le u \,\} = \bigvee \{\, a \in A \mid a \le k \rightarrow u \,\} = k \rightarrow u$
and so $k\rightarrow u \in \Open(\A^\delta)$.

4. Observe that $ k - u = \bigwedge \{\,w \in \Open(\A^\delta) \mid k -u \le w \,\} = \bigwedge\{\,w \in \Open(\A^\delta) \mid k \le w \vee u \,\}$.
Since each $a \in A$ is an element of $\Open(\A^\delta)$ we have that
$\bigwedge\{\,a \in A \mid k-u\le a \,\}=\bigwedge\{\,a \in A \mid k \le u \vee a\,\} \ge k-u$.
Let $w \in \Open(\A^\delta)$ such that $k \le u \vee w$. Now $w = \bigvee \{\, b_i \in A \mid b_i \le w\,\}$ and
$u \vee w$ is open. By compactness there exists a finite subset of the $b_i$ such that
$k \le \bigvee^m_{j=1}( u \vee b_j ) = (\bigvee^m_{j=1} b_j) \vee u$. Now $\bigvee^m_{j=1} b_j = a_w \in A$
and $k \le a_w \vee u$ and $a_w \le w$. Therefore
$\bigwedge\{\, a \in A \mid k \le u \vee a \,\} \le \bigwedge\{\, w \in \Open(\A^\delta) \mid k \le u \vee w\,\}=k-u$
and so $k - u \in \Clos(\A^\delta)$.
\end{proof}

\bibliographystyle{siam}
\bibliography{mucan}

\end{document}